\documentclass[leqno]{tac}
\usepackage{amsmath,amsfonts}

\author{George M. Bergman}
\thanks{%
Much of this work was done in 1986, when the author was
partially supported by NSF grant DMS~85-02330.
\protect\\\indent\indent 
arXiv:0711.0674.
}
\address{%
Department of Mathematics,\\
University of California\\
Berkeley, CA 94720-3840, USA}
\title{Colimits of representable algebra-valued functors}
\copyrightyear{2008}
\keywords{representable functor among varieties of algebras,
initial representable functor, colimit of representable functors,
final coalgebra, limit of coalgebras;
binar (set with one binary operation), semigroup, monoid, group,
ring, Boolean ring, Stone topological algebra}
\amsclass{Primary: 18A30, 18D35.
Secondary: 06E15, 08C05, 18C05, 20M50, 20N02
}
\eaddress{gbergman@math.berkeley.edu}

\newcommand{\<}{\kern.0833em}

\newtheorem{question}[theorem]{Question}

\newcommand{\fb}{\mathbf}
\newcommand{\Rep}[2]{\fb{Rep}(#1,\protect\nolinebreak[2]#2)}
\newcommand{\coalg}[2]{\fb{Coalg}(#1,\protect\nolinebreak[2]#2)}
\newcommand{\pscoalg}[2]{\fb{Pseudocoalg}(#1,\protect\nolinebreak[2]#2)}
\newcommand{\C}{\fb{C}\kern.05em}
\mathbfdef{D}
\mathbfdef[Bi]{Binar}
\mathbfdef[Se]{Semigp}
\mathbfdef[Gp]{Group}
\mathrmdef{ari}
\mathrmdef[cd]{card}
\newcommand{\cP}{\raisebox{.1em}{$\<\scriptscriptstyle\coprod$}\nolinebreak[2]}
\newcommand{\limit}{\varprojlim}
\newcommand{\colim}{\varinjlim}
\newcommand{\ba}{_\mathrm{base}}
\newcommand{\circsm}{\kern.25em{\scriptstyle\circ}\kern.25em\nolinebreak[3]}

\begin{document}
\maketitle

\begin{abstract}
If $\C$ and $\D$ are varieties of algebras in the sense
of general algebra, then by a representable functor
$\C\to\D$ we understand a functor which, when
composed with the forgetful functor $\D\to\fb{Set},$ gives a
representable functor in the classical sense; Freyd showed that these
functors are determined by $\!\D\!$-coalgebra objects of $\C.$
Let $\Rep{\C}{\D}$ denote the category of all
such functors, a full subcategory of $\fb{Cat}(\C,\D),$
opposite to the category of $\!\D\!$-coalgebras in $\C.$

It is proved that $\Rep{\C}{\D}$ has small
colimits, and in certain situations, explicit constructions for
the representing coalgebras are obtained.

In particular, $\Rep{\C}{\D}$ always has an initial object.
This is shown to be ``trivial'' unless
$\C$ and $\D$ either both have {\em no} zeroary operations, or both
have {\em more than one} derived zeroary operation.
In those two cases, the functors in question may have
surprisingly opulent structures.

It is also shown that every set-valued representable functor
on $\C$ admits a universal morphism to a $\!\D\!$-valued
representable functor.

Several examples are worked out in detail, and areas for further
investigation are noted.
\end{abstract}
\section*{}

In \S\S\ref{S.intro}-\ref{S.preco>co} below we develop our general
results, and in \S\S\ref{S.BiBi}-\ref{S.Ring>Ab}, some examples.
(One example is also worked in \S\ref{S.Set}, to motivate
the ideas of \S\ref{S.precoalg}.)

R.~Par\'e has pointed out to me that my main result,
Theorem~\ref{T.final}, can be deduced from
\cite[Theorem~6.1.4, p.143, and following remark, and
{\em ibid.}\ Corollary~6.2.5, p.149]{M+P}.
However, as he observes, it is useful to have a direct proof.

\vspace{1em}
{\samepage\begin{center}{\bf I. GENERAL RESULTS.}\end{center}

\section{Conventions; algebras, coalgebras, and representable
functors.}\label{S.intro}

In} this note, morphisms in a category will be composed like
set-maps written to the left of their arguments.
The composition-symbol $\circsm$ will sometimes be introduced as an
aid to the eye, with no change in meaning.
So given morphisms $f:X\to Y$ and $g:Y\to Z,$ their composite
is $g\<f$ or $g\circsm f:X\to Z.$

Throughout, $\C$ and $\D$ will denote fixed varieties
(equational classes) of algebras, in which the operations may have
infinite arities and/or be infinite in number unless the contrary
is stated, but are always required to form a small set.

Let us set up our notation, illustrating it with the case
of the variety $\C.$
We assume given a set $\Omega_\C,$ which will index
the operations, and a cardinal-valued function $\ari_\C$
on $\Omega_\C,$ the arity function.
We fix a regular infinite cardinal $\lambda_\C$ greater than every
$\ari_\C(\alpha)$ $(\alpha\in\Omega_\C).$
An {\em $\!\Omega_\C\!$-algebra} will mean a pair
$A=(|A\<|,(\alpha_A)_{\alpha\in\Omega_\C}),$ where $|A\<|$ is
a set, and for each $\alpha\in\Omega_\C,$ $\alpha_A$ is a set-map
$|A\<|^{\ari_\C(\alpha)}\to|A\<|.$
We denote the category of such algebras, with the obvious
morphisms, by \mbox{$\Omega_\C\!$-$\!\fb{Alg}$}.

For every cardinal $\kappa,$ let $T_{\Omega_\C}(\kappa),$
the ``term algebra'' on $\kappa,$ denote a
free $\!\Omega_\C\!$-algebra on a $\!\kappa\!$-tuple of indeterminates,
and let $\Phi_\C$ be a subset of $\bigcup_{\kappa<\lambda_\C}
|T_{\Omega_\C}(\kappa)|\times |T_{\Omega_\C}(\kappa)|,$
our set of intended identities.
Then we define an object
of $\C$ to mean an $\!\Omega_\C\!$-algebra $A$ with the property
that for every $\kappa<\lambda_\C,$
every map $x:\kappa\to|A\<|$ and every element $(s,t)\in\Phi_\C\cap
(|T_{\Omega_\C}(\kappa)|\times |T_{\Omega_\C}(\kappa)|),$
the $\!\kappa\!$-tuple $x$ ``satisfies the identity $s=t\!$'',
in the sense that the images of $s$ and $t$ under the homomorphism
$T_{\Omega_\C}(\kappa)\to A$ extending $x,$ which
we may denote $s(x)$ and $t(x),$ are equal.
We take $\C$ to be the full subcategory
of \mbox{$\Omega_\C\!$-$\!\fb{Alg}$} with this object-set, and write
$U_\C:\C\to\fb{Set}$ for its underlying set functor.

Note that if $\C$ has no zeroary operations
(i.e., if $\Omega_\C$ has no elements of arity $0),$
then the empty set has a (unique) structure of $\!\C\!$-algebra,
and gives the initial object of~$\C.$

The corresponding notation, $\Omega_\D,$ $\ari_\D,$
$\lambda_\D,$ $\Phi_\D,$ etc., applies to $\D.$
We will generally abbreviate $\ari_\C(\alpha)$ or $\ari_\D(\alpha)$
to $\ari(\alpha)$ when there is no danger of ambiguity.

(Observe that our definition requires that every variety $\C$ be given
with a distinguished set of defining identities $\Phi_\C.$
The choices of $\Phi_\C$ and $\Phi_\D$ will not come into the
development of our main results in \S\S\ref{S.motivate}-\ref{S.Set}.
But $\Phi_\D$ will be called on in
\S\ref{S.precoalg}-\ref{S.preco>co}, where we develop methods
for the explicit construction of our colimit functors.)

If $\fb{A}$ is a category in which all families of
$<\lambda_\D$ objects have coproducts, an
{\em $\!\Omega_\D\!$-coalgebra} $R$ in $\fb{A}$ will mean
a pair $R=(|R\<|,(\alpha^R)_{\alpha\in\Omega_\C}),$ where $|R\<|$ is
an object of $\fb{A},$ and each $\alpha^R$ is a morphism
$|R\<|\to\coprod_{\ari(\alpha)}|R\<|;$
the $\alpha^R$ are the {\em co-operations} of the coalgebra.
For each $\alpha\in\Omega_\D$ and each object $A$ of $\fb{A},$
application of the hom-functor $\fb{A}(-,A)$ to the co-operation
$\alpha^R$ induces an operation
$\alpha_{\fb{A}(R,A)}: \fb{A}(|R\<|,A)^{\ari(\alpha)}\to
\fb{A}(|R\<|,A);$ these together
make $\fb{A}(|R\<|,A)$ an $\!\Omega_\D\!$-algebra,
which we will denote $\fb{A}(R,A).$
We thus get a functor
$\fb{A}(R,-):\fb{A}\to\mbox{$\Omega_\D\!$-$\!\fb{Alg}.$}$

Note that this is a nontrivial extension of standard notation:
If $A$ and $B$
are objects of $\fb{A},$ then $\fb{A}(A,B)$ denotes the hom-{\em set}\/;
if $R$ is an $\!\Omega_\D\!$-coalgebra in $\fb{A},$
then $\fb{A}(R,B)$ denotes an {\em $\!\Omega_\D\!$-algebra} with the
hom-set $\fb{A}(|R\<|,B)$ as underlying set.
We are also making the symbol $|\ |$ do double duty:
If $A$ is an object of the variety
$\C,$ then $|A\<|$ denotes its underlying set; if
$R$ is a coalgebra object in $\C,$ we denote by $|R\<|$
its underlying $\!\C\!$-algebra; thus
$||R\<||$ will be the underlying set of that underlying algebra.
In this situation, ``an element of $R\!$'',
``an element of $|R\<|\!$'' and ``an element of $||R\<||\!$'' will all
mean the same thing, the first two being shorthand for the third.
In symbols we will always write this $r\in||R\<||.$
We also extend the standard language under which the functor
$\fb{A}(|R\<|,-): \fb{A}\to\fb{Set}$ is said to
be {\em representable} with representing object $|R\<|,$ and call
$\fb{A}(R,-): \fb{A}\to\Omega_\D$\!-\!$\fb{Alg} $ a representable
(algebra-valued) functor, with representing coalgebra $R.$

A more minor notational point:  Whenever we write something
like $\coprod_\kappa A,$ this will denote a $\!\kappa\!$-fold copower
of the object $A,$ with the index which ranges over $\kappa$ unnamed;
thus, a symbol such as $\coprod_\kappa A_\iota$ denotes the
$\!\kappa\!$-fold copower of a single object $A_\iota.$
On the few occasions (all in \S\ref{S.preco>co}) where we consider
coproducts other than copowers, we will show the variable index
explicitly in the subscript on the coproduct-symbol,
writing, for instance, $\coprod_{\iota\in\kappa}A_\iota$ to denote
(in contrast to the above)
a coproduct of a family of objects $A_\iota$ $(\iota\in\kappa).$

If $R=(|R\<|,(\alpha^R)_{\Omega_\C})$ is
an $\!\Omega_\D\!$-coalgebra in $\fb{A},$
and $s,t\in|T_{\Omega_\D}(\kappa)|,$ then the necessary
and sufficient condition for the algebras $\fb{A}(R,A)$ to
satisfy the identity $s=t$ for all objects $A$ of $\fb{A}$ is that the
$\!\kappa\!$-tuple of coprojection maps in the $\!\Omega_\D\!$-algebra
$\fb{A}(R,\,\coprod_{\kappa}|R\<|)$ satisfy that relation.
In this case, we shall say that the
$\!\Omega_\D\!$-coalgebra $R$ {\em cosatisfies} the identity $s=t.$
If $R$ cosatisfies all the identities in $\Phi_\D,$
we shall call $R$ a {\em $\!\D\!$-coalgebra} object of $\fb{A},$
and call the functor it represents
a representable $\!\D\!$-valued functor on $\fb{A}.$

Whenever a functor $F:\fb{A}\to\D$ has the property
that its composite with the underlying set functor $U_\D$
is representable in the classical sense, then the $\!\D\!$-algebra
structures on the values of $F$ in fact arise
in this way from a $\!\D\!$-coalgebra
structure on the representing $\!\fb{A}\!$-object.
If $\fb{A}$ not only has small coproducts but general small colimits,
the $\!\D\!$-valued functors that are representable
in this sense are precisely those that have left adjoints
\cite{Freyd}, \cite[Theorem~9.3.6]{245}, \cite[Theorem~8.14]{coalg}.
We shall write $\coalg{\fb{A}}{\D}$ for the category of
$\!\D\!$-coalgebras in $\fb{A},$ taking for the morphisms
$R\to S$ those morphisms $|R\<|\to|S|$ that make
commuting squares with the co-operations, and we shall
write $\Rep{\fb{A}}{\D}$ for the category of representable functors
$\fb{A}\to\D,$ a full subcategory of $\fb{Cat}(\fb{A},\D).$
Up to equivalence, $\Rep{\fb{A}}{\D}$ and
$\coalg{\fb{A}}{\D}$ are opposite categories.

The seminal paper on these concepts is \cite{Freyd} (though the
case where $\C$ is a variety of commutative rings
and $\D$ the variety of groups, or, occasionally, rings,
was already familiar to algebraic geometers, the representable
functors corresponding to the ``affine algebraic groups and rings'').
For a more recent exposition, see \cite[\S\S9.1-9.4]{245}.
In \cite{coalg} the structure of $\coalg{\C}{\D}$ is determined
for many particular varieties $\C$ and $\D.$

(The term ``coalgebra'' is sometimes used for a more elementary
concept:  Given any functor $F: \fb{Set}\to\fb{Set},$
a set $A$ given with a morphism $A\to F(A)$ is called,
in that usage, an $\!F\!$-coalgebra.
The existence of final coalgebras in that sense has also been
studied \cite{PA+NM}, \cite{Barr}, and, as we shall see,
has a slight overlap with the concept studied in this note.
Still another use of ``coalgebra'', probably the earliest,
which also has relations to these two, and is basic to the theory of
Hopf algebras, is that of a module $M$ over a commutative ring,
given with a map $M\to M\otimes M;$ cf.\ \cite[pp.4 to end]{MS} and
\cite[\S\S29-32, \S43]{coalg}.
And in fact, as the referee has pointed out, results analogous to some
of those in this note were proved for coalgebras in that sense over
thirty years ago, by a similar approach~\cite{Barr_74}.)

The first step toward our results will be an easy observation.

\begin{lemma}\label{L.colim}
If $\fb{A}$ has small colimits, then so does $\coalg{\fb{A}}{\D};$
equivalently, $\Rep{\fb{A}}{\D}$ has small limits.
Moreover, the underlying $\!\fb{A}\!$-object of a colimit of
$\!\D\!$-coalgebras in $\fb{A}$ is the colimit of the underlying
$\!\fb{A}\!$-objects of these algebras.
Equivalently, the composite
of the limit of the corresponding representable functors with the
underlying set functor $U_\D$ is the limit of the
composites of the given functors with $U_\D,$ and may thus be evaluated
at the set level by taking limits of underlying sets.
\end{lemma}

\begin{proof}
Since $\D$ has small limits, the category of $\!\D\!$-valued functors
on any category also has small limits, which may be computed
object-wise; hence by properties of limits of
algebras, these commute with passing to underlying sets.
In particular, a functor $F$ from a small category $\fb{E}$ to
$\Rep{\fb{A}}{\D}$ will have a limit $L$ in $\D^\fb{A},$
and $U_\D\circsm L$ will
be the limit of the set-valued functors $U_\D\circsm F(E).$
The latter limit is represented by the colimit of the underlying
$\!\fb{A}\!$-objects of the coalgebras representing the $F(E);$
call this colimit object $|R\<|.$
As noted above, representability of $U_\D\circsm L: \fb{A}\to\fb{Set}$
by $|R\<|$ implies representability of $L:\fb{A}\to\D$
by a coalgebra $R$ with underlying $\!\fb{A}\!$-object $|R\<|.$

(One can get the same result starting at the coalgebra end,
using the general fact that ``colimits commute with colimits'' to
deduce that a colimit of underlying objects of a diagram of coalgebras
inherits co-operations from these, and verify that the resulting
coalgebra has the universal property of the desired colimit.)
\end{proof}

Knowing that $\Rep{\fb{A}}{\D}$ has small limits, to prove that it has
small {\em colimits} we need ``only'' prove that it satisfies the
appropriate solution-set condition (\cite[Theorem~V.6.1]{CW},
\cite[Theorem~7.10.1]{245}).
Easier said than done!

We shall get such a result in the case where $\fb{A}$ is a variety $\C.$
In the next section, we motivate the technique to be used (after
giving a couple of results showing cases to avoid when
thinking about examples), then preview the remainder of the paper.

\section{Background; trivial cases; motivation of the
proof; overview of the paper.}\label{S.motivate}

Let me first describe what led me to the questions answered below.

Let $\fb{Ring}^1$ denote the category of associative unital rings,
and for any field $K,$ let $\fb{Ring}^1_K$ denote the category of
associative unital $\!K\!$-algebras.
(When I write ``$\!K\!$-algebra'', ``algebra'' will be meant in
the ring-theoretic sense; otherwise it is
always meant in the general sense.)
In \cite[\S25]{coalg}, descriptions are obtained of all
representable functors from $\fb{Ring}^1_K$ into a number of
varieties of algebras (general sense!), including $\fb{Ring}^1.$
The result for the lastmentioned variety says
that for every representable functor
$F: \fb{Ring}^1_K\to\fb{Ring}^1,$ there exist two
{\em linearly compact} associative unital $\!K\!$-algebras $A$
and $B,$ such that $F$ is isomorphic to the functor
taking every object $S$ of $\fb{Ring}^1_K$ to the direct product
of completed tensor-product rings
$S\hat{\otimes}A\,\times\,S^{\mathrm{op}}\hat{\otimes}B.$
(One doesn't need to know precisely what these terms mean
to appreciate the point that is coming up.
For general background:
the category of linearly compact $\!K\!$-vector spaces is dual to the
category of all $\!K\!$-vector spaces, and a
linearly compact topology on an associative unital $\!K\!$-algebra
$A$ makes it an inverse limit of finite-dimensional $\!K\!$-algebras
\cite[\S24]{coalg}.
In this situation, the {\em completed} tensor product of a
$\!K\!$-algebra $S$ with $A$ is the inverse limit of the
tensor products of $S$ with those finite-dimensional $\!K\!$-algebras.)

Now the category of linearly compact associative unital $\!K\!$-algebras
has an initial object, the field $K$ given with the discrete topology,
and under the above characterization of representable
functors, the operation of completed tensor product with this
object describes the forgetful functor $\fb{Ring}^1_K\to\fb{Ring}^1.$
Hence the result cited shows
that $\Rep{\fb{Ring}^1_K}{\,\fb{Ring}^1}$ has an initial object, the
functor taking $S$ to $S\times S^{\mathrm{op}},$ regarded as a ring.
The coalgebra representing this functor has for underlying
object the free associative unital $\!K\!$-algebra
in two indeterminates, $K\langle x,y\rangle.$

(The same conclusion is true for $K$ any commutative ring, and,
in fact, in a still more general context \cite[\S28]{coalg}.
However, there is no analog in these contexts to the duality between
vector spaces and linearly compact vector spaces, and hence no
interpretation of representable functors in terms of completed
tensor products.
Indeed, the existence and description of the initial
object might not have been discovered
without the motivation of the case where $K$ is a field.)

The above result suggests that for general varieties $\C$ and
$\D,$ the category $\Rep{\C}{\D}$ might have an initial object,
which might have a non-obvious form.
It was this tantalizing hint that led to the present investigations.

Curiously, if, in the preceding example, the unitality condition is
dropped from the class of rings taken as the domain variety, {\em or}
the codomain variety, {\em or} both, then $\Rep{\C}{\D}$
still has an initial object, but not an ``interesting'' one
-- it is trivial, represented by the $\!0\!$- or
$\!1\!$-dimensional $\!K\!$-algebra depending on the case.
This sort of triviality occurs in some very general classes of
situations.
Let us now prove this, so that the reader who wishes to think about
the arguments of later sections in the light of examples of her or
his choosing will be able to consider
cases that have a chance of being nontrivial.

\begin{theorem}\label{T.7/9}
\textup{(i)}\ \ Suppose $\fb{A}$ is a variety of algebras
with no zeroary operations \textup{(}or more generally,
is a category with small coproducts such
that the initial object of $\fb{A}$ admits no
morphisms from non-initial objects into it\textup{)}.
Let $\D$ be
a variety of algebras having {\em at least one} zeroary operation.
Then for any representable functor $F:\fb{A}\to\D,$
the underlying $\!\fb{A}\!$-object of the coalgebra representing
$F$ is the initial object of $\fb{A};$ hence $F$ is the functor taking
every object to the $\!1\!$-element $\!\D\!$-algebra.
\textup{(}So, in particular, $\Rep{\fb{A}}{\D}$ has this trivial functor
as its initial object.\textup{)}\\[.5em]
\textup{(ii)}\ \ Suppose $\fb{A}$ is a variety of algebras
with a unique derived zeroary operation \textup{(}or more generally,
is a category with small coproducts whose initial object is also
a final object\textup{)}, and let $\D$ be any variety of algebras.
Then $\Rep{\fb{A}}{\D}$ has an initial object, namely the functor
taking all objects of $\fb{A}$ to the $\!1\!$-element algebra in $\D,$
represented by the initial-final object of $\fb{A}$ with the
unique $\!\D\!$-coalgebra structure that it admits.
\textup{(}However, in this case, $\Rep{\fb{A}}{\D}$ may have
nontrivial {\em non-initial} objects.\textup{)}\\[.5em]
\textup{(iii)}\ \ Suppose $\fb{A}$ is a variety of algebras with
more than one derived zeroary operation \textup{(}or more generally,
is a category with small coproducts and a final
object, such that the unique morphism
from the initial to the final object is an epimorphism but
not invertible\textup{),} and let $\D$ be any variety of algebras
having at most one derived zeroary operation.
Then $\Rep{\fb{A}}{\D}$ has an initial object $F.$
Namely --

\vspace{.5em}
\textup{(iii.a)}\ \ If $\D$ has no zeroary operations, $F$ is
represented by the final object of $\fb{A},$ which under the
above hypotheses has a unique $\!\D\!$-coalgebra
structure, and takes objects of $\fb{A}$ that admit morphisms
from the final object into them to the $\!1\!$-element
\textup{(}final\textup{)} $\!\D\!$-algebra, and objects which do not
admit such a morphism to the empty \textup{(}initial\textup{)}
$\!\D\!$-algebra.

\textup{(iii.b)}\ \ If $\D$ has exactly one derived zeroary operation,
then $F$ is represented by the initial object of $\fb{A},$ with
its unique $\!\D\!$-coalgebra
structure, and so takes all objects to the $\!1\!$-element algebra.

\textup{(}But in these situations, too, $\Rep{\fb{A}}{\D}$ may have
nontrivial non-initial objects.\textup{)}

\vspace{.5em}
Thus, assuming $\fb{A}$ a variety of algebras, the only situations
where $\Rep{\fb{A}}{\D}$ can have an initial object whose values are
not exclusively $\!0\!$- or $\!1\!$-element algebras are if
{\em neither} $\fb{A}$ nor $\D$ has zeroary operations, or if
{\em both} $\fb{A}$ and $\D$ have {\em more than} one derived zeroary
operation.
\end{theorem}

\begin{proof}
First, some general observations.
Recall that if a category $\fb{A}$ has an initial object $I,$ then
this is the colimit of the empty diagram, and in particular, is
the coproduct of the empty family of algebras.
It follows that any copower of $I$ is again $I.$
From this we can see that $I$ has a unique
$\!\D\!$-coalgebra structure for every variety $\D.$

More generally, suppose $\fb{A}$ has an initial object $I$
and that $J$ is an epimorph of $I.$
(For example, $\fb{A}$ might be the category of unital commutative
rings, so that $I$ is the ring
$\mathbb{Z}$ of integers, and $J$ might be a prime
field, $\mathbb{Z}/p\mathbb{Z}$ or $\mathbb{Q}.)$
Then a copower $\coprod_\kappa J$ can be identified with the
pushout of the system of maps from $I$
to a $\!\kappa\!$-tuple of copies of $J$ (because
$I$ is initial), so by the assumption that $I\to J$ is an epimorphism,
if $\kappa\neq 0$ that pushout will again be $J.$
(On the other hand, if $\kappa=0,$ then that pushout is $I.)$
The argument of the preceding paragraph now generalizes to show that
for every $\D$ without zeroary operations, such an
object $J$ also has a unique structure of $\!\D\!$-coalgebra.

With these observations in mind, we shall prove the various statements
of the theorem, in each case under the ``more general'' hypothesis.
(In each case, it is easy to see that {}that hypothesis holds
for the particular class of varieties with which the statement begins.)

In the situation of~(i), since $\D$ has a zeroary
operation, the representing object $R$ of $F$ must
have a zeroary co-operation, i.e., a morphism in $\fb{A}$ from
$|R\<|$ to the initial object.
But by assumption on $\fb{A},$
that can only happen if $|R\<|$ {\em is} the initial
object, giving the indicated conclusion.

In the situation of~(ii), we can form a $\!\D\!$-coalgebra in $\fb{A}$
by taking the initial-final object (often called a {\em zero} object)
$Z$ as underlying object, and noting as above that as the
initial object, $Z$ has a unique $\!\D\!$-coalgebra structure.
Because $Z$ is also a final object, the underlying $\!\fb{A}\!$-object
of every $\!\D\!$-coalgebra in $\fb{A}$ admits a unique map to $Z,$
and this clearly forms commuting squares with the co-operations.
Thus, the corresponding coalgebra is final in $\coalg{\fb{A}}{\D},$
and so determines an initial representable functor.

To verify the parenthetical assertion that representable functors other
than this one may also exist, let $\fb{A}=\Gp,$ and note that
the forgetful functor to $\Se$ is representable, as is
the identity functor of $\fb{A}.$
These two examples cover the cases where $\D$ has no zeroary
operations or a unique derived zeroary operation.
If $\D$ has more than one derived zeroary operation,
a representable functor from a category $\fb{A}$ of the indicated sort
must take values in the proper subvariety of $\D$ determined by the
identities saying that the values of all these operations are equal,
since an object of $\fb{A}$ clearly has a unique zeroary co-operation.
Although for some $\D$ (e.g., $\fb{Ring}^1)$ that subvariety is
trivial, for others it is not.
For instance, if $\D$ is the variety of groups with an
additional distinguished element, then the subvariety in
question consists of groups with the identity as that
element, and the functor from $\Gp$ to that subvariety
which leaves the group structure unchanged is not trivial.

In case~(iii.a), the assumption that the unique map from the initial to
the final object of $\fb{A}$ is an epimorphism implies (by the
observations of the second paragraph of this proof) that the
final object of $\fb{A}$ admits a unique $\!\D\!$-coalgebra structure
for every variety $\D$ with no zeroary operations.
It is easy to see that the resulting coalgebra is final in
$\coalg{\fb{A}}{\D},$ as asserted; the description
of the functor represented is also clear.

In the case~(iii.b), where $\D$ has a unique derived zeroary
operation, the conclusion actually requires no assumption
on $\fb{A}$ but that it have an initial object.
(Indeed, statements~(i) and~(ii) imply the same conclusion for
such $\D.)$
To get that conclusion, observe that the value of the unique zeroary
derived operation of $\D$ yields a unique one-element subalgebra in
every object of $\D$ (e.g., when $\D=\Gp$ or $\fb{Monoid},$
the subalgebra $\{e\}).$
Hence the functor represented by the initial object of $\fb{A}$
with its unique $\!\D\!$-coalgebra structure, taking every
object to the one-element $\!\D\!$-algebra, has a unique morphism to
every functor $\fb{A}\to\D,$ hence is initial in $\D^\fb{A},$ and
so, a fortiori, in~$\Rep{\fb{A}}{\D}.$

In these two situations, identity functors and forgetful functors again
show that not every representable functor need be trivial.
\end{proof}

The above theorem, in the case where $\fb{A}$ is
a variety $\C$ of algebras, is summarized in the chart below.
In that chart, $I$ means that the initial representable functor is
represented by the initial object, $F$ means that it is represented
by the final object,
$IF$ means it is represented by the initial-final object, an
exclamation point means that the functor so represented is
the only representable functor, and an exclamation point in
parenthesis means that, though the functor in question may not
be the only one, there is a strong restriction on representable
functors, namely that they take as values algebras in which all
derived zeroary operations are equal.
Stars mark the two cases in which nontrivial initial representable
functors can occur.

\begin{picture}(350,150)
\put(140,145){zeroary derived operations of $\D$}
\put(160,126){$0$}
\put(210,126){$1$}
\put(255,126){$>1$}
\multiput(140,120)(0,-35){4}{\line(1,0){150}}
\multiput(140,120)(50,0){4}{\line(0,-1){105}}
\put(121,098){$0$}
\put(160,098){\Large{$*$}}
\put(205,098){$I$!}
\put(256,098){$I$!}
\put(121,063){$1$}
\put(155,063){$IF$}
\put(205,063){$IF$}
\put(246,063){$IF$(!)}
\put(117,028){${>}\<1$}
\put(160,028){$F$}
\put(210,028){$I$}
\put(259,028){\Large{$*$}}
\put(065,105){zeroary}
\put(065,078){derived}
\put(050,051){operations}
\put(075,024){of $\C$}
\end{picture}

In contrast to the triviality of the seven cases covered by the above
theorem, the structure of the initial representable functor
in the two starred cases can be surprisingly rich; we shall
see this for a case belonging to the upper left-hand corner in
\S\ref{S.Set}, and for further examples of both cases
in \S\S\ref{S.BiBi}-\ref{S.Group}.
However, let us note a subcase of the upper left-hand
case where the functors one gets are again fairly degenerate.

\begin{proposition}\label{P.unary}
Suppose $\D$ is a polyunary variety; i.e.,
that $\ari(\alpha)=1$ for all $\alpha\in\Omega_\D.$
Then for any variety $\fb{A}$ of algebras \textup{(}or more
generally, for any category $\fb{A}$ having small colimits and a final
object\textup{)}, the category $\Rep{\fb{A}}{\D}$ has an initial object
$F,$ represented by the final object $T$ of $\fb{A},$ with its
unique $\!\D\!$-coalgebra structure, namely the structure in which
every primitive co-operation is the identity.

Thus, if $\fb{A}$ is a variety, then for every object $A$ of $\fb{A},$
the algebra $F(A)$ has for underlying set
the set of those $x\in|A\<|$ such that $\{x\}$ forms a
subalgebra of $A,$ and has the $\!\D\!$-algebra structure in
which every primitive operation of $\D$ acts as the identity.
\end{proposition}

\begin{proof}
Since $\fb{A}(T,T)$ is a trivial monoid, there is a
unique way to send the primitive unary operations of $\D$
to maps $T\to T,$ and this will make $T$ a $\!\D\!$-coalgebra.
Recalling that $T$ is final in $\fb{A}$ we see that for every
$\!\D\!$-coalgebra $R$ in $\fb{A},$ the unique morphism
$|R\<|\to T$ becomes a morphism of coalgebras, so the coalgebra
described is final in $\coalg{\fb{A}}{\D}.$

The description of the functor represented by this
coalgebra when $\fb{A}$ is a variety is immediate.
\end{proof}

From this point on, when the contrary is not stated,
the domain category of our functors will
be the variety $\C$ (assumed fixed in \S\ref{S.intro});
so when we use the terms ``representable functor'' and
``coalgebra'' without qualification, these will mean ``representable
functor $\C\to\D\!$'', and ``$\!\D\!$-coalgebra in $\C\!$''.

Let me now motivate the technique of proof of the existence of final
objects in $\coalg{\C}{\D},$ then
indicate how to extend that technique to other limits.

Our Lemma~\ref{L.colim} and Freyd's Initial Object Theorem will
establish the existence of an initial representable functor if we can
find a small set $W_\mathrm{funct}$ of representable functors such that
every object of $\Rep{\C}{\D}$ admits a morphism from some member of
$W_\mathrm{funct};$ equivalently, a small set $W_\mathrm{coalg}$ of
coalgebras such that every coalgebra admits a morphism
into a member of $W_\mathrm{coalg}.$
Let us call an object $R$ of $\coalg{\C}{\D}$ ``strongly quasifinal''
if every morphism of coalgebras with domain $R$
which is surjective on underlying algebras is an isomorphism.
(I will motivate this terminology in a moment.)
It will not be hard to show that every coalgebra admits a
morphism onto a strongly quasifinal coalgebra, so it will suffice to
show that up to isomorphism, there is only a small set of these.
To see how to get this smallness condition, note that no strongly
quasifinal coalgebra $R$ can be the codomain of
two distinct morphisms of coalgebras with a common domain $R';$
for if it were, then their coequalizer in $\coalg{\C}{\D}$ (which
by Lemma~\ref{L.colim} has for underlying $\!\C\!$-algebra the
coequalizer of the corresponding maps in $\C)$
would contradict the strong quasifinality condition.
(It is natural to call the property of admitting at most one
morphism from every object ``quasifinality'', hence our use of
``strong quasifinality'' for the above condition that implies it.)
Thus, if we can find a small set $V$ of coalgebras such
that every coalgebra is a union of subcoalgebras isomorphic to
members of $V$ (where we shall define a ``subcoalgebra'' of $R$
to mean a coalgebra which can be mapped into $R$ by a coalgebra
morphism that is one-to-one on underlying $\!\C\!$-algebras, and
also induces one-to-one maps of their copowers),
then the cardinalities of strongly quasifinal coalgebras $R$ will
be bounded by the sum of the cardinalities
of the members of $V$ (since each member of $V$ can be mapped
into $R$ in at most one way), giving the required smallness condition.

Naively, we would like to say that each coalgebra $R$ is the
union of the subcoalgebras ``generated'' by the elements $r\in|R\<|,$
bound the cardinalities of coalgebras that can be
``generated'' by single elements, and take the set of such
coalgebras as our $V.$
Unfortunately, there is not a well-defined concept of the
subcoalgebra generated by an element.
Nevertheless, we shall be able to build up, starting with any
element, or more generally, any subset $X$ of $||R\<||,$ a
subcoalgebra which contains $X,$ and whose cardinality can be
bounded in terms of that of $X.$

How?
We will begin by closing $X$ under the $\!\C\!$-algebra
operations of $|R\<|.$
We will then consider the image of the resulting subalgebra of $|R\<|$
under each {\em co-operation}
$\alpha^R: |R\<|\to\coprod_{\ari(\alpha)}|R\<|$ $(\alpha\in\Omega_\D).$
This image will be contained in the subalgebra of
$\coprod_{\ari(\alpha)}|R\<|$ generated by the images,
under the $\ari(\alpha)$ {\em coprojections}
$|R\<|\to\coprod_{\ari(\alpha)}|R\<|,$ of certain elements of $|R\<|.$
The set of these elements is not, in general, unique, but the number
of them that are needed can be bounded with the help of $\lambda_\C.$
Taking the subalgebra of $|R\<|$ that they generate, we then repeat this
process; after $\lambda_\C$ iterations, it will stabilize, giving
a subalgebra $|R\<|'\subseteq|R\<|$ such that each co-operation
$\alpha^R$ carries $|R\<|'$ into the subalgebra
of $\coprod_{\ari(\alpha)}|R\<|$ generated by the images of $|R\<|'$
under the coprojections.

This $|R\<|'$ still may not define a subcoalgebra of $R,$ because when
we use the inclusion $|R\<|'\subseteq|R\<|$ to induce homomorphisms
\begin{equation}\begin{minipage}[c]{35pc}\label{x.R'R}
$\coprod_{\ari(\alpha)}|R\<|'\ \to\ \coprod_{\ari(\alpha)}|R\<|
\qquad(\alpha\in\Omega_\D),$
\end{minipage}\end{equation}
these may not be one-to-one; hence the co-operations
$\alpha^R: |R\<|\to\coprod_{\ari(\alpha)}|R\<|,$
though they carry $|R\<|'$ into the image of~(\ref{x.R'R}),
may not lift to co-operations $|R\<|'\to\coprod_{\ari(\alpha)}|R\<|'.$

However, given $|R\<|',$ we can now find a larger subalgebra
$|R\<|'',$ whose cardinality we can again bound, such that the inclusion
$|R\<|''\subseteq|R\<|$ does induce one-to-one maps
$\coprod_\kappa|R\<|''\to\coprod_\kappa|R\<|$ for
all cardinals $\kappa.$
We then have to repeat the process of the preceding paragraph
so that the images of our new algebra under the co-operations
of $R$ are again contained in the subalgebras generated by
the images of the coprojections.
Applying these processes alternately (or better, applying one step
of each alternately, since nothing is gained by iterating one process
to completion before beginning the other), we get a {\em subcoalgebra}
of $R$ containing $X,$ whose cardinality we can bound.

Taking such a bound $\mu$ for the case $\cd(X)=1,$ any set of
representatives of the isomorphism classes of coalgebras of cardinality
$\leq\mu$ gives the $V$ needed to complete our proof.

In view of the messiness of the above construction, the bound on the
cardinality of the
final coalgebra that it leads to is rather large.
But this reflects the reality of the situation.
For instance, if $\C$ is $\fb{Set},$ and $\D$ the variety of sets given
with a single binary operation, we shall see in \S\ref{S.Set} that
the final object of $\coalg{\C}{\D}$ is the Cantor set, with
co-operation given by the natural bijection from that set
to the disjoint union of two copies of itself.
Since this coalgebra has cardinality $2^{\aleph_0},$ the functor
it represents takes a $\!2\!$-element set to a $\!\D\!$-algebra
of cardinality $2^{2^{\aleph_0}}.$
(Incidentally, because in this example, $\C=\fb{Set}$ and the variety
$\D$ is defined without the use of identities, this
final coalgebra is also the
final coalgebra in the sense of \cite{PA+NM} and \cite{Barr},
with respect to the functor taking
every set to the disjoint union of two copies of itself.)

How will we modify the above construction of final
coalgebras to get general small limits in $\coalg{\C}{\D}$?
Consider the task of finding a product of objects
$R_1$ and $R_2$ in this category, i.e., an object universal
among coalgebras $R$ with morphisms $f_1:R\to R_1$ and $f_2:R\to R_2.$
Such a pair of maps corresponds, at the algebra level, to a map
$f:|R\<|\to|R_1|\times|R_2|;$ let us call the latter algebra $S\ba.$
To express in terms of $f$ the fact that $f_1$ and $f_2$
are compatible with the co-operations
of $R_1$ and $R_2,$ let us, for each $\alpha\in\Omega_\D,$ define
$S_\alpha$ to be the $\!\C\!$-algebra
$\coprod_{\ari(\alpha)}|R_1|\times\coprod_{\ari(\alpha)}|R_2|.$
The $\ari(\alpha)$ coprojection maps
$|R_i|\to\coprod_{\ari(\alpha)}|R_i|$ $(i=1,2)$ induce
$\ari(\alpha)$ maps $S\ba\to S_\alpha;$ let us call these
``pseudocoprojections''; thus, each $S_\alpha$ is an object of
$\C$ with $\ari(\alpha)$ pseudocoprojection maps
of $S\ba$ into it, and an additional map of $S\ba$ into it,
induced by $\alpha^{R_1}$ and $\alpha^{R_2},$ which
we shall call the ``pseudo-co-operation'' $\alpha^S.$
We shall call systems $S$ of objects and morphisms of the
sort exemplified by this construction ``pseudocoalgebras''
(Definition~\ref{D.pseudo} below).

Then a coalgebra $R$ with morphisms into $R_1$ and $R_2$ can be regarded
as a coalgebra with a morphism into the above pseudocoalgebra $S.$
More generally, if we are given any small diagram of coalgebras,
then a cone from a coalgebra $R$ to that diagram is equivalent to a
morphism from $R$ to an appropriate pseudocoalgebra.
Thus, it will suffice to show that for every pseudocoalgebra $S,$
the category of coalgebras with morphisms to $S$ has a final object.
The construction sketched above for final coalgebras in fact
goes over with little change to this context.

After obtaining this existence result in the next two
sections, we will show that in many cases,
these colimits can be constructed more explicitly,
as inverse limits of what we shall call ``precoalgebras''.
(The example mentioned above where the final coalgebra
is the Cantor set will lead us to that approach.)

\section{Subcoalgebras of bounded cardinality.}\label{S.sub}

Recall that $\lambda_\C$ is a regular infinite
cardinal such that every primitive operation of $\alpha\in\Omega_\C$
has $\ari(\alpha)<\lambda_\C.$
A standard result is

\begin{lemma}[{\cite[Lemma~8.2.3]{245}}]\label{L.gen_by}
If $A$ is an algebra in $\C$ and $X$ a generating set for $A,$
then every element of $A$ is contained in a subalgebra generated
by $<\lambda_\C$ elements of $X.$\endproof
\end{lemma}

Let us fix a notation for algebras presented by
generators and relations.
For any set $X,$ let $F_\C(X)$ denote the free algebra on $X$ in $\C.$
If $Y$ is a subset of $|F_\C(X)|\times|F_\C(X)|,$
let $\langle X\mid Y\rangle_\C$ be the quotient of
$F_\C(X)$ by the congruence generated by $Y.$
If $A$ is a $\!\C\!$-algebra, a presentation of $A$
will mean an isomorphism with an algebra $\langle X\mid Y\rangle_\C.$
Every algebra $A$ has a canonical presentation, with $X=|A\<|,$ and
$Y$ consisting of all pairs $(\alpha_{F_\C(X)}(x),\alpha_A(x))$
with $\alpha\in\Omega_\C$ and $x\in|A\<|^{\ari(\alpha)}.$

If $X_0$ is a subset of $X,$ we shall often regard $F_\C(X_0)$
as a subalgebra of $F_\C(X).$

Here is the analog of the preceding lemma for relations.

\begin{corollary}\label{C.rels}
Let $A=\langle X\mid Y\rangle_\C,$ let $X_0$ be a subset
of $X,$ and let $p$ and $q$ be elements of $F_\C(X_0)$ which
fall together under the composite of natural
maps $F_\C(X_0)\hookrightarrow F_\C(X)\to A\<.$

Then there exist a set $X_1$ with $X_0\subseteq X_1\subseteq X,$
and a set $Y_1\subseteq Y\cap(|F_\C(X_1)|\times|F_\C(X_1)|),$ such that
the difference-set $X_1-X_0$ and the set $Y_1$ both
have cardinality $<\lambda_\C,$ and such that
$p$ and $q$ already fall together under the composite map
$F_\C(X_0)\hookrightarrow F_\C(X_1)\to
\langle X_1\mid Y_1\rangle_\C.$
\end{corollary}

\begin{proof}
By hypothesis, $(p,q)$ lies in the congruence on $F_\C(X)$
generated by $Y.$
A congruence on $F_\C(X)$ can be described as a subalgebra
of $F_\C(X)\times F_\C(X)$ which is also an equivalence relation
on $|F_\C(X)|,$ and the latter condition can
be expressed as saying that it contains all pairs $(r,r)$
$(r\in|F_\C(X)|),$ and is closed under both the unary
operation $(r,s)\mapsto(s,r)$ and the {\em partial}
binary operation carrying a pair of elements of
the form $((r,s),\ (s,t))$ to the element $(r,t).$

We can apply the preceding lemma to this situation, either using
the observation that the proof of that lemma works equally well
for structures with partial operations, or by noting that if we
extend the above partial operation to a total operation by
making it send $((r,s),\ (s',t))$ to $(r,s)$ if $s\neq s',$
then closure under that total operation is equivalent to closure
under the given partial operation.
Either way, we get the conclusion that if we extend the
$\!\C\!$-algebra structure of $F_\C(X)\times F_\C(X)$ to embrace the
two additional operations expressing symmetry and transitivity, then
our given element $(p,q)$ lies in the subalgebra of the
resulting structure generated by some subset
\begin{quote}
$Y_0\ \subseteq\ Y\cup\{(r,r)\mid r\in|F_\C(X)|\}$
\end{quote}
of cardinality $<\lambda_\C.$
The elements of $Y_0$ coming from $Y$ will form
a set $Y_1$ of cardinality $<\lambda_\C,$ and
by the preceding lemma, each component of each of these elements
will lie in the subalgebra of $F_\C(X)$
generated by some subset of $X$ of cardinality $<\lambda_\C,$
and the same will be true of each of the $<\lambda_\C$ elements $r$
occurring in pairs $(r,r)\in Y_0.$
As $\lambda_\C$ is a regular cardinal, the union of these
subsets will be a set
$X_1'\subseteq X$ of cardinality $<\lambda_\C$ such that $Y_0$
lies in $|F_\C(X_1')|\times |F_\C(X_1')|.$
By construction, $(p,q)$ lies in the subalgebra of this product
generated by $Y_0$ under our extended algebra structure.
Letting $X_1=X_0\cup X_1',$ we conclude that $p$ and $q$ fall together
in $\langle X_1\mid Y_1\rangle_\C,$ as claimed, and that $X_1$
and $Y_1$ satisfy the desired cardinality restrictions.

(Incidentally, this proof would not have worked if we had rendered the
elements $(r,r)$ by zeroary operations, rather than generating elements,
since $|F_\C(X_1)|\times |F_\C(X_1)|$ would not have been closed
in $|F_\C(X)|\times |F_\C(X)|$ under all these operations.)
\end{proof}

To go from the bounds on the cardinalities of the sets constructed
in the above lemma and corollary to bounds
on the cardinalities of the algebras they generate, we will want

\begin{definition}\label{D.exp<}
If $\kappa$ is a cardinal and $\lambda$ an infinite cardinal, then
$\kappa^{\lambda-}$ will denote the least cardinal $\mu\geq\kappa$
such that $\mu^\iota=\mu$ for all $\iota<\lambda.$
\end{definition}

To see that this makes sense, note that $\kappa^\lambda$ is
not changed on exponentiating by $\lambda,$ hence, a fortiori,
it is not changed on exponentiating by any positive $\iota<\lambda,$
so the class of cardinals $\mu$ with that property is nonempty; hence
it has a least member.

Immediate consequences of the above definition are
\begin{equation}\begin{minipage}[c]{35pc}\label{x.^*l-twice}
$(\kappa^{\lambda-})^{\lambda-}=\ \kappa^{\lambda-}.$
\end{minipage}\end{equation}
\begin{equation}\begin{minipage}[c]{35pc}\label{x.sup}
$\max(\kappa,\mu)^{\lambda-}\ =\ %
\max(\kappa^{\lambda-},\mu^{\lambda-}).$
\end{minipage}\end{equation}
Note also that for $\iota<\lambda$ and $\kappa>1$ we have
$\kappa^{\lambda-}\geq\kappa^\iota>\iota.$
Hence
\begin{equation}\begin{minipage}[c]{35pc}\label{x.geq*l}
If $\kappa>1,$ then $\kappa^{\lambda-}\geq\lambda.$
\end{minipage}\end{equation}
We also see
\begin{equation}\begin{minipage}[c]{35pc}\label{x.^aleph0}
For all $\kappa>1$ one has
$\kappa^{\aleph_0-}=\ \max(\kappa,\,\aleph_0).$
\end{minipage}\end{equation}

(Leo Harrington has pointed out to me that for $\lambda$ a
regular cardinal, which will always be the case below,
what I am calling $\kappa^{\lambda-}$ can be shown
equal to what set theorists call $\kappa^{<\lambda},$
namely $\sup_{\iota<\lambda}\kappa^\iota.$
However, we shall not need this fact.)

The following result is very likely known.

\begin{lemma}\label{L.gen_card}
If a $\!\C\!$-algebra $A$ is generated by a set $X,$ then
\begin{quote}
$\cd(|A\<|)\ \leq\ \max(\cd(X),\nolinebreak[2]\,
\cd(\Omega_\C),\,2)^{\lambda_\C-}.$
\end{quote}
\end{lemma}

\begin{proof}
We construct subsets $X_0\subseteq X_1\subseteq\dots
\subseteq X_{\lambda_\C}$ of $|A\<|$ as follows:
Take $X_0=X.$
For every successor ordinal $\iota\<{+}1$ let
$X_{\iota+1}$ consist of all elements of $X_\iota$ and all
elements of the form $\alpha_A(x)$ where $\alpha\in\Omega_\C$
and $x=(x_\gamma)_{\gamma\in\ari(\alpha)}\in X_\iota^{\ari(\alpha)}.$
For every limit ordinal $\iota$ let
$X_\iota=\bigcup_{\eta\in\iota} X_\eta.$

Since $\lambda_\C$ is a regular cardinal exceeding all the
cardinals $\ari(\alpha)$ $(\alpha\in\Omega_\C),$ we see that
$X_{\lambda_\C}$ is closed under the operations of $\C,$
so as it contains $X_0=X,$ it is all of $|A\<|.$
Let us now show by induction that for
all $\iota\leq\lambda_\C,$
\begin{equation}\begin{minipage}[c]{35pc}\label{x.leqmu}
$\cd(X_\iota)\ \leq\ \max(\cd(X),\,\cd(\Omega_\C),\,2)^{\lambda_\C-}.$
\end{minipage}\end{equation}
Clearly,~(\ref{x.leqmu}) holds for $\iota=0.$
Let us write $\mu$ for the right-hand side of~(\ref{x.leqmu}), which
is independent of $\iota,$ and by~(\ref{x.geq*l}) is $\geq\lambda_{\C}.$
At each successor ordinal $\iota\<{+}1,$ the number of elements
we adjoin as values of each operation $\alpha_A$ is at most
$\cd(X_\iota)^{\ari(\alpha)}\leq\mu^{\ari(\alpha)}\leq\mu,$
since $\ari(\alpha)<\lambda_\C.$
Hence, doing this for all $\cd(\Omega_\C)$ operations brings
in $\leq\mu\cdot\cd(\Omega_\C)\linebreak[0]\leq\mu$ elements.
Likewise, at a limit ordinal $\iota,$ we take the union of a family
of $\cd(\iota)\leq\linebreak[0]\lambda_\C\leq\mu$
sets of cardinality $\leq\mu,$
hence again get a set of cardinality $\leq\mu.$

Taking $\iota=\lambda_\C$ in~(\ref{x.leqmu}), we get
$\cd(|A\<|)\leq\mu,$ as required.
\end{proof}

For the step in our proof where we will enlarge an arbitrary
subalgebra of $|R\<|$ to a subcoalgebra of $R,$ we will first need
to define ``subcoalgebra''.
For this in turn we will need

\begin{definition}\label{D.pure}
A subalgebra $A$ of a $\!\C\!$-algebra $B$ will be called
{\em copower-pure} if for every cardinal $\kappa,$ the induced map
\begin{equation}\begin{minipage}[c]{35pc}\label{x.pure}
$\coprod_\kappa A\ \to\ \coprod_\kappa B$\ \ is one-to-one.
\end{minipage}\end{equation}
When this holds, we shall often identify
$\coprod_\kappa A$ with its image in~$\coprod_\kappa B\<.$
\end{definition}

An example of a subalgebra which is not copower-pure was noted
in \cite[discussion preceding Question~4.5]{embed}:
Let $\C$ be the variety of groups determined by the identities satisfied
in the infinite dihedral group, which include $x^2 y^2=y^2 x^2,$
but not $xy=yx,$ nor $x^n=1$ for any $n>0.$
Let $B$ be the infinite cyclic group $\langle x\rangle,$
which is free on one generator in $\C,$
and $A$ the subgroup $\langle x^2\rangle\subseteq B\<.$
Then $B\cP B$ is the free $\!\C\!$-algebra on two generators,
$x_0$ and $x_1,$ hence is noncommutative,
hence the same is true of $A\cP A.$
But the image of $A\cP A$ in $B\cP B$ is generated by $x_0^2$
and $x_1^2,$ which commute by the identity noted; so the
map $A\cP A\to B\cP B$ is not an embedding.

\begin{lemma}\label{L.pure<*l}
A subalgebra $A$ of a $\!\C\!$-algebra $B$ is copower-pure if and
only if~\textup{(\ref{x.pure})} holds for all $\kappa<\lambda_\C.$
\end{lemma}

\begin{proof}
``Only if'' is clear; for the converse,
assume~(\ref{x.pure}) holds whenever $\kappa<\lambda_\C.$
Suppose we are given $\kappa$ not necessarily $<\lambda_\C,$ and
distinct elements $p,q\in|\coprod_\kappa A\<|.$
Since $\coprod_\kappa A$ is generated by the images of $A$ under
the $\kappa$ coprojection maps, we see from Lemma~\ref{L.gen_by} that
$p$ and $q$ will lie in the subalgebra generated by the
copies of $A$ indexed by some subset $I\subseteq\kappa$
with $0<|I|<\lambda_\C.$
Now a set-theoretic retraction of $\kappa$ onto $I$
(a left inverse to the inclusion of $I$ in $\kappa)$
induces algebra retractions $\coprod_\kappa A\to\coprod_I A$ and
$\coprod_\kappa B\to\coprod_I B,$ making a commuting square
with the maps $\coprod_\kappa A\to\coprod_\kappa B$
and $\coprod_I A\to\coprod_I B.$
By choice of $I,$ the elements $p$ and $q$ lie in
the subalgebra $\coprod_I A\subseteq\coprod_\kappa A,$
and since $|I|<\lambda_\C,$ the map
$\coprod_I A\to\coprod_I B$ is one-to-one; so $p$ and $q$ have
distinct images in $\coprod_I B,$ and hence in $\coprod_\kappa B.$
\end{proof}

We are now ready for

\begin{definition}\label{D.subcoalg}
If $R$ and $R'$ are $\!\D\!$-coalgebras in $\C,$ then
we will call $R'$ a {\em subcoalgebra} of $R$ if $|R'|$ is
a copower-pure $\!\C\!$-subalgebra of $|R\<|,$ and for
each $\alpha\in\Omega_\D,$ the co-operation $\alpha^{R'}$
is the restriction to $|R'|\subseteq|R\<|$ of $\alpha^R.$
\end{definition}

Thus, for a subalgebra $A\subseteq |R\<|$ to yield a subcoalgebra
of $R,$ it must be copower-pure, and have the property that
each $\alpha^R$ carries $A$ into $\coprod_{\ari(\alpha)}A.$

(Remark: It would probably be more natural in the above definition
to require~(\ref{x.pure}) to hold only for $\kappa<\lambda_\D;$
and perhaps to remove that condition entirely when the arities of the
operations of $\D$ are all $\leq 1.$
But for simplicity, we will stick with the above definition.)

We shall now prove the existence of subcoalgebras satisfying
cardinality bounds, as sketched earlier.
(Note that the ``$\!\mu\!$'' of the next result is not necessarily the
same as the value so named in the proof of Lemma~\ref{L.gen_card},
since it also involves $\cd(\Omega_\D).)$

\begin{theorem}\label{T.subcoalg}
Let $R$ be a $\!\D\!$-coalgebra object of $\C,$ and $X$
a subset of $||R\<||.$
Then $R$ has a subcoalgebra $R'$ whose underlying set contains $X,$
and has cardinality at most
\begin{equation}\begin{minipage}[c]{35pc}\label{x.mu}
$\mu\ =\ %
\max(\cd(X),\,\cd(\Omega_\C),\,\cd(\Omega_\D),\,2)^{\lambda_\C-}.$
\end{minipage}\end{equation}
\end{theorem}

\begin{proof}
We shall construct a chain of subalgebras
$A_0\subseteq A_1\subseteq\dots\subseteq A_{\lambda_\C}$ of $|R\<|,$
and show that $A_{\lambda_\C}$ is the underlying $\!\C\!$-algebra of
a subcoalgebra $R'$ with the asserted properties.

We take for $A_0$ the subalgebra of $|R\<|$ generated by $X;$
by Lemma~\ref{L.gen_card} this has cardinality $\leq\mu.$

Assuming for some $\iota<\lambda_\C$ that $A_\iota$ has been
constructed, and has cardinality at most $\mu,$ we obtain
$A_{\iota+1}$ by adjoining $\leq\mu$ further elements chosen as follows.

First, for every $\kappa<\lambda_\C,$ and every pair of elements
$p,q\in|\coprod_\kappa A_\iota|$ which have equal image under the
natural map $\coprod_\kappa A_\iota\to\coprod_\kappa |R\<|,$
I claim we can adjoin to $A_\iota$ a set of
$<\lambda_\C$ elements whose presence causes the images of these
elements in the copower of the resulting algebra to fall together.
Indeed, this follows from Corollary~\ref{C.rels}, and the observation
that given a presentation $|R\<|=\langle X\mid Y\rangle_\C,$ the
copower $\coprod_\kappa |R\<|$ can be presented by taking the union of
$\kappa$ copies of $X,$ and for each of these, a copy of $Y.$
Note also that $\coprod_\kappa A_\iota$ is generated by
$\kappa$ copies of $A_\iota,$ which has cardinality $\leq\mu;$
hence by Lemma~\ref{L.gen_card} it itself has cardinality $\leq\mu,$
hence there are at most $\mu$ such pairs $p,q$ to deal with; so
for each $\kappa<\lambda_\C,$ we are adjoining at most $\mu$ elements.
The number of such cardinals $\kappa$ is $\leq\lambda_\C;$
so in handling all such pairs $p,q,$ for all $\kappa<\lambda_\C,$
we adjoin $\leq\mu$ new elements of $|R\<|$ to $A_\iota.$

In addition, for each operation symbol $\alpha\in\Omega_\D,$ and
each $p\in|A_\iota|,$ consider the image of $p$ under
the co-operation $\alpha^R: |R\<|\to\coprod_{\ari(\alpha)}|R\<|.$
By Lemma~\ref{L.gen_by}, this will lie in the subalgebra
of $\coprod_{\ari(\alpha)}|R\<|$ generated by a subset $X_{p,\alpha},$
having cardinality $<\lambda_\C,$ of the generating set for
that copower given by the union of the images of the coprojections.
So $X_{p,\alpha}$ is contained in the
union of the images, under those coprojections, of a
set $X'_{p,\alpha}$ of $<\lambda_\C$ elements of $|R\<|.$
For each $p\in|A_\iota|$ and $\alpha\in\Omega_\D,$
let us include such a set $X'_{p,\alpha}$
in the set of elements we are adjoining to $A_\iota.$
Letting $p$ run over the $\leq\mu$ elements of $A_\iota,$
and $\alpha$ over the elements of $\Omega_\D,$
we see that this process adjoins
$\leq\mu\cdot\cd(\Omega_\D)\cdot\lambda_\C\leq\mu$ new elements.
Let $A_{\iota+1}$ be the subalgebra of $|R\<|$ obtained by
adjoining to $A_\iota$ the two families of elements described in
this and the preceding paragraph.

On the other hand, if $\iota$ is a limit ordinal, we let $A_\iota$
be the $\!\C\!$-subalgebra of $|R\<|$ generated by
$\bigcup_{\eta\in\iota} |A_\eta|.$
That union, being a union of $\leq\lambda_\C$ sets of
cardinality $\leq\mu,$ will itself have
cardinality $\leq\mu,$ and it follows from Lemma~\ref{L.gen_card}
that $A_\iota$ will as well.

Now consider the subalgebra $A = A_{\lambda_\C}$ of $|R\<|.$
Since $\lambda_\C$ is by assumption a regular
cardinal, the union $\bigcup_{\eta\in\lambda_\C} |A_\eta|$
involved in the construction of $A$ is over a chain of cofinality
$\lambda_\C,$ which strictly majorizes the arities of
all operations of $\C;$ hence that union is closed
under those operations, so $|A\<|$ is that union.

I claim that $A$ is copower-pure in $|R\<|.$
Indeed, given any $\kappa<\lambda_\C,$ and elements $p,\ q$ of
$\coprod_\kappa A$ that fall together $\coprod_\kappa |R\<|,$
we can find $<\lambda_\C$ elements of $A$ such that $p$ and $q$ lie
in the subalgebra of $\coprod_\kappa A$ generated by images of
those elements under coprojection maps; and we can then find
some $\iota$ such that all those elements lie in $A_\iota.$
Thus we get $p',q'\in|\coprod_\kappa A_\iota|$ which map to
$p,q\in|\coprod_\kappa A\<|;$ so their images under the composite map
$\coprod_\kappa A_\iota\to \coprod_\kappa A \to\coprod_\kappa |R\<|$
fall together.
Hence, by the construction of $A_{\iota+1},$ the images of $p'$
and $q'$ fall together in $\coprod_\kappa A_{\iota+1},$ hence
they do so in $\coprod_\kappa A,$ i.e., $p=q,$ as required.

It remains to show that
each co-operation $\alpha^R$ $(\alpha\in\Omega_\D)$ carries
$A\subseteq|R\<|$ into the subalgebra
$\coprod_{\ari(\alpha)}A$ of $\coprod_{\ari(\alpha)}|R\<|.$
Every element $p\in|A\<|$ lies in some $A_\iota,$
and by construction, $A_{\iota+1}$ contains elements which
guarantee that $\alpha^R(p)$ lies in the subalgebra of
$\coprod_{\ari(\alpha)}|R\<|$ generated by the image of
$\coprod_{\ari(\alpha)}A_{\iota+1},$ hence, a fortiori,
in $\coprod_{\ari(\alpha)}A\<.$
\end{proof}

\begin{corollary}\label{C.subcoalg}
If $R$ is a $\!\D\!$-coalgebra object of $\C,$ then for every element
$p$ of $|R\<|$ there is a subcoalgebra $R'$ of $R$ whose underlying
$\!\C\!$-algebra contains $p,$ and has cardinality at most
$\max(\cd(\Omega_\C),\,\cd(\Omega_\D),\,2)^{\lambda_\C-}.$

In particular, if $\C$ and $\D$ each have at most countably
many operations, and all operations of $\C$ are finitary,
then every element of a $\!\D\!$-coalgebra object of $\C$ is
contained in a countable or finite subcoalgebra.\endproof
\end{corollary}

\section{Pseudocoalgebras, and the solution set
condition.}\label{S.pseudo}

We now come to the pseudocoalgebras of our sketched development.

\begin{definition}\label{D.pseudo}
By a {\em $\!\D\!$-pseudocoalgebra} in a category $\bf{A},$ we
shall mean a $\!4\!$-tuple
\begin{equation}\begin{minipage}[c]{35pc}\label{x.pseudo}
$S\ =\ (S\ba\<,\ (S_\alpha)_{\alpha\in\Omega_\D},\ %
(c^S_{\alpha,\<\iota})_{\alpha\in\Omega_\D,\,\<\iota\in\ari(\alpha)},\ %
(\alpha^S)_{\alpha\in\Omega_\D}),$
\end{minipage}\end{equation}
where $S\ba$ and the $S_\alpha$ are objects of $\fb{A}$
\textup{(}the ``base object'' and
the ``pseudocopower objects''\textup{)},
and for each $\alpha\in\Omega_\D,$ $\alpha^S$ \textup{(}the
$\!\alpha\!$-th ``pseudo-co-operation''\textup{)} and
the $c^S_{\alpha,\<\iota}$ \textup{(}the ``pseudocoprojections'',
one for each $\iota\in\ari(\alpha))$ are morphisms $S\ba\to S_\alpha.$

A morphism of $\!\D\!$-pseudocoalgebras $f:S\to S'$
will mean a family of morphisms
%
\begin{equation}\begin{minipage}[c]{34pc}\label{x.psmorph}
$f\ba: S\ba\to S'\ba\<,$\quad and\quad
$f_\alpha: S_\alpha\to S'_\alpha$\quad
$(\alpha\in\Omega_\D)$
\end{minipage}\end{equation}
which make commuting squares with the
$c^S_{\alpha,\<\iota}$ and $c^{S'}_{\alpha,\<\iota},$
and with the $\alpha^S$ and $\alpha^{S'}.$
The category of $\!\D\!$-pseudocoalgebras in $\fb{A}$ will be
denoted $\pscoalg{\fb{A}}{\D}.$
\end{definition}

Note that the operation-set $\Omega_\D$ and arity-function $\ari_\D$
of $\D$ come into the definition of $\!\D\!$-pseudo\-coalgebra,
but the identities of $\D$ do not.
In the use we will make of pseudocoalgebras, the fact that those
identities are, by definition, cosatisfied by the
{\em $\!\D\!$-coalgebras} we map to them will be all that matters.
Let us make clear in what sense one can map coalgebras to
pseudocoalgebras.

\begin{definition}\label{D.psi}
If $R$ is a $\!\D\!$-coalgebra, or more generally,
an $\!\Omega_\D\!$-coalgebra, in $\fb{A},$ then we define
the associated $\!\D\!$-pseudocoalgebra $\psi(R)$ to have
\begin{quote}
$\psi(R)\ba\ =\ |R\<|,$\\[.17em]
$\psi(R)_\alpha\ =\ \coprod_{\ari(\alpha)}|R\<|,$\\[.17em]
$c^{\psi(R)}_{\alpha,\<\iota}=$ the $\!\iota\!$-th coprojection:
$|R\<|\to\coprod_{\ari(\alpha)}|R\<|,$\quad and\\[.17em]
$\alpha^{\psi(R)}\ =\ \alpha^R:\ |R\<|\to\coprod_{\ari(\alpha)}|R\<|\<.$
\end{quote}

Clearly, $\psi$ yields a full and faithful functor
$\coalg{\fb{A}}{\<\Omega_\D\mbox{-}\fb{Alg}}\to\pscoalg{\fb{A}}{\D},$
so when there is no danger of ambiguity, we shall
treat $\coalg{\fb{A}}{\<\Omega_\D\mbox{-}\fb{Alg}}$ and
its subcategory $\coalg{\fb{A}}{\D}$ as full subcategories of
$\pscoalg{\fb{A}}{\D};$ in particular
we shall speak of morphisms from coalgebras to pseudocoalgebras.
If $S$ is a $\!\D\!$-pseudocoalgebra,
then a $\!\D\!$-coalgebra given with a morphism
to $S,$ i.e., an object of the comma category
$(\coalg{\fb{A}}{\D}\downarrow S),$
will be called a $\!\D\!$-coalgebra {\em over}~$S.$

We shall say that a pseudocoalgebra $S$ ``is an
$\!\Omega_\D\!$-coalgebra'' if it is isomorphic to
$\psi(R)$ for some $\!\Omega_\D\!$-coalgebra $R,$ in
other words, if for every $\alpha,$ the object $S_\alpha$
is the copower $\coprod_{\ari(\alpha)}S\ba,$ with the
pseudocoprojections $c^S_{\alpha,\<\iota}$ as the coprojections.
We will say that $S$ is a $\!\D\!$-coalgebra if it is
an $\!\Omega_\D\!$-coalgebra $R$ which cosatisfies
the identities of~$\D.$
\end{definition}

Note that if $R$ is a coalgebra and $S$ a pseudocoalgebra, then
every morphism $f:R\to S$ is determined by
$f\ba:|R\<|\to S\ba,$ since once this is given, the components
$f_\alpha$ are uniquely determined by the property of
making commuting squares with the coprojections and pseudocoprojections,
via the universal property of $\coprod_{\ari(\alpha)}|R\<|.$
Of course, in general not every map $f\ba:|R\<|\to S\ba$
induces a morphism $f:R\to S,$ since the maps $f_\alpha$ so
determined by $f\ba$ may fail to satisfy
the remaining condition, that the squares they make with the
co-operations and pseudo-co-operations commute.

We now again restrict attention to the case where the
variety $\C$ plays the role of $\fb{A}.$
To the convention that ``coalgebra'', unmodified, means
``$\!\D\!$-coalgebra in $\!\C\!$''
we add the convention that ``pseudocoalgebra'', unmodified, means
``$\!\D\!$-pseudocoalgebra in $\!\C\!$''.
Note that these pseudocoalgebras are simply a kind of many-sorted
algebra, so there is no difficulty constructing limits of such objects.

We recall from Lemma~\ref{L.colim}
that $\coalg{\C}{\D}$ has small colimits,
given on underlying $\!\C\!$-algebras by the colimits of
the corresponding algebras in $\C.$
It follows that if we have a diagram of coalgebras over a fixed
pseudocoalgebra $S,$ its colimit coalgebra has an induced
morphism into $S,$ and so will also be the colimit of the given diagram
in the comma category $(\coalg{\C}{\D}\downarrow S).$

\begin{definition}\label{D.quasifinal}
A morphism of coalgebras will be called {\em surjective} if
it is surjective on underlying $\!\C\!$-objects.

If $f:R\to R'$ is a surjective morphism in
$(\coalg{\C}{\D}\downarrow S),$ for some pseudocoalgebra $S,$
then $R'$ \textup{(}given with the map $f$ from $R)$
will be called an {\em image} coalgebra of $R$ over $S.$
To avoid dealing with the non-small set of isomorphic copies of each
such image, we shall call an image coalgebra $R'$ of $R$ {\em standard}
if the map $||R\<||\to||R'||$ is the canonical map from a set to its
set of equivalence classes under an equivalence relation.

A coalgebra $R$ over $S$ will be called {\em strongly
quasifinal} over $S$ if the only surjective morphisms out of $R$ in
$(\coalg{\C}{\D}\downarrow S)$ are the isomorphisms.
\end{definition}

Given a coalgebra $R$ over a pseudocoalgebra $S,$
the category of all standard images of $R$ over $S$ will form a
partially ordered set, isomorphic to a sub-poset
of the lattice of congruences on the $\!\C\!$-algebra $|R\<|.$
We cannot expect that the set of congruences on $|R\<|$ such that
the $\!\D\!$-coalgebra structure of $R$ extends to the resulting
factor-algebra will be closed under intersections; but it will be
closed under arbitrary joins, since, as just noted, colimits of
coalgebras over $S$ correspond to
colimits of underlying $\!\C\!$-objects.
We deduce

\begin{lemma}\label{L.quasifinal}
Let $S$ be a pseudocoalgebra and $R$ a coalgebra over $S.$
Then the standard images of $R$ over $S$ form a
\textup{(}small\textup{)} complete lattice.
The greatest element of this lattice is, up to isomorphism, the
unique strongly quasifinal homomorphic image of $R$ over~$S.$\endproof
\end{lemma}

We can now show that, up to isomorphism, the strongly
quasifinal coalgebras over $S$ form a small set.

\begin{lemma}\label{L.no_two_iso}
Let $S$ be a pseudocoalgebra, and $R$ a strongly quasifinal
coalgebra over $S.$
Then distinct subcoalgebras of $R$ are nonisomorphic as
coalgebras over~$S.$

Hence in view of Corollary~\ref{C.subcoalg},
$\cd(||R\<||)$ is $\leq$ the sum of the cardinalities of
all \textup{(}up to isomorphism over $S)$ coalgebras over $S$
of cardinality at most
\begin{equation}\begin{minipage}[c]{34pc}\label{x.mu=}
$\max(\cd(\Omega_\C),\,\cd(\Omega_\D),\,2)^{\lambda_\C-}.$
\end{minipage}\end{equation}
This sum is at most
\begin{equation}\begin{minipage}[c]{34pc}\label{x.card}
$\max(\cd(|S\ba|),\lambda_\D)%
^{\max(\cd(\Omega_\C),\,\cd(\Omega_\D),\,2)^{\lambda_\C-}}.$
\end{minipage}\end{equation}

Hence, up to isomorphism, there is only a small set of coalgebras
$R$ strongly quasifinal over $S.$
\end{lemma}

\begin{proof}
If we had two distinct embeddings into $R$ over $S$ of some coalgebra
$R'$ over $S,$ then the coequalizer of the resulting diagram
$R'\stackrel{\longrightarrow}{\scriptstyle{\longrightarrow}}R$ would
be a proper image of $R$ over $S,$ contradicting the
strong quasifinality of $R.$
This gives the assertion of the first paragraph.
Hence if we break the subcoalgebras of $R$ of cardinality at
most~(\ref{x.mu=}) into their isomorphism classes over $S,$ no more
than one copy of each can occur, and by Corollary~\ref{C.subcoalg},
such subcoalgebras have union $R,$ giving the second assertion,
from which the final sentence of the lemma clearly follows.

To get the explicit bound~(\ref{x.card}), let us write $\mu$ for
the cardinal~(\ref{x.mu=}) and $\nu$ for~(\ref{x.card}).
We shall show that the number of structures of coalgebra over
$S$ on a set $X$ of cardinality $\leq\mu$ is at most $\nu.$
Since there are $\leq\mu<\nu$ cardinalities $\leq\mu,$
this will give $\leq\nu$ isomorphism classes of strongly quasifinal
coalgebras over $S$ altogether, and since each such coalgebra
has cardinality $\leq\mu<\nu,$ the sum of their cardinalities
will be $\leq\nu,$ as required.

To bound the number of structures of coalgebra over $S$ on $X,$
note that such a structure is determined by
several maps (subject to restrictions that we will not repeat
because they do not come into our calculations):
\begin{equation}\begin{minipage}[c]{34pc}\label{x.toS}
a map $X\to|S\ba|,$
\end{minipage}\end{equation}
\begin{equation}\begin{minipage}[c]{34pc}\label{x.fromprod}
for each $\alpha\in\Omega_\C,$ a map $X^{\ari(\alpha)}\to X,$
\end{minipage}\end{equation}
making $X$ a $\!\C\!$-algebra $A;$ and once this has been done,
\begin{equation}\begin{minipage}[c]{34pc}\label{x.tocoprod}
for each $\alpha\in\Omega_\D,$ a map $A\to\coprod_{\ari(\alpha)}A,$
\end{minipage}\end{equation}
giving the $\!\D\!$-coalgebra structure.

Given $X$ of cardinality $\leq\mu,$ the number of
possible 
choices for~(\ref{x.toS}) is bounded by $\cd(|S\ba|)^\mu.$
For each $\alpha\in\Omega_\C,$ the number of choices for the map
in~(\ref{x.fromprod}) is $\leq \mu^{\mu^{\ari(\alpha)}},$ but by
definition (see~(\ref{x.mu=})), $\mu$ is not increased by
exponentiation by
$\ari(\alpha)<\lambda_\C,$ so this bound is $\leq\mu^\mu.$
Letting $\alpha$ run over $\Omega_\C,$ we conclude that the
number of choices for~(\ref{x.fromprod}) is
$\leq(\mu^\mu)^{\cd(\Omega_\C)}=\mu^{\mu\,\cd(\Omega_\C)}.$
But again by definition, $\mu\geq\cd(\Omega_\C),$ so the product
$\mu\,\cd(\Omega_\C)$ simplifies to $\mu;$ hence the
number of choices for~(\ref{x.fromprod}) is $\leq\mu^\mu.$

Finally, for each $\alpha\in\Omega_\D,$ the copower
in~(\ref{x.tocoprod}) will be generated by an $\!\ari(\alpha)\!$-tuple
of copies of $X,$ hence by a set of cardinality
$\leq\ari(\alpha)\,\mu\leq\lambda_\D\,\mu,$
so by Lemma~\ref{L.gen_card} that copower has cardinality
$\leq\max(\lambda_\D\,\mu,\,\cd(\Omega_\C),\,2)^{\lambda_\C-}=
(\lambda_\D\,\mu)^{\lambda_\C-}$ (since
$\cd(\Omega_\C)$ and $2$ are majorized by $\mu).$
Hence for each $\alpha\in\Omega_\D$
the number of maps as in~(\ref{x.tocoprod}) is
$\leq((\lambda_\D\,\mu)^{\lambda_\C-})^\mu
\leq((\lambda_\D\,\mu)^\mu)^\mu=(\lambda_\D\,\mu)^\mu.$
Letting $\alpha$ run over $\Omega_\D,$ we get an additional
factor of $\cd(\Omega_\D)$ in the exponent, but this again is
absorbed by $\mu.$
Bringing together the choices made in~(\ref{x.toS}),~(\ref{x.fromprod})
and~(\ref{x.tocoprod}), we get the bound
\begin{equation}\begin{minipage}[c]{34pc}\label{x.total}
$\cd(|S\ba|)^\mu\ \mu^\mu\ (\lambda_\D^\mu\ \mu^\mu)$
\end{minipage}\end{equation}
on the number of possible structures.
Note also that for any infinite cardinal $\lambda$ and any
cardinal $\kappa>1,$ one knows that $\kappa^\lambda>\lambda,$
hence $\lambda^\lambda\leq(\kappa^\lambda)^\lambda=
\kappa^{\lambda\lambda}=\kappa^\lambda.$
Applying this with $\mu$ in the role
of $\lambda$ and $\lambda_\D$ in the role of $\kappa,$
we see that the $\mu^\mu$ terms can be dropped from~(\ref{x.total}).
Rewriting the product as a maximum, and putting in the
definition~(\ref{x.mu=}) of $\mu,$ we get the desired
bound~(\ref{x.card}).
\end{proof}

We deduce

\begin{theorem}\label{T.final}
Let $S$ be a pseudocoalgebra.
Then the category of coalgebras over $S$ has a final object.
The cardinality of the underlying set of this object is
at most~\textup{(\ref{x.card})}.
\end{theorem}

\begin{proof}
Since the category has small colimits, and every object has a
morphism into a strongly quasifinal object, and up to
isomorphism there is only a small set of such objects, the dual
statement to the Initial Object Theorem gives the required final object.
Since a final object is in particular strongly quasifinal, the
cardinality of its underlying set is bounded by~(\ref{x.card}).
\end{proof}

We can now show $\coalg{\C}{\D}$ complete.
Given a small category $\fb{E}$ and
a functor $F:\fb{E}\to\coalg{\C}{\D},$
each coalgebra $F(E)$ $(E\in\mathrm{Ob}(\fb{E}))$ can be
regarded as a pseudocoalgebra, $\psi(F(E)),$ and
this system of pseudocoalgebras has a limit, which can be
constructed objectwise.
Let
\begin{quote}
$S\ =\ \limit_\fb{E}\ \psi\,F.$
\end{quote}

If $R$ is a coalgebra, then a cone in $\coalg{\C}{\D}$ from
$R$ to the diagram of coalgebras $F$ is equivalent to
a morphism of pseudocoalgebras $R\to S.$
Hence the final object of $(\coalg{\C}{\D}\downarrow S)$
corresponds to a limit of $F.$
This gives

\begin{theorem}\label{T.lim}
$\coalg{\C}{\D}$ has small limits; equivalently,
$\Rep{\C}{\D}$ has small colimits.
Moreover, given a small category $\fb{E}$ and
a functor $F:\fb{E}\to\coalg{\C}{\D},$ we have
\begin{equation}\begin{minipage}[c]{34pc}\label{x.limcard}
$\cd(||\limit F||)\ \leq\ %
\max\<(\cd(\limit_\fb{E}||F(E)||),\,\lambda_\D)%
^{\max(\cd(\Omega_\C),\,\cd(\Omega_\D),\,2)^{\lambda_\C-}}.$
\end{minipage}
\end{equation}
\endproof
\end{theorem}
\vspace{.8em}

Since most of algebra is done with finitary operations, and
often with only finitely many of them, let us record what our result
says in that case.

\begin{corollary}\label{C.c}
If $\C$ and $\D$ each have only finitely many operations, and all
of these are finitary, then the final $\!\D\!$-coalgebra object
of $\C$ has underlying set of at most continuum cardinality.

In fact, this remains true if ``finitely many operations''
is generalized to ``at most countably many operations'', the assumption
of finite arity on the operations of $\D$ \textup{(}but not of
$\C)$ is generalized to that of arity less than the continuum,
and ``the final $\!\D\!$-coalgebra of $\C\!$''
is generalized to ``the limit of any
diagram of at most countably many $\!\D\!$-coalgebra objects of $\C,$
each of which has underlying set of at most continuum cardinality.''
\end{corollary}\noindent{\em Proof.}
In the situation of the generalized statement, the right-hand
side of~(\ref{x.limcard}) is bounded by
\begin{quote}
$\max((2^{\aleph_0})^{\aleph_0},\,
2^{\aleph_0})^{\max(\aleph_0,\,\aleph_0,\,2)^{\aleph_0-}}
=\ (2^{\aleph_0})^{\aleph_0^{\aleph_0-}}
=\ (2^{\aleph_0})^{\aleph_0}\ =\ 2^{\aleph_0}.$
\\[-3.2em]
\end{quote}
\endproof
\vspace{.8em}

We can apply Theorem~\ref{T.final} to other pseudocoalgebras $S$ than
those arising from limit diagrams.
One such application gives the next result;
we omit the cardinality estimate, which the reader can easily supply.

\begin{theorem}\label{T.cofree}
If $A$ is an object of $\C,$ then there is a universal
$\!\D\!$-coalgebra object $R$ of $\C$ with a $\!\C\!$-algebra
homomorphism $|R\<|\to A\<.$
In other words, the forgetful functor
$\coalg{\C}{\D}\to\C$ has a right adjoint; equivalently, the functor
$\Rep{\C}{\D}\to\Rep{\C}{\fb{Set}}$ given by composition with the
underlying set functor on $\D$ has a left adjoint.

More generally, suppose $\D'$ is a variety whose type $\Omega_{\D'}$
is a ``subtype'' of $\Omega_\D;$ i.e., such that the operation-symbols
of $\D'$ form a subset of the operation-symbols of $\D,$
and its arity function is the restriction of that of $\D.$
\textup{(}We make no assumption on the identities of $\D'.)$
Then for any $\!\D'\!$-coalgebra $R'$ of $\C$ there exists a universal
$\!\D\!$-coalgebra object $R$ of $\C$ with a $\!\C\!$-algebra
homomorphism $|R\<|\to|R'|$ that respects $\!\D'\!$-co-operations.
\end{theorem}

\begin{proof}
Let us establish the first assertion, then note how to adapt
the argument to the second.
Given $A,$ we define a pseudocoalgebra $S$ by letting
$S\ba=A,$ and, for each $\alpha\in\Omega_\D,$
taking $S_\alpha$ to be the final object of $\C,$ with the unique maps
from $S\ba$ to that final object serving (necessarily)
both as pseudocoprojections
$c^S_{\alpha,\<\iota}$ and as pseudo-co-operations $\alpha^S.$
Then a morphism from a $\!\D\!$-coalgebra object $R$ of $\C$
to $S$ is equivalent to a $\!\C\!$-algebra homomorphism $f:|R\<|\to A.$
Thus, the final coalgebra $R$ over $S$ of Theorem~\ref{T.final}
is a coalgebra with a universal map $|R\<|\to A\<.$

Given $\D'$ and $R'$ as in the second statement,
we form the $\!\D'\!$-pseudocoalgebra $\psi(R'),$
and extend this to a $\!\D\!$-pseudocoalgebra $S$ by letting
$S_\alpha$ be the final object of $\C$
for each $\alpha\in\Omega_\D-\Omega_{\D'}.$
The pseudocoprojections and pseudo-co-operations correspondinging
to these $\alpha$ are uniquely determined, and the $R$ given
by Theorem~\ref{T.final} again has the desired universal property.
\end{proof}

The construction of the first paragraph of the above
theorem can be thought of as giving a ``cofree coalgebra'' on
each object $A$ of $\C;$ equivalently a ``free'' representable
$\!\D\!$-valued functor $G$ on each representable set-valued
functor $E$ on $\C.$
The latter is not, of course, the composite of $E$ with the
free $\!\D\!$-algebra functor $F:\fb{Set}\to\D,$ which is
in general not representable.
For instance, for $\C=\Gp,$ $\D=\fb{Ab},$ and $E$ the
underlying set functor $U_\mathbf{Group},$
the representable functor $G:\Gp\to\fb{Ab}$
``free'' on $E$ is the trivial functor, since
there are no nontrivial representable functors $\Gp\to\fb{Ab},$
as follows from the description of all representable
functors $\Gp\to\Gp$ that will be recalled in \S\ref{S.Group} below.
In contrast, we will see in that section that for $\C=\D=\Gp$
or $\C=\D=\fb{Monoid},$ the free representable
$\!\D\!$-valued functor on any nontrivial representable set-valued
functor is nontrivial, and easy to describe.

A generalization of the first paragraph of Theorem~\ref{T.cofree}
that naturally suggests itself is that for every representable
functor $F:\D\to\fb{E}$ between varieties, the induced functor
$F\circsm-:\Rep{\C}{\D}\to\Rep{\C}{\fb{E}}$ should have a left adjoint.
This seems likely to be true: In the case where $F$ is
underlying-set-preserving, one can prove it from the present second
assertion of the theorem, by replacing $\D$ with an equivalent
variety, having the original operations of $\D,$ plus some additional
ones corresponding to those derived operations of $\D$ that yield the
operations of $\fb{E},$ so that $F$ becomes a forgetful functor.

However, this generalization of the first
assertion of the theorem, if true, would still not subsume the second
assertion, since the latter does not assume that the forgetful functor
$\D\to\Omega_{\D'}\!$-$\!\fb{Alg}$ has values in $\D'.$
The question of the ``right'' generalization of
Theorem~\ref{T.cofree} obviously merits investigation.
We shall, however, devote the rest of this note to exploring the
constructions described by our theorems as stated.

\section{Some examples with $\C=\fb{Set}.$}\label{S.Set}

The proofs in the preceding section were quite uninformative
as to the structures of the limit coalgebras obtained.
To investigate these, let us use the heuristic of
\cite[\S3.2]{245}:  To construct an object with a right universal
property, consider an element $x$ of a (not necessarily universal)
object $A$ of the desired sort, ask what data one can describe
that such an element will determine, and what restrictions one can find
that these data must satisfy.
Then see whether the set of all possible values for data satisfying
those restrictions can be given a structure of object of the
indicated sort, induced by that of the assumed object $A.$
If so, the result should be the desired universal object.

We will here apply this approach to the construction of the final object
of $\coalg{\fb{Set}}{\Bi},$ where $\Bi$ is the category of sets with a
single binary operation $\beta,$ not subject to any identities.

Let $S$ be the final
$\!\Bi\!$-pseudocoalgebra in $\fb{Set},$ with $S\ba$ and
$S_\beta$ both $\!1\!$-element sets.
We want to find a right-universal coalgebra
over $S,$ so our heuristic says we should think about
a general $\!\Bi\!$-coalgebra $R$ in $\fb{Set}$ over $S,$
and an element $x\in|R\<|.$
(Since $\C$ is here $\fb{Set},$ what we would otherwise write
$||R\<||$ is simply $|R\<|.)$

Given $x,$ the first datum that we get from
it is its image in $S\ba\<;$ but this tells us nothing.

The structure of $\!\Bi\!$-coalgebra on $R$ also gives a map
$|R\<|\to|R\<|\cP|R\<|.$
Combining this with our ``useless'' map
$|R\<|\to S\ba,$ we get a map
\begin{equation}\begin{minipage}[c]{34pc}\label{x.01}
$|R\<|\ \to\ |R\<|\cP|R\<|\ \to\  S\ba\cP S\ba\<.$
\end{minipage}\end{equation}
Identifying the right-hand side with $\{0,1\},$ we thus see that
we can associate to every element of $|R\<|$ a member of the latter set.

Now that we have this map $|R\<|\to\{0,1\},$ we can
combine it in turn with the co-operation, getting a composite map
\begin{equation}\begin{minipage}[c]{34pc}\label{x.01^2}
$|R\<|\ \to\ |R\<|\cP|R\<|\ \to\ \{0,1\}\cP\{0,1\}\ \cong\ %
\{0,1\}\times\{0,1\}\ =\ \{0,1\}^2.$
\end{minipage}\end{equation}
Here, in the ``$\!\cong\!$'' step, we are using the
identification $X\cP X\cong\{0,1\}\times X$ for any set $X.$
Thus, the {\em first} coordinate of the image of an element
$x\in|R\<|$ under~(\ref{x.01^2}) specifies which copy of $|R\<|$
it is mapped into by $\beta:|R\<|\to|R\<|\cP|R\<|.$
This was already determined by the image of $x$ under~(\ref{x.01});
so the image of $x$ under~(\ref{x.01^2}) is gotten by bringing in a
second coordinate, specifying which ``side'' of the indicated copy
of $|R\<|$ the element $x$ maps to.

Repeating this process, we find at the next step that each
element of $|R\<|$ determines an element of $\{0,1\}^3$ extending its
image under~(\ref{x.01^2}); and so on.
Thus, after countably many steps, we get a map
\begin{equation}\begin{minipage}[c]{34pc}\label{x.01^w}
$|R\<|\ \to\ \{0,1\}^\omega.$
\end{minipage}\end{equation}

There is no evident way to get more information out of an
element of $|R\<|,$ since its image under
\begin{equation}\begin{minipage}[c]{34pc}\label{x.01^w+}
$|R\<|\ \to\ |R\<|\cP|R\<|\ \to\ %
\{0,1\}^\omega\cP\{0,1\}^\omega\ \cong\ \{0,1\}\times\{0,1\}^\omega$
\end{minipage}\end{equation}
is determined by its image in $\{0,1\}^\omega:$ an element
which maps to $(a_0,a_1,a_2,\dots)$ under~(\ref{x.01^w}) can be seen
to map to $(a_0,(a_1,a_2,\dots))$ under~(\ref{x.01^w+}).
This translates to say that~(\ref{x.01^w}) is a morphism
of coalgebras from $R$ to the coalgebra with underlying
$\!\fb{Set}\!$-object $\{0,1\}^\omega,$ and co-operation
\begin{equation}\begin{minipage}[c]{34pc}\label{x.01coop}
$(a_0,a_1,a_2,\dots)\ \mapsto\ (a_0,(a_1,a_2,\dots))\in
\{0,1\}\times\{0,1\}^\omega
\ \cong\ \{0,1\}^\omega\,\cP\,\{0,1\}^\omega.$
\end{minipage}\end{equation}
It is straightforward to verify that
the map~(\ref{x.01^w}) we have built up is the unique
morphism from $R$ to this coalgebra; so this coalgebra structure
on $\{0,1\}^\omega,$
the Cantor set, makes it the final object of $\coalg{\fb{Set}}{\Bi}.$
(This shows, by the way, that the cardinality bound of
Corollary~\ref{C.c} is best possible, even in the
case of the first paragraph thereof!
The proof of that result shows that $R$ is, however, a union of strongly
quasifinal subcoalgebras of cardinality $\leq 2^{\aleph_0-}=\aleph_0;$
the reader might find it interesting to identify these.)

The representable functor $\fb{Set}\to\Bi$
determined by the above final coalgebra carries every
set $A$ to the set of all $\!A\!$-valued functions (with no continuity
condition) on the Cantor set, furnished with
the binary operation that takes
two such functions $f$ and $g$ to the function $\beta(f,g)$ which can
be described -- using the standard geometric picture of the Cantor
set as a subset of the unit interval -- as gotten by drawing the graphs
of $f$ and $g$ on the real line, with a unit distance between them,
then ``compressing'' the abscissa by a factor of $1/3,$
so that the resulting function again has domain the Cantor set.

The heuristic cited above said that after finding
data that an element $x$ of a coalgebra would determine,
we should look for restrictions those data would have to satisfy.
Above, we found no such restrictions, because the
variety $\C$ had trivial algebra structure, the
variety $\Bi$ involved no identities, and
the pseudocoalgebra $S$ was trivial.

The above example would be enough to motivate the construction
of the next section; but having begun to look at representable
functors from $\fb{Set}$ to algebras with one binary operation,
let us see what happens when identities are imposed on that operation.

Because of the paucity of algebraic structure on $\fb{Set},$
many natural identities can be realized by representable functors
only in trivial ways.
For instance, for the commutative identity
\begin{quote}
$\beta(a,a')\ =\ \beta(a',a),$
\end{quote}
to be cosatisfied by $S$ means that if we map
any $x\in|S\<|$ by $\beta^S$ to $|S\<|\cP|S\<|,$ and then
apply to its image the map that interchanges the two copies of $|S\<|$
in that coproduct, the image is unchanged.
No element $x$ can have this property;
hence $|S\<|$ must be the empty set;
so the only representable functor from $\fb{Set}$ to algebras
with a commutative binary operation is the one taking
all sets to the $\!1\!$-element algebra.

An example of an identity of a less restrictive sort is
\begin{equation}\begin{minipage}[c]{34pc}\label{x.01*}
$\beta(a,\beta(a',a''))\ =\ \beta(a,\beta(a',a''')),$
\end{minipage}\end{equation}
saying that $\beta(-,\beta(-,-))$ is independent of its last argument.
Letting $\D$ be the subvariety of $\Bi$ determined by~(\ref{x.01*}),
we find, as before, that the data one can get from an element
$x$ of any object $S$ of $\coalg{\fb{Set}}{\D}$
is a string of $\!1\!$'s and $\!0\!$'s, and that the
co-operation on these is given by~(\ref{x.01coop}).
However, condition~(\ref{x.01*}) says that if the image of $x$
under $\beta^S$ lies in the second copy of $|S\<|,$
then on applying $\beta^S$ again to that copy, the new image cannot
also lie in the second copy of $|S\<|;$ for its image therein
would have to have equal images under
any pair of maps $a'',\ a'''$ from $|S\<|$ into any set.
This says that the image of $x\in|S\<|$ under~(\ref{x.01^2})
cannot be the string $11.$
Moreover,~(\ref{x.01*}) applies to occurrences of $\beta(-,\beta(-,-))$
within larger expressions; hence the image of $x$
under~(\ref{x.01^w}) cannot have two successive $\!1\!$'s anywhere.

The set $|R'|$ of sequences satisfying this condition forms
the underlying set of the final object $R'$ of $\coalg{\fb{Set}}{\D};$
it constitutes a closed subset of the Cantor set which is uncountable,
but has measure zero under the natural probability measure on that set.
(That this subset has measure zero follows from the fact that
the number of finite strings of length $n$ with no two
successive $\!1\!$'s is the Fibonacci number $f_{n+2},$ and
that $f_{n+2}/2^n\to 0$ as $n\to\infty.)$
Geometrically, one obtains $|R'|$ from the full Cantor set by deleting
the right-hand half of the right-hand half-Cantor-set, and doing the
same recursively to all the natural copies of the Cantor set within
itself.
(What I am  calling a ``half-Cantor-set'' is the part of the Cantor set
within one of the non-deleted ``thirds'' in its geometric construction.)
Given any set $A,$ the value of our representable functor
$\fb{Set}\to\D$ at $A$ consists of all $\!A\!$-valued functions on
$|R'|.$
To compose two such functions $f$ and $g,$ one again draws
their graphs with unit space between them,
but this time, one strikes off the right-hand side of the
graph of $g;$ only then will compressing the abscissa give a
function $\beta(f,g)$ on $|R'|.$
It is not hard to see directly that this operation $\beta$
satisfies~(\ref{x.01*}).

Let us look next at the subvariety $\D$ of $\Bi$ determined
by the identity superficially similar to~(\ref{x.01*}),
\begin{equation}\begin{minipage}[c]{34pc}\label{x.0*3}
$\beta(a,\beta(a',a'''))\ =\ \beta(a,\beta(a'',a''')),$
\end{minipage}\end{equation}
saying that $\beta(-,\beta(-,-))$ is independent of its middle argument.
For an object $S$ of $\coalg{\fb{Set}}{\D},$ this
translates to say that the image of any $x\in|S\<|$ under~(\ref{x.01^w})
can contain no $1$ followed immediately by a $0.$
Thus, the only strings that can occur are those consisting
of a (possibly empty, possibly infinite) string of $\!0\!$'s
followed, if it is finite, by an infinite string of $\!1\!$'s.
This set is still infinite, but is now countable.
Writing $x_i$ $(i\in\omega)$ for the string consisting of
$i$ $\!0\!$'s followed by an infinite string of $\!1\!$'s, and
$x_{\omega}$ for the infinite string of $\!0\!$'s, we conclude that
the final object of $\coalg{\fb{Set}}{\D}$ has underlying set
$\{x_0,\<x_1,\<\dots\<,\<x_{\omega}\},$
and that the induced representable
functor takes every set $A$ to the set of all sequences
$(a_0,a_1,\dots,a_\omega)\in A^{\omega+1},$ with operation
$\beta((a_0,a_1,\dots,a_\omega),(b_0,b_1,\dots,b_\omega))=
(b_0,a_0,a_1,\dots,a_\omega).$
One finds, incidentally, that this algebra satisfies the stronger
identity
\begin{equation}\begin{minipage}[c]{34pc}\label{x.bb=b}
$\beta(a,\beta(a',a''))\ =\ \beta(a,a'').$
\end{minipage}\end{equation}

An identity with still more restrictive consequences is associativity,
\begin{equation}\begin{minipage}[c]{34pc}\label{x.assoc}
$\beta(a,\beta(a',a''))\ =\ \beta(\beta(a,a'),a'').$
\end{minipage}\end{equation}
Here one easily verifies
that the only strings in $\{0,1\}^\omega$ that satisfy the
corresponding condition are $00\<...\<0\<...$ and $11\<...\<1\<...\,.$
Calling these $x_0$ and $x_1,$ we find that
the resulting representable functor takes
a set $A$ to $A\times A,$ with operation
$\beta((a_0,a_1),(b_0,b_1))=(a_0,b_1);$
in semigroup theorists' language, a ``rectangular band'' \cite{C+P}.

Of course, the coalgebras described above are merely the final
objects in their respective categories, not the only such objects.
For instance, for $\D$ the subvariety of $\Bi$ determined
by~(\ref{x.0*3}), a random example of a $\!\D\!$-coalgebra object $S$
of $\fb{Set}$ is the $\!2\!$-element set $|S\<|=\{x,y\}$ with the
co-operation
that, written as a map $|S\<|\to\{0,1\}\times|S\<|,$ has the form
\begin{equation}\begin{minipage}[c]{34pc}\label{x.x,y}
$x\ \mapsto\ (1,y),\qquad y\ \mapsto\ (1,x).$
\end{minipage}\end{equation}
This induces the functor that takes every set $A$ to $A\times A$
with the binary operation
\begin{equation}\begin{minipage}[c]{34pc}\label{x.x,y*b}
$\beta((a,b),\,(a',b'))\ =\ (b',a'),$
\end{minipage}\end{equation}
which is easily seen to satisfy~(\ref{x.0*3}) (but which, unlike the
initial representable functor with that property, does not
satisfy~(\ref{x.bb=b})).
The unique map of the above coalgebra to the final object in the
variety of $\!\D\!$-coalgebras carries both $x$ and $y$ to
``$\!11\<...\<1\<...\!$''.
Equivalently, the unique morphism from our initial representable
functor $\fb{Set}\to\D$ to the functor given by~(\ref{x.x,y*b})
carries every sequence $(a_0,a_1,\dots;a_\omega)$ to $(a_0,a_0).$

\section{Precoalgebras.}\label{S.precoalg}

When we determined the final object of $\coalg{\fb{Set}}{\Bi}$
in the preceding section, the structure of its underlying
algebra (i.e., set), the Cantor set, arose as the inverse limit of
the system of finite sets
\begin{equation}\begin{minipage}[c]{34pc}\label{x.->->}
$\dots\to\{0,1\}^n\to\dots\to\{0,1\}^2\to\{0,1\}^1\to\{0,1\}^0\<.$
\end{minipage}\end{equation}

For none of the above sets were we given a map
$\{0,1\}^n\to\{0,1\}^n\cP\{0,1\}^n.$
Rather, for each of them we had a map
$\{0,1\}^n\to\{0,1\}^{n-1}\cP\{0,1\}^{n-1}.$
We shall abstract this sort of structure in the next definition.

I have so far avoided choosing a notation
for the coprojection maps from an object to a copower
of itself, though I have written the $\!\iota\!$-th
{\em pseudo\/}coprojection map of a pseudocoalgebra $S$ as
$c^S_{\alpha,\<\iota}:S\ba\to S_\alpha.$
Below, we will need to write down formulas relating
pseudocoprojections and genuine coprojections;
so given a cardinal $\kappa,$ an object $A$ of a category $\fb{A}$
having a $\!\kappa\!$-fold copower, and an $\iota\in\kappa,$ let us
write the $\!\iota\!$-th coprojection map of $A$ into that copower as
\begin{equation}\begin{minipage}[c]{34pc}\label{x.coproj}
$q^A_{\kappa,\<\iota}:\ A\ \to\ \coprod_\kappa A\<.$
\end{minipage}\end{equation}

In the next definition, I use Latin rather than Greek letters
for elements of the ordinal $\theta$ because
in our main application, $\theta$ will be $\omega\<{+}1,$
and most of what we do will concern its finite elements.
Also, on encountering expressions such as
$\coprod_{\ari(\alpha)}\,(P_j)\ba,$ the reader should
recall that, as indicated in \S\ref{S.intro}, this denotes the
$\!\ari(\alpha)\!$-fold copower of the single object $(P_j)\ba,$
not the coproduct of a family of different objects.

\begin{definition}\label{D.pre}
Let $\theta$ be an ordinal and $\fb{A}$ a
category with small coproducts.
By a $\!\theta\!$-indexed {\em precoalgebra} $P$ of type $\Omega_\D$ in
$\fb{A}$ we will mean an inverse system of
$\!\D\!$-pseudo\-coalgebras
$(P_k)_{k\in\theta}$ with connecting maps
\begin{equation}\begin{minipage}[c]{34pc}\label{x.pjk}
$p_{j,\<k}:\ P_k\ \to\ P_j$ $(j\leq k\in\theta)$
\end{minipage}\end{equation}
satisfying \textup{(}in addition to the properties defining an
inverse system\textup{)} the following two conditions:\\[.5em]
\textup{(i)}\ \ For every {\em successor} ordinal $j\<{+}1\in\theta$
and every $\alpha\in\Omega_\D,$ we have
\begin{equation}\begin{minipage}[c]{34pc}\label{x.P_j+1*a}
$(P_{j+1})_\alpha\ =\ \coprod_{\ari(\alpha)}\,(P_j)\ba\<,$
\end{minipage}\end{equation}
with the pseudocoprojection associated
with each $\iota\in\ari(\alpha)$ given by
\begin{equation}\begin{minipage}[c]{34pc}\label{x.c_j+1*a}
$c^{P_{j+1}}_{\alpha,\<\iota}\ =\ %
q^{(P_j)\ba}_{\ari(\alpha),\<\iota}\circsm
(p_{j,\<j+1})\ba,$\quad
i.e., the composite mapping\\[.5em]
$(P_{j+1})\ba\,\to\,(P_j)\ba\,\to\,%
\coprod_{\ari(\alpha)}\,(P_j)\ba\,=\,(P_{j+1})_\alpha,$
\end{minipage}\end{equation}
and with the connecting map
$(p_{j,\<j+1})_\alpha:(P_{j+1})_\alpha\to(P_j)_\alpha$
taken to be the unique map making commuting triangles with the
pseudocoprojections $c^{P_j}_{\alpha,\<\iota}$
of $P_j$ and the genuine coprojection maps
$q^{(P_j)\ba}_{\ari(\alpha),\<\iota}: (P_j)\ba\to
\coprod_{\ari(\alpha)}\,(P_j)\ba=(P_{j+1})_\alpha;$
that is, the unique map $(p_{j,\<j+1})_\alpha$ satisfying
\begin{equation}\begin{minipage}[c]{34pc}\label{x.pjk*a}
$(p_{j,\<j+1})_\alpha\circsm q^{(P_j)\ba}_{\ari(\alpha),\<\iota}%
\ =\ c^{P_j}_{\alpha,\<\iota}$\quad for all $\iota\in\ari(\alpha).$
\end{minipage}\end{equation}
\vspace{.17em}
\textup{(ii)}\ \ For every nonzero {\em limit{}} ordinal $k\in\theta$
and every $\alpha\in\Omega_\D,$ we have
\begin{equation}\begin{minipage}[c]{34pc}\label{x.P_k*a}
$(P_k)_\alpha\ =\ \limit_{j<k} (P_j)_\alpha,$\quad with the connecting
maps $(p_{j,\<k})_\alpha:(P_k)_\alpha\to(P_j)_\alpha$ given by the
universal cone associated with this inverse limit.
\end{minipage}\end{equation}

A morphism $f:P\to P'$ of $\!\theta\!$-indexed
precoalgebras will mean a morphism
of inverse systems of pseudocoalgebras, which ``respects''
the relation\textup{~(\ref{x.P_j+1*a})}, in the sense
that for every $\alpha\in\Omega_\D$ and every successor
ordinal $j\<{+}1\in\theta,$ the component
$(f_{j+1})_\alpha: (P_{j+1})_\alpha\to
(P'_{j+1})_\alpha$ is the $\!\ari(\alpha)\!$-fold copower
of the component $(f_j)\ba: (P_j)\ba\to(P'_j\<)\ba\<.$
\end{definition}

Note that the above definition does not specify the objects
$(P_k)\ba,$ the connecting maps $(p_{j,\<k})\ba,$
or the pseudo-co-operations $\alpha^{P_k},$
nor, for $k$ a limit ordinal, the pseudocoprojections
$c^{P_k}_{\alpha,\<\iota}.$
However, the description of $P$ as an inverse system,
together with the definition of a morphism of pseudocoalgebras,
imply certain relations that these must satisfy.
In particular, for $k$ a limit ordinal, the pseudocoprojections
and pseudo-co-operations are uniquely determined by~(\ref{x.P_k*a}) as
soon as $(P_k)\ba$ and its cone of maps $(p_{j,\<k})\ba$ are given.
Let us also verify that given ordinals $j\leq j'$
with $j'\<{+}1\in\theta,$ we have
\begin{equation}\begin{minipage}[c]{34pc}\label{x.jj'+1*a}
$(p_{j+1,\<j'+1})_\alpha\ =\ \coprod_{\ari(\alpha)}\,(p_{j,\<j'})\ba\<,$
\end{minipage}\end{equation}
as one would expect from~(\ref{x.P_j+1*a}).
This is trivial for $j=j'.$
In the contrary case, by the definition of the
right-hand side of~(\ref{x.jj'+1*a}), it suffices to prove
that for all $\iota\in\ari(\alpha)$ we have
$(p_{j+1,\<j'+1})_\alpha\circsm q^{P_{j'}}_{\ari(\alpha),\iota}=
q^{P_j}_{\ari(\alpha),\iota}\circsm(p_{j,\<j'})\ba.$
Substituting $(p_{j+1,\<j'+1})_\alpha=
(p_{j+1,\<j'})_\alpha\circsm (p_{j',\<j'+1})_\alpha$ into
the left side of this equation, and
$(p_{j,\<j'})\ba=(p_{j,\<j+1})\ba\circsm(p_{j+1,\<j'})\ba$
into the right side, and simplifying the results using~(\ref{x.pjk*a})
and~(\ref{x.c_j+1*a}) respectively, we see that the desired
equality is a case of the commuting square relating
pseudocoprojections under the map $p_{j+1,\<j'}:P_{j'}\to P_{j+1}.$

The above definition, like the definition of a $\!\D\!$-pseudocoalgebra
$S$ in \S\ref{S.pseudo}, brings in $\D$ only through its type,
$\Omega_\D.$
In the definition of $\!\D\!$-pseudocoalgebra,
no more was needed; the identities of $\D$ came in via
the genuine $\!\D\!$-coalgebras
mapped to $S,$ whose colimit gave our final $\!\D\!$-coalgebra.
Here, however, we will be creating $\!\D\!$-coalgebras out of
precoalgebras, so we will need to introduce the concept of
a precoalgebra cosatisfying an identity.

The analogs, in a precoalgebra $P,$ of the primitive co-operations
$\alpha^R$ $(\alpha\in\Omega_\D)$ of a coalgebra
are the pseudo-co-operations $\alpha^{P_k}.$
We shall now form from these further
maps, corresponding to {\em derived} operations, i.e.,
$\!\Omega_\D\!$-terms -- expressions constructed by formal composition
from primitive operations and projection operations (with the
latter appearing only at the ``bottom'' step).

We need one bit of notation:  If $A$ and $B$ are objects of
a category $\fb{A},$ $\kappa$ is a cardinal such that
$\coprod_\kappa A$ exists, and we are given a family of morphisms
$f_\iota: A\to B$ $(\iota\in\kappa),$ then
$\bigvee_{\kern-.1em\iota\in\kappa}\<f_\iota: \coprod_\kappa A\to B$
will denote the morphism (whose existence and uniqueness are guaranteed
by the universal property of $\coprod_\kappa A)$ satisfying
\begin{equation}\begin{minipage}[c]{34pc}\label{x.bigvee}
$(\bigvee_{\kern-.1em\iota\in\kappa}f_\iota)\,
\circsm\,q^A_{\kappa,\<\eta}
\ =\ f_\eta$\quad for all $\eta\in\kappa.$
\end{minipage}\end{equation}
\vspace{.17em}

\begin{definition}\label{D.instance}
Suppose $s$ is an $\!\Omega_\D\!$-term of arity $\gamma,$ and $P$ a
$\!\theta\!$-indexed precoalgebra of type $\Omega_\D$ in a
category $\fb{A}.$
For $j\leq k\in\theta,$ a {\em $\!(j,k)\!$-instance} of $s$ in
$P$ \textup{(}unique, by the next lemma, if it exists for the
given $k,$ $j$ and $s)$ will mean a map
\begin{quote}
$(P_k)\ba\ \to\ \coprod_\gamma\,(P_j)\ba\<,$
\end{quote}
arising under the following recursive construction:

If $s$ is the $\!\iota\!$-th projection map for some $\iota\in\gamma,$
then for all $j\leq k\in\theta,$ the $\!(j,k)\!$-instance of $s$ is
given by $q^{(P_j)\ba}_{\gamma,\<\iota}\circsm(p_{j,\<k})\ba\<.$
\textup{(}In particular, taking $j=k,$ the $\!(k,k)\!$-instance of
the $\!\iota\!$-th projection $s$ is the $\!\iota\!$-th
coprojection, $q^{(P_k)\ba}_{\gamma,\<\iota}.)$

If $s$ has the form $\alpha(u_0,u_1,\dots),$ where $\alpha\in\Omega_\D,$
and $(u_0,u_1,\dots)$ is an $\!\ari(\alpha)\!$-tuple
of $\!\Omega_\D\!$-terms \textup{(}some or all of
which may be projections and/or primitive operations\textup{)}, each
of arity $\gamma,$ and if for some $i\leq j<k\in\theta$ we have,
for each $\iota\in\ari(\alpha),$ an $\!(i,j)\!$-instance of $u_\iota,$
which we denote $U_\iota: (P_j)\ba\to\coprod_\gamma\,(P_i)\ba,$
then an $\!(i,k)\!$-instance of $s$ is given by the composite map
\begin{equation}\begin{minipage}[c]{34pc}\label{x.instder}
$(\bigvee_{\kern-.1em\iota\in\ari(\alpha)}\<U_\iota)
\circsm\alpha^{P_{j+1}}\circsm (p_{j+1,\<k})\ba:\\[.17em]
\hspace*{4em}(P_k)\ba\,\to\,(P_{j+1})\ba\,\to\,(P_{j+1})_\alpha=
\coprod_{\ari(\alpha)}\,(P_j)\ba\,\to\,\coprod_\gamma\,(P_i)\ba\<.$
\end{minipage}\end{equation}
\end{definition}

Point~(iii) of the following lemma is the uniqueness result mentioned
in the first paragraph of the above definition.

\begin{lemma}\label{L.uniq_inst}
Let $s$ be an $\!\Omega_\D\!$-term of
arity $\gamma,$ and $P$ a $\!\theta\!$-indexed precoalgebra
of type $\Omega_\D,$ for some ordinal $\theta.$
Then\\[.5em]
\textup{(i)}\ \ If for some $i\leq k\leq k'\in\theta$ there exists an
$\!(i,k)\!$-instance $S$ of $s$ in $P,$ then
there is also an $\!(i,k')\!$-instance, given by
$S\circsm(p_{k,\<k'})\ba\<.$\\[.5em]
\textup{(ii)}\ \ If for some $i'\leq i\leq k\in\theta$ there exists an
$\!(i,k)\!$-instance $S$ of $s$ in $P,$ then
there is also an $\!(i',k)\!$-instance, given by
$(\coprod_\gamma (p_{i',\<i})\ba)\circsm S.$\\[.5em]
\textup{(iii)}\ \ If for some $i\leq k\in\theta$ there exist
$\!(i,k)\!$-instances of $s$ in $P,$ then all such instances are equal.
\end{lemma}

\begin{proof}
If $s$ is a projection map, then~(i) and~(ii) follow
immediately from the relation $p_{i,k}=p_{i,j}\circsm p_{j,k},$
and~(iii) holds because there are
no choices in that case of Definition~\ref{D.instance}.
So assume $s=\alpha(u_0,u_1,\dots)$ as in the second part of
Definition~\ref{D.instance}, and suppose inductively that the three
statements are true for all the $u_\iota$ $(\iota\in\ari(\alpha)).$

Statement~(i) follows again from the
formula $(p_{j,\<k'})\ba=(p_{j,\<k})\ba\circsm(p_{k,\<k'})\ba,$
applied in this case
to the definition of an instance of $s=\alpha(u_0,u_1,\dots).$
Statement~(ii) is obtained by combining that definition
with the inductive assumption that~(ii) holds for all $u_\iota.$

In verifying~(iii), note that given the inductive
assumption that~(iii)
holds for all $u_\iota,$ the only place nonuniqueness can come into the
definition of an $\!(i,k)\!$-instance of $s$ is in the choice of $j.$
So suppose we have $i\leq j<j'<k,$ and that for each
$\iota\in\ari(\alpha)$
there is an $\!(i,j)\!$-instance $U_\iota$ of $u_\iota.$
Then by part~(i) and our inductive assumption,
the unique $\!(i,j')\!$-instance of each $u_\iota$ is
given by $U'_\iota=U_\iota\circsm(p_{j,\<j'})\ba\<.$
We now evaluate~(\ref{x.instder}) with $j'$ for $j,$ getting the map
\begin{equation}\begin{minipage}[c]{34pc}\label{x.jj'}
\hspace*{1.25em}$(\bigvee_{\kern-.1em\iota\in\ari(\alpha)}\<U'_\iota)
\circsm\<\alpha^{P_{j'+1}}\circsm(p_{j'+1,\<k})\ba\\[.5em]
=\ (\bigvee_{\kern-.1em\iota\in\ari(\alpha)}
\<(U_\iota\circsm(p_{j,\<j'})\ba))
\circsm\<\alpha^{P_{j'+1}}\circsm(p_{j'+1,\<k})\ba\\[.5em]
=\ (\bigvee_{\kern-.1em\iota\in\ari(\alpha)}\<U_\iota)\circsm
(\coprod_{\ari(\alpha)}\<(p_{j,\<j'})\ba)\circsm
\alpha^{P_{j'+1}}\circsm(p_{j'+1,\<k})\ba\\[.5em]
=\ (\bigvee_{\kern-.1em\iota\in\ari(\alpha)}\<U_\iota)\circsm
\<(p_{j+1,\<j'+1})_\alpha\circsm
\alpha^{P_{j'+1}}\circsm(p_{j'+1,\<k})\ba$
\qquad(by~(\ref{x.jj'+1*a}))\\[.5em]
$=\ (\bigvee_{\kern-.1em\iota\in\ari(\alpha)}\<U_\iota)\circsm\<
\alpha^{P_{j+1}}\circsm(p_{j+1,\<j'+1})\ba
\circsm(p_{j'+1,\<k})\ba\\[.5em]
=\ (\bigvee_{\kern-.1em\iota\in\ari(\alpha)}\<U_\iota)\circsm\<
\alpha^{P_{j+1}}\circsm(p_{j+1,\<k})\ba\<,$
\end{minipage}\end{equation}
giving the required equality of our two $\!(i,k)\!$-instances of $s.$
\end{proof}

Let us define the {\em depth} of an $\!\Omega_\D\!$-term by
taking projections to have depth~$0,$ and recursively defining
$\mathrm{depth}(\alpha(u_0,u_1,\dots)),$ where $\alpha$ is
a primitive operation, to be the ordinal
$\sup_{\iota\in\ari(\alpha)}(\mathrm{depth}(u_\iota))+1.$
Then one can verify by induction that
in a $\!\theta\!$-indexed precoalgebra of type $\Omega_\D,$
an $\!\Omega_\D\!$-term $s$ has a $\!(j,\<k)\!$-instance if and only if
$k\geq j+\nolinebreak\mathrm{depth}(s).$

\vspace{.5em}
We can now say what it means for a precoalgebra
of type $\Omega_\D$ to cosatisfy a system of identities.

\begin{definition}\label{D.cosatisfy}
Let $P$ be a $\!\theta\!$-indexed precoalgebra of type $\Omega_\D$ in a
category $\fb{A}.$

Given $\!\Omega_\D\!$-terms $s$ and $t$
of the same arity $\gamma,$ we shall
say that $P$ {\em cosatisfies} the identity $s=t$ if
for all $j<k\in\theta$ for which both a $\!(j,k)\!$-instance of $s$
and a $\!(j,k)\!$-instance of $t$ exist,
these are equal as maps $(P_k)\ba\to\coprod_\gamma (P_j)\ba\<.$

If our category $\fb{A}$ is a variety $\C$ of algebras,
we shall say that the identity $s=t$ is cosatisfied
{\em at} an element $x\in|(P_k)\ba|$ if
for all $j<k\in\theta$ for which both a $\!(j,k)\!$-instance of $s$
and a $\!(j,k)\!$-instance of $t$ exist, these agree on~$x.$

Recalling from \textup{\S\ref{S.intro}} that $\Phi_\D$ denotes
a set of defining identities for the variety $\D,$
we shall say that $P$ is a {\em $\!\D\!$-precoalgebra}
if it is a precoalgebra of type $\Omega_\D$ which
cosatisfies each member of $\Phi_\D.$
\end{definition}

{\em Remarks.}
If two sets of identities for $\!\Omega_\D\!$-algebras,
$\Phi_\D$ and $\Phi',$ are each derivable from the other by the
standard method of computing with identities, then the same
$\!\Omega_\D\!$-algebras satisfy $\Phi_\D$ and $\Phi',$
and the same $\!\Omega_\D\!$-coalgebras cosatisfy them.
But it is not in general true that the same
$\!\Omega_\D\!$-precoalgebras cosatisfy them.
This is because when (for example) we
deduce an identity $s=t$ in $\Phi'$ from the
identities in $\Phi_\D,$ the $\!\Omega_\D\!$-terms $s$ and $t$ may both
have smaller depth than some terms of the identities of $\Phi_\D$ from
which we derived them, or than some terms appearing in intermediate
steps of the derivation.
(E.g., the group identities and the depth-$\!3\!$
identity $((x\cdot x)\cdot x)\cdot x=x\cdot x$
imply the depth-$\!1\!$ identity $x\cdot x=e.)$
Then $s$ and $t$ will admit $\!(j,k)\!$-instances
for some pairs $(j,k)$ such that not all terms of the original
identities and/or intermediate steps do,
so there is no way to deduce the cosatisfaction of all
instances of $s=t$ from the cosatisfaction
of all instances of the identities of~$\Phi_\D.$

On the other hand, for any $j$ and $k$ sufficiently far apart to
accommodate all steps of a derivation of $s=t$ from the identities
of $\Phi_\D,$ cosatisfaction of the identities of $\Phi_\D$
does imply cosatisfaction of the $\!(j,k)\!$-instance of $s=t.$
(This is why the coalgebras we will get by a limit process
from our $\!\D\!$-precoalgebras in the next section
will cosatisfy all identities holding in the variety $\D,$
regardless of our choice of defining identities.)

It is to make the concept of a $\!\D\!$-precoalgebra well defined
despite the above phenomenon that we are using
a nonstandard definition of variety, in which a variety $\D$
must be given with a specified set $\Phi_\D$ of defining identities.
Another way to get well-definedness would have been to
require that a $\!\D\!$-precoalgebra
cosatisfy {\em all} identities holding in $\D.$
But this would not have gotten around the underlying problem.
For instance, if $\D_1$ and $\D_2$ are two
varieties of the same type $\Omega_\D,$ then under that
definition, an $\!\Omega_\D\!$-precoalgebra might be both
a $\!\D_1\!$-precoalgebra and a $\!\D_2\!$-precoalgebra
without being a $\!\D_1\cap\D_2\!$-precoalgebra, since it
could cosatisfy all identities of $\D_1$ and of $\D_2,$
but fail to cosatisfy some identities that these together imply.
And more important for us, in calculations such as those we will carry
out in \S\S\ref{S.BiBi}-\ref{S.>1zeroary} for
particular varieties, it will be simplest if one only needs to study
cosatisfaction of a fixed family of defining identities.

\vspace{.5em}
Back to the task at hand.
Let $S$ be a $\!\D\!$-pseudocoalgebra in $\C.$
Then for any ordinal $\theta>0$ there is a natural recursive
construction of a $\!\theta\!$-indexed $\!\D\!$-precoalgebra $P$
having $P_0=S.$
Namely, suppose $0<k\in\theta,$ and that the pseudocoalgebras $P_j$ and
connecting maps among them have been defined for all $j<k,$ so
as to make $(P_j)_{j<k}$ a $\!k\!$-indexed $\!\D\!$-pseudocoalgebra.
Let us describe how to define $P_k$ and its maps to the~$P_j.$

First suppose $k$ is a successor ordinal, $k=j\<{+}1.$
The objects $(P_{j+1})_\alpha$ $(\alpha\in\Omega_\D)$ and the
connecting maps $(p_{j,\<j+1})_\alpha: (P_{j+1})_\alpha\to(P_j)_\alpha$
are determined by~(\ref{x.P_j+1*a}) and~(\ref{x.pjk*a}),
so we need only construct $(P_{j+1})\ba$ and the maps coming from it.
Let us start by forming a universal object having maps to all
the necessary objects:
\begin{equation}\begin{minipage}[c]{34pc}\label{x.Qj+1}
$Q_{j+1}\ =\ (P_j)\ba\times\,
\prod_{\alpha\in\Omega_\D} (P_{j+1})_\alpha.$
\end{minipage}\end{equation}
We regard $Q_{j+1}$ as a candidate for
$(P_{j+1})\ba,$ with the projection to $(P_j)\ba$ as the candidate
for the map $(p_{j,\<j+1})\ba,$ the projections to the
$(P_{j+1})_\alpha$ as the candidates for the
pseudo-co-operations $\alpha^{P_{j+1}},$
and finally, with the candidate for each pseudocoprojection
$c^{P_{j+1}}_{\alpha,\<\iota}$ given, as required
by~(\ref{x.c_j+1*a}), by the composite of
the projection $Q_{j+1}\to(P_j)\ba$ with the $\!\iota\!$-th coprojection
$q^{(P_j)\ba}_{\ari(\alpha),\<\iota}:
(P_j)\ba\to\coprod_{\ari(\alpha)}(P_j)\ba= (P_{j+1})_\alpha.$

In general, the maps described above will not satisfy the conditions
required to give a morphism of pseudocoalgebras $P_{j+1}\to P_j.$
However, if we replace $Q_{j+1}$ by the joint equalizer of the pairs of
maps that need to agree, this gives us a universal
$\!k\<{+}\<1\<\!$-indexed $\!\Omega_\D\!$-precoalgebra extending the
given $\!j\<{+}1\!$-indexed $\!\D\!$-precoalgebra.
Passing, further, to the equalizer of all pairs of maps out of
this subalgebra of $Q_{j+1}$ that must agree if our
$\!\Omega_\D\!$-precoalgebra is
to cosatisfy the identities of $\Phi_\D,$ and calling this new
equalizer $(P_{j+1})\ba,$ we get a universal extension of our
$\!j\<{+}1\!$-indexed $\!\D\!$-precoalgebra $(P_i)_{i<j+1}$ to a
$\!k\<{+}\<1\<\!$-indexed $\!\D\!$-precoalgebra $(P_i)_{i\leq j+1}.$

If $k$ is a limit ordinal, we similarly know that the
objects $(P_k)_\alpha$ $(\alpha\in\Omega_\D)$
must be defined by~(\ref{x.P_k*a}).
In this case we take as our candidate for $(P_k)\ba$ the object
\begin{equation}\begin{minipage}[c]{34pc}\label{x.Pkbaselim}
$Q_k\ =\ \limit_{j<k} (P_j)\ba\<,$
\end{minipage}\end{equation}
and as candidates for the maps $(p_{j,k})\ba,$ the universal cone
of maps from this object to the $(P_j)\ba$ $(j<k).$
This automatically has pseudocoprojection and pseudo-co-operation
maps, induced under the limit process by the pseudocoprojection
and pseudo-co-operation maps of the $P_j$ $(j<k),$ which will
make the necessary diagrams commute to give us an
$\!\Omega_\D\!$-precoalgebra.
Will it also cosatisfy the identities of $\D$?
It will if $\D$ is finitary, since in that case,
if we take such an identity $s=t,$ and for some $i<k$ calculate
the $\!(i,k)\!$-instances of $s$ and $t,$ then our calculation
will involve only finitely many indices $j<k,$ so these
indices will have a common strict upper bound $k_0<k;$
and we can get $\!(i,k_0\<{+}1)\!$-instances of $s$ and $t,$ which
by hypothesis will be equal; whence by Lemma~\ref{L.uniq_inst}(i) and
(iii), our $\!(i,k)\!$-instances of $s$ and $t$ are also equal.
For not necessarily finitary $\D,$ the same argument
works if $k$ has cofinality $\geq\lambda_\D.$
If, finally, $\D$ is infinitary and $k$ has cofinality
$<\lambda_\D,$ and {\em not} all identities of $\D$ are cosatisfied
by this $\!\Omega_\D\!$-precoalgebra, we replace $Q_k,$ as
before, with the joint equalizer of the relevant pairs of maps, i.e.,
we use the subalgebra $(P_k)\ba\subseteq Q_k$ consisting of all
elements at which these identities are cosatisfied.

Let us record the universal property of the above construction.

\begin{proposition}\label{P.P_is_final}
Let $S$ be a $\!\D\!$-pseudocoalgebra in $\C,$ let $\theta>0$
be an ordinal, and let $P$ be the $\!\theta\!$-indexed
$\!\D\!$-precoalgebra with $P_0=S$ constructed recursively
as described above.
Then $P$ is final among $\!\theta\!$-indexed
$\!\D\!$-precoalgebras $P'$ with $P'_0=S,$ or more
generally, given with a morphism $P'_0\to S.$

In particular, if $R$ is a genuine $\!\D\!$-coalgebra with a morphism
to $S,$ then the precoalgebra formed by taking $\theta$ copies of
$R$ \textup{(}regarded as pseudocoalgebras via the
functor $\psi),$ and connecting them by identity morphisms, admits
a unique map to $P$ extending the given map to $S=P_0.$\endproof
\end{proposition}

Let me sketch concretely how the morphism of the last
paragraph above is obtained.
Given $k>0,$ suppose that for all $j<k$ we have obtained a unique
family of morphisms of pseudocoalgebras $f_j:R\to P_j$ that make
commuting triangles with the maps $p_{i,\<j}:P_j\to P_i,$ and such that
$f_0$ is the given map $R\to P_0.$
If $k$ is a successor ordinal $j\<{+}1,$
then for each $\alpha\in\Omega_\D,$
we must, by the definition of a morphism of precoalgebras, take
$(f_{j+1})_\alpha$ to be $\coprod_{\ari(\alpha)}(f_j)\ba:
\coprod_{\ari(\alpha)}|R\<|\to
\coprod_{\ari(\alpha)}(P_j)\ba=(P_{j+1})_\alpha.$
To construct $(f_{j+1})\ba,$ we first map $|R\<|$
into each of the components of the product algebra $Q_{j+1}$
of~(\ref{x.Qj+1}): into $(P_j)\ba$ by $(f_j)\ba,$
and into each $(P_{j+1})_\alpha$ by composing the
co-operation $\alpha^R$ of $R$ with $(f_{j+1})_\alpha,$ described above.
The combined map $|R\<|\to Q_{j+1}$ is easily shown to make commuting
squares with maps down to the components of $P_j.$
Moreover, the identities of $\D$ are cosatisfied at all
elements of the image of that map $|R\<|\to Q_{j+1},$
because they are cosatisfied by the coalgebra $R.$
Hence the above map factors through
the subalgebra $(P_{j+1})\ba\subseteq Q_{j+1},$ giving the
desired map $(f_{j+1})\ba\<.$

If $k$ is a limit ordinal, on the other hand, the maps
$(f_k)_\alpha$ are obtained from the maps $(f_j)_\alpha$ $(j<k)$ by the
universal property of $(P_k)_\alpha$ as a limit algebra,
while $(f_k)\ba$ is obtained, by this same property, as a map
into $Q_k,$ and is again seen to land in $(P_k)\ba$
because $R$ cosatisfies the identities of $\D.$

\vspace{.5em}
Before applying precoalgebras to get a concrete description of the
universal coalgebras of Theorem~\ref{T.final},
let us note a couple of examples of the phenomenon alluded
to in the remark following Definition~\ref{D.cosatisfy}: that
two systems $\Phi_\D$ and $\Phi'$ of identities for
$\!\Omega_\D\!$-algebras which are satisfied by the same class of
algebras, and hence cosatisfied by the same class of coalgebras,
may be cosatisfied by different classes of precoalgebras.
As suggested there, let us take $\Phi_\D$ to consist of the usual
identities for groups together with the depth-$\!3\!$
identity $((x\cdot x)\cdot x)\cdot x=x\cdot x\<,$ and
$\Phi'$ to differ from $\Phi_\D$ only in that
the lastmentioned identity is replaced by $x\cdot x=e.$
Let $\C$ be any variety having a cogroup $R$ which does
not cosatisfy $x\cdot x=e\<;$ e.g., $\C=\Gp,$ and $R=$
the coalgebra representing the identity functor.
Let us define a $\!2\!$-indexed $\!\Omega_\Gp\!$-precoalgebra $P$
in $\C$ by letting $P_0$ and
$P_1$ both be $R,$ with $p_{0,1}$ the identity morphism.
Then $P$ cosatisfies the identities of $\Gp,$ and
vacuously cosatisfies the depth-$\!3\!$ identity
$((x\cdot x)\cdot x)\cdot x=x\cdot x,$ since there are no instances
of it to test; but it does not cosatisfy $x\cdot x=e,$ because there
is an instance thereof, and its cosatisfaction would be equivalent to
the cosatisfaction of that identity by $R.$

If the use of a $\!2\!$-indexed precoalgebra here feels unsatisfying, we
can get from this example an $\!\omega\!$-indexed precoalgebra $P$ with
the same properties:  Take $P_0,$ $P_1,$ and $p_{0,1}$ as above, while
for $2\leq k\in\omega,$ let $(P_k)\ba$ be the initial algebra in $\C.$
The remainder of the structure of $P$ is uniquely determined
by Definition~\ref{D.pre}, and it is easy to see that $P$ still
has the properties described above.

(The reader might find it amusing to try to find a precoalgebra
cosatisfying $\Phi',$ i.e., the group identities together
with $x\cdot x=e,$ but not the set $\Phi''$ consisting of those
same identities together with $x\cdot y=y\cdot x,$ known to be implied
in groups by the identity $x\cdot x=e.$ Since $x\cdot y=y\cdot x$
has the same depth, $1,$ as $x\cdot x=e,$ one needs to modify
the above trick -- but only slightly.)

Let us also note why, in defining what it means for
a precoalgebra $P$ to cosatisfy an identity $s=t,$
we restricted attention
to $\!(j,\<k)\!$-instances with $j$ strictly less than $k.$
The only operations admitting $\!(k,k)\!$-instances are
the projections; so essentially the only nontrivial identity having
$\!(k,k)\!$-instances is $x=y,$ which defines the trivial variety
of any type.
Cosatisfaction of the $\!(k,k)\!$-instance of this identity
would be the statement that the two coprojections
$(P_k)\ba\stackrel{\longrightarrow}{\scriptstyle{\longrightarrow}}
(P_k)\ba\,\cP (P_k)\ba$ coincide.
But the way we constructed the level $P_k$ in our universal
$\!\D\!$-precoalgebra $P$ involved taking equalizers of pairs of maps
to objects of the previously constructed
lower levels $P_j$ $(j<k);$ so that
construction would have to be adjusted considerably to make it possible
to impose cosatisfaction of the $\!(k,k)\!$-instance of $x=y.$
Further, one of the objects so modified would be $(P_0)\ba,$
contradicting our choice to let $P_0$ be the given pseudocoalgebra $S.$
In contrast, our slightly weaker definition of cosatisfaction
allows us to treat $x=y$ like other identities.

\section{Coalgebras from precoalgebras.}\label{S.preco>co}

Let us now see how to get coalgebras from the universal
precoalgebras $P$ constructed in~Proposition~\ref{P.P_is_final}.

A nice situation is if the construction of the $\!P_k\!$
eventually ``stabilizes'', in the sense that after some point, as one
adds new steps, the connecting morphisms
$p_{j,\<k}$ are all isomorphisms of precoalgebras.
(Having an isomorphism at {\em one} step is not enough
to guarantee this.
One may get a run of isomorphisms which, if extended far
enough to test a certain identity, would fail to cosatisfy it;
instead, our recursive construction modifies $P_k$ by that step.
One can bound the lengths of such unstable runs in terms of the depths
of the identities of $\D;$ the bound will be at most $\lambda_\D.)$
If our construction does stabilize, then once that happens, the
pseudocoprojection maps $(P_k)\ba\to (P_k)_{\alpha}$ become genuine
coprojection maps $(P_k)\ba\to\coprod_{\ari(\alpha)}(P_k)\ba,$ and the
common value of the $P_k$ will be the
desired final $\!\D\!$-coalgebra over the pseudocoalgebra~$S.$
Conversely, if for some $k,$ $P_k$ is in fact a $\!\D\!$-coalgebra
(i.e., if the pseudocopower objects are copowers of the base object,
and the identities of $\D$ are cosatisfied in the resulting
$\!\Omega_\D\!$-coalgebra),
then our construction will stabilize at that point.

I do not know whether such stabilization always occurs:

\begin{question}\label{Q.stabilize}
For all choices of varieties $\C$ and $\D,$
and pseudocoalgebra $S,$ does the construction of
Proposition~\ref{P.P_is_final} eventually stabilize, in the
sense that there exists an ordinal $\theta_0$ such that whenever
$\theta_0\leq\nolinebreak[4]i< j<\theta,$ the morphism
$p_{i,\<j}: P_j\to P_i$ of $(P_k)_{k\in\theta}$ is an isomorphism?
\end{question}

Assuming $\D$ finitary, we shall show below that {\em if}
$\C$ belongs to a certain rather restricted class of varieties,
then our construction indeed stabilizes by $\theta_0=\omega\<;$
and that for a larger class of varieties $\C,$ containing more cases of
interest, one can take $P_\omega,$ do a little trimming, and again
get the desired final coalgebra over $S.$

To state the relevant conditions, recall that given small
categories $\fb{E}$ and $\fb{F}$ and a functor $G$ from
$\fb{E}\times\fb{F}$ into a category $\fb{A},$
if one forms {\em limits} over $\fb{E}$ and {\em colimits}
over $\fb{F}$ in the two possible orders (assuming these
limits and colimits exist),
then their universal properties yield a {\em comparison morphism}
\begin{quote}
$\colim_{Y\in\fb{F}}\ \limit_{X\in\fb{E}} G(X,Y)\ \to\ %
\limit_{X\in\fb{E}}\ \colim_{Y\in\fb{F}} G(X,Y),$
\end{quote}
and that when this is an isomorphism, one says the indicated
limit and colimit commute.
We shall now show that if $\fb{A}$ is a variety $\C$ such that the
above commutativity condition holds whenever we take for
$\fb{E}$ the opposite of the ordered set $\omega$
and for $\fb{F}$ any finite set (each regarded as a category
in the natural way), then the desired stabilization condition holds
for all $\!\D\!$-precoalgebras in $\C$ if the variety $\D$ is finitary;
and we will note some $\C$ to which this result is applicable.

(We repeat a notational warning given earlier:
In expressions like $\limit_{k\in\omega}\,\coprod_n A_k,$ the symbol
$\coprod_n$ indicates that we are
taking the $\!n\!$-fold copower of what follows.
The index running over $n$ is not shown; in particular,
it is not $k,$ which here indexes the inverse limit.
Note also that when we speak of an inverse limit over $\omega,$
or write $\limit_{k\in\omega},$ it is
understood that $\omega$ is made a category in the way opposite
to its order structure.)

\begin{proposition}\label{P.compar_iso}
Suppose $\fb{A}$ is a category having finite copowers and having
inverse limits over $\omega,$ and suppose that these commute; i.e., that
\begin{equation}\begin{minipage}[c]{34pc}\label{x.limvscP}
For every inverse system of objects $A_k$ of
$\fb{A}$ indexed by $\omega,$ and every natural number $n,$
the comparison morphism $\coprod_n\ \limit_{k\in\omega} A_k\to\,
\limit_{k\in\omega}\,\coprod_n A_k$ is an isomorphism.
\end{minipage}\end{equation}
Suppose, further, that $\D$ is a finitary variety,
and $(P_k)_{k\in\omega}$ an $\!\omega\!$-indexed $\!\D\!$-precoalgebra
in $\fb{A},$ and that we extend this to an $\!\omega\<{+}1\!$-indexed
precoalgebra of type $\Omega_\D$ by letting
\begin{equation}\begin{minipage}[c]{34pc}\label{x.P*wbase}
$(P_\omega)\ba\ =\ \limit_{k\in\omega} (P_k)\ba\<,$
\end{minipage}\end{equation}
with connecting maps $(p_{k,\<\omega})\ba$ given by the canonical
cone of maps from the above object to the objects $(P_k)\ba\<.$
\textup{(}As noted following Definition~\ref{D.pre}, this is enough
to determine the extended $\!\Omega_\D\!$-precoalgebra,
since $\omega$ is a limit ordinal.\textup{)}

Then $(P_k)_{k\in\omega+1}$ is again a $\!\D\!$-precoalgebra.
Moreover, $P_\omega$ is a $\!\D\!$-coalgebra.

In particular, if $\fb{A}$ is a variety $\C,$
and $P$ is the universal $\!\omega\<{+}1\!$-indexed precoalgebra
built from a pseudocoalgebra $S$ as
in Proposition~\ref{P.P_is_final}, then $P_\omega$
is the final $\!\D\!$-coalgebra in $\C$ over $S.$
\end{proposition}

\subsubsection*{\textsc{Sketch of Proof}}
For each $\alpha\in\Omega_\D,$ note that
\begin{quote}
\begin{tabbing}
$(P_\omega)_\alpha$\=$=\ \limit_{k\in\omega}(P_k)_\alpha\ =\ %
\limit_{0<k\in\omega}
\coprod_{\ari(\alpha)}(P_{k-1})\ba\ =\ %
\limit_{k\in\omega}
\coprod_{\ari(\alpha)}(P_k)\ba$\\[.5em]
\>$\cong\ \coprod_{\ari(\alpha)}\,
\limit_{k\in\omega}(P_k)\ba\ =\ %
\coprod_{\ari(\alpha)}(P_\omega)\ba\,.$
\end{tabbing}
\end{quote}

One finds that the pseudocoprojection maps
$(P_\omega)\ba\to (P_\omega)_\alpha\cong
\coprod_{\ari(\alpha)}(P_\omega)\ba$ are the genuine coprojection maps,
making $P_\omega$ a coalgebra.
As discussed in the lines following~(\ref{x.Pkbaselim}), the fact that
$\D$ is finitary and $(P_k)_{k\in\omega}$ cosatisfies the defining
identities of $\D$ as a precoalgebra implies that
$(P_k)_{k\in\omega+1}$ also cosatisfies them as a precoalgebra;
and one deduces from this that $P_\omega$
cosatisfies them as a coalgebra, i.e., is a $\!\D\!$-coalgebra.

The final assertion is easily seen from Proposition~\ref{P.P_is_final}.
\endproof

The next proposition establishes a sufficient condition
for~(\ref{x.limvscP}) (and another property we will
subsequently find useful) to hold in a variety $\C.$
Let me illustrate the technical-looking hypothesis by contrasting the
properties of coproducts in the two varieties $\Se$ and $\fb{Monoid}.$

Given any semigroup $A,$ consider the natural homomorphism from
$A\cP A$ to the coproduct $\{x\}\cP\{x\}$ of two copies
of the final (one-element) semigroup $\{x\}.$
The two generators of $\{x\}\cP\{x\}$ are
$q^{\{x\}}_{\<2,\<0}(x)$ and $q^{\{x\}}_{\<2,\<1}(x);$
and we see from the normal form for coproducts of
semigroups that if an element $u$ of $A\cP A$ maps to, say,
$q^{\{x\}}_{\<2,\<0}(x)\ q^{\{x\}}_{\<2,\<1}(x)\ %
q^{\{x\}}_{\<2,\<0}(x)$ in $\{x\}\cP\{x\},$ then it must have the form
$q^A_{\<2,\<0}(y_1)\ q^A_{\<2,\<1}(y_2)\ q^A_{\<2,\<0}(y_3)$
in $A\cP A,$ for unique $y_1,y_2,y_3\in|A\<|.$
On the other hand, in the variety $\fb{Monoid},$ the coproduct of
two copies of the final monoid $\{e\}$ is again $\{e\},$ so
every element of a monoid $A\cP A$ maps to the unique
element thereof, so knowing that an element $u\in|A\cP A\<|$
maps to $e$ does not yield a fixed monoid word that we can
say represents $u$ in terms of elements
$q^A_{\<2,\<0}(y_i)$ and $q^A_{\<2,\<1}(y_j).$
We now state for a general variety $\C$ the condition we have
shown that $\Se$ has, but $\fb{Monoid}$ does not have,
letting $\{x\}$ now denote the final $\!(1\!$-element) algebra of $\C:$
\begin{equation}\begin{minipage}[c]{34pc}\label{x.u=}
For every $u\in|\{x\}\cP\{x\}|$ there exists
a derived operation $s_u$ of $\C$ such that
\begin{quote}
$u\ =\ s_u(q^{\{x\}}_{\<2,\<0}(x),\,\dots,\,q^{\{x\}}_{\<2,\<0}(x),\ %
q^{\{x\}}_{\<2,\<1}(x),\,\dots,\,q^{\{x\}}_{\<2,\<1}(x)),$
\end{quote}
(say with $n_0$ arguments $q^{\{x\}}_{\<2,\<0}(x)$ and
$n_1$ arguments $q^{\{x\}}_{\<2,\<1}(x)),$ and
such that for every $\!\C\!$-algebra $A,$ and every element
$b\in|A\cP A\<|$ which is carried to $u$ by the map
$A\cP A\to\{x\}\cP\{x\}$ induced by the unique map $A\to\{x\},$ there
exist {\em unique} elements $b_1,\dots,b_{n_0+n_1}\in|A\<|$ such that
\begin{quote}
$b\ =\ s_u(q^{A}_{\<2,\<0}(b_1),\,\dots,\,q^{A}_{\<2,\<0}(b_{n_0}),\ %
q^{A}_{\<2,\<1}(b_{n_0+1}),\,\dots,\,q^{A}_{\<2,\<1}(b_{n_0+n_1})).$
\end{quote}
\end{minipage}\end{equation}

The next result shows that condition~(\ref{x.u=}) is useful,
but restrictive.

\begin{proposition}\label{P.red}
Let $\C$ be a finitary variety satisfying~\textup{(\ref{x.u=})}.
Then it satisfies the two conditions
\begin{equation}\begin{minipage}[c]{34pc}\label{x.arblimvscP}
For every functor $F:\fb{E}\to\C$
where $\fb{E}$ is a {\em connected} small category, and every
cardinal $\kappa,$ the comparison morphism
$\coprod_\kappa\,\limit_\fb{E}\,F(E)\to\,
\limit_\fb{E}\ \coprod_\kappa\<F(E)$ is an isomorphism.
\end{minipage}\end{equation}
\begin{equation}\begin{minipage}[c]{34pc}\label{x.cP1-1}
For every one-to-one homomorphism $f:A\to B$ of objects of $\C,$
and every cardinal $\kappa,$ the induced map
$\coprod_\kappa f: \coprod_\kappa A\to\coprod_\kappa B$ is one-to-one.
\textup{(}In the language of Definition~\ref{D.pure}, every subalgebra
of a $\!\C\!$-algebra is copower-pure.\textup{)}
\end{minipage}\end{equation}

Moreover,\textup{~(\ref{x.u=})} holds, inter alia, in every
finitary variety without zeroary operations and without identities,
and also in the varieties $\Se$ and $\fb{Ab};$ on the other hand,
it does not hold in $\fb{Monoid},$ $\Gp$ or $\fb{Lattice}.$
\end{proposition}

\begin{proof}
Using the observation that an $\!n\<{+}1\!$-fold
coproduct $A_0\cP\dots\cP A_n$ can be written as the
$\!n\!$-fold coproduct $(A_0\cP A_1)\cP A_2\cP\dots\cP A_n,$
it is easy to show by induction that~(\ref{x.u=})
implies, for all positive integers $n,$
the corresponding condition with pairwise copowers
replaced by $\!n\!$-fold copowers; and using the fact that
$\C$ is finitary, one passes in turn to the same result for
$\!\kappa\!$-fold copowers, for any cardinal $\kappa.$

Fixing $\kappa,$ let us now choose for each $u\in|\coprod_\kappa\{x\}|$
a derived operation $s_u$ as in the version of~(\ref{x.u=})
for $\!\kappa\!$-fold copowers, say of arity $n_u.$
Then we see that the $\!\kappa\!$-fold
copower functor on $\C$ can be described, at the set level, as the
disjoint union over $u$ of the functors $X\mapsto X^{n_u}.$
The disjoint union operation on sets commutes with limits over connected
categories, and constructions $X\mapsto X^n$ commute with all limits,
yielding the commutativity result~(\ref{x.arblimvscP}).
Similarly, disjoint unions and constructions $X\mapsto X^n$
respect one-one-ness, giving~(\ref{x.cP1-1}).

Let us now verify the assertions about particular varieties.

If $\C$ is a variety without zeroary operations and without identities,
then every element of a coproduct algebra $A\cP B,$
where for notational convenience we assume $A$ and $B$ disjoint,
can be written uniquely as a word in
elements of $A$ and elements of $B,$ subject only to the
condition that no operation occurring in this word has for
its arguments elements of $A$ alone or elements of $B$ alone.
(If there were zeroary operations, then these, and words involving
them, would be counterexamples to this statement.)
Comparing this normal form for elements of $A\cP A$ and
their images in $\{x\}\cP\{x\},$
we see that~(\ref{x.u=}) holds in such varieties.
The same applies to coproducts of semigroups, using the normal form for
$A\cP B$ given by parenthesis-free products of elements
of $A$ and $B,$ in which no two successive factors
come from the same semigroup,
and for $\fb{Ab},$ using the normal form in which every element is
written as the sum of an element of $A$ and an element of~$B.$

To prove (by a method different from our earlier informal
sketch) that~(\ref{x.u=})
does not hold for $\C=\fb{Monoid},$ and to get the corresponding
results for $\Gp$ and $\fb{Lattice},$ we first
note that $\{x\}\cP\{x\}$ is finite in each of these cases
(the one-element algebra in the first two cases, since $x$ is the
identity element; a four-element lattice in the last).
Hence if~(\ref{x.u=}) held, the coproduct
of two copies of any finite algebra in $\C$ would be finite.
But from the standard normal forms we see that the coproduct of two
copies of any nontrivial finite group or monoid is infinite.
In the lattice case, we recall that the free lattice
$L$ on two generators is finite
$(\<|L|=\{x_0,\,x_1,\,x_0\<{\vee}\<x_1,\,x_0\<{\wedge}\<x_1\}\<),$
but that the coproduct of two copies of this $L,$ the free lattice on
four generators, is infinite
\cite[\S VI.8,~Exercise~4]{GB.Lat} \cite[Exercise~5.3.9]{245}.
\end{proof}

The variety $\fb{Monoid}$ not only fails to satisfy~(\ref{x.u=}),
but also the consequence~(\ref{x.arblimvscP}), including the case
we are interested in,~(\ref{x.limvscP}).
For an element on the right side of the morphism of~(\ref{x.limvscP})
can have images in the monoids $\coprod_n A_k$ whose lengths as monoid
words increase without bound as
a function of $k$ (since the length of such a word will shorten under a
monoid homomorphism whenever some factor maps to the identity element),
and such an element cannot lie in the image of that morphism.
This behavior has analogs in most naturally arising varieties of
algebras.
For instance, in lattices, the identity $(x\vee y)\wedge x=x$
has the consequence that a length-$\!3\!$ word
$(x\vee y)\wedge x'$ collapses to length~$1$ when $x$
and $x'$ fall together, and one can use
this to build counterexamples to~(\ref{x.limvscP}).

In later sections, we will use Propositions~\ref{P.compar_iso}
and~\ref{P.red} to get explicit descriptions
of final coalgebras in certain varieties $\C$ satisfying~(\ref{x.u=}),
and hence~(\ref{x.limvscP}).
But since~(\ref{x.limvscP}) excludes so many important cases,
let us consider a condition on a variety $\C$ that is a little
weaker, but has consequences that are almost as pleasant:
\begin{equation}\begin{minipage}[c]{34pc}\label{x.cPlim}
For every $\!\omega\!$-indexed inverse system of objects $A_k$ of $\C,$
and every natural number $n,$ the comparison morphism
$\coprod_n\ \limit_{k\in\omega} A_k\,\to\,%
\limit_{k\in\omega}\,\coprod_n A_k$ is {\em one-to-one}.
\end{minipage}\end{equation}

As before, we will supplement the next result with examples of
varieties that satisfy its hypotheses and varieties that don't.
But these will take more work than before
(inter alia because positive examples are
more plentiful), so we will make them a separate following lemma.

\begin{proposition}\label{P.compar_1-1}
Suppose $\C$ is a variety of algebras
which satisfies~\textup{(\ref{x.cP1-1})} and\textup{~(\ref{x.cPlim})},
$\D$~is a finitary variety, and $P$
is any $\!\omega\!$-indexed $\!\Omega_\D\!$-precoalgebra in $\C.$
Let $(P_k)_{k\in\omega+1}$ be the $\!\omega\<{+}1\!$-indexed
extension of $P$ obtained by constructing $P_\omega$
using~\textup{(\ref{x.P*wbase})}.
Then --\\[.5em]
\textup{(i)}\ \ For each subalgebra $A$ of $(P_\omega)\ba$ and
$\alpha\in\Omega_\D,$ the subalgebra of $(P_\omega)_\alpha$
generated by the images $c^{P_\omega}_{\alpha,\iota}(A)$
of $A$ $(\iota\in\ari(\alpha))$ can be identified with
$\coprod_{\ari(\alpha)}A,$
with the restrictions of the maps $c^{P_\omega}_{\alpha,\iota}$
as its coprojections $q^A_{\ari(\alpha),\iota}.$\\[.5em]
\textup{(ii)}\ \ The set of subalgebras $A\subseteq(P_\omega)\ba$
such that each pseudo-co-operation $\alpha^{P_\omega}$ carries $A$
into the above
subalgebra $\coprod_{\ari(\alpha)}A\subseteq(P_\omega)_\alpha,$
and which thereby become $\!\Omega_\D\!$-coalgebras with
the restrictions of the
$\alpha^{P_\omega}$ as their co-operations, has a largest element.
Let us denote this by $|R\<|,$ and the resulting coalgebra
by $R.$\\[.5em]
\textup{(iii)}\ \ If the $\!\Omega_\D\!$-precoalgebra $P$ is
a $\D\!$-precoalgebra, then the above coalgebra $R$ will
be a $\!\D\!$-coalgebra.
If $P$ is in fact the universal $\!\omega\!$-indexed
$\!\D\!$-precoalgebra built from a pseudocoalgebra $S,$ as in
Proposition~\ref{P.P_is_final},
then $R$ is the final $\!\D\!$-coalgebra over~$S.$
\end{proposition}

\begin{proof}
For each $\alpha$ in $\Omega_\D,$ the case of~(i) with $A=(P_\omega)\ba$
follows from~(\ref{x.P*wbase}),~(\ref{x.P_k*a}) and~(\ref{x.cPlim}),
while~(\ref{x.cP1-1}) allows us to go from this case to that of
any $A\subseteq(P_\omega)\ba.$

Let me indicate two ways of obtaining the $|R\<|$ of~(ii):
``from above'' and ``from below''.

The former approach uses a transfinite recursion:
Start with $|R\<|_{(0)}=(P_\omega)\ba,$ and assuming recursively
that $|R\<|_{(\eta)}$ has been defined, let $|R\<|_{(\eta+1)}$ be the
subalgebra thereof consisting of all elements whose image
in each $(P_\omega)_\alpha$ under $\alpha^{P_\omega}$
$(\alpha\in\Omega_\D)$ lies in the subalgebra generated by the
images of $|R\<|_{(\eta)}$ under the $c^{P_\omega}_{\alpha,\<\iota}.$
For limit ordinals $\eta,$
let $|R\<|_{(\eta)}=\bigcap_{\epsilon<\eta} |R\<|_{(\epsilon)}.$
Thus, we get a descending chain of subalgebras,
which must stabilize because $(P_\omega)\ba$ has
small underlying set, yielding the desired~$|R\<|.$

For the alternative construction ``from below'',
let $X$ be the set of all subalgebras
$A$ of $(P_\omega)\ba$ such that the image of $A$ under
each pseudo-co-operation lies in the subalgebra of $(P_\omega)_\alpha$
generated by the images of $A$ under the pseudocoprojections.
Then the algebra generated by all members of $X$ will lie in
$X,$ and be the desired~$|R\<|.$

For $P$ a $\!\D\!$-precoalgebra, as in~(iii), we observed
following~(\ref{x.Pkbaselim}) that (if $\D$ was
finitary, as it is here,) the identities of $\D$
were cosatisfied in $(P_k)_{k\in\omega+1}.$
As before, it is easy to deduce that they will also
be cosatisfied in $|R\<|$ as a coalgebra, and that
for $P$ obtained from $P_0=S$ as in Proposition~\ref{P.P_is_final},
$R$ is the final $\!\D\!$-coalgebra over~$S.$
\end{proof}

Conditions~(\ref{x.cP1-1}) and~(\ref{x.cPlim}) tend to
hold in varieties where coproducts have nice normal forms,
as illustrated by the proof of the assertions of
the first paragraph of the next result.

\begin{lemma}\label{L.normal}
In addition to the varieties $\C$ satisfying\textup{~(\ref{x.u=})},
conditions~\textup{(\ref{x.cP1-1})}
and~\textup{(\ref{x.cPlim})} hold in the varieties $\Gp$
and $\fb{Monoid},$ and in the varieties of all unital or nonunital,
commutative or not-necessarily-commutative
\textup{(}four combinations in
all\textup{)} associative algebras over a field $K.$
Hence, in these cases, Proposition~\ref{P.compar_1-1} shows how to
construct the final $\!\D\!$-coalgebra in $\C$ whenever
$\D$ is a variety of finitary algebras.

On the other hand, neither~\textup{(\ref{x.cP1-1})}
nor~\textup{(\ref{x.cPlim})} holds in the
{\em subvarieties} of $\Gp,$ $\fb{Monoid}$ and $\Se$ generated by the
infinite dihedral group, nor in any of the four varieties of rings
$\!(\mathbb{Z}\!$-algebras\textup{)} corresponding to the
four varieties of $\!K\!$-algebras just mentioned.
\end{lemma}

\begin{proof}
To prove~(\ref{x.cP1-1}) and~(\ref{x.cPlim}) in $\fb{Monoid},$ recall
that a coproduct $\coprod_{\iota\in\kappa}A_\iota$ in that category
has a normal form almost like that of coproducts in $\Se{\<:}$
we again assume for notational simplicity that the $A_\iota$ are
disjoint, and the normal form again consists of parenthesis-free
finite products with no two successive
factors coming from the same monoid $A_\iota.$
But these products are now subject to the restriction
that no factor be the identity element $1$ of the monoid from which
it comes, while the length-zero product is allowed,
and becomes the identity element of the coproduct monoid.

It is immediate that given submonoids $A_\iota\subseteq B_\iota,$ an
expression for an element of $\coprod_{\iota\in\kappa} A_\iota$ will
be in normal form if and only if it is so in
$\coprod_{\iota\in\kappa} B_\iota,$ so the inclusion of
the former in the latter is one-to-one.
The case where all $B_\iota$ are copies of a single monoid,
and the submonoids $A_\iota$ are likewise equal, is~(\ref{x.cP1-1}).
Under the hypothesis of~(\ref{x.cPlim}), if an element of
$\coprod_n\,\limit_{k\in\omega} A_k$ is written in normal form, then
the image of that normal form expression in any $\coprod_n A_k$
automatically satisfies the condition that successive factors come from
different objects, while the statement that none of those factors
is $1$ constitutes
a finite set of conditions, which will hold simultaneously for
all {\em sufficiently large} $k.$
Thus, elements in normal form in $\coprod_n\,\limit_{k\in\omega} A_k$
map to elements in the $\coprod_n A_k$
that are in normal form for sufficiently large $k;$
and increasing $k$ still further if necessary, we can get any
two elements of $\coprod_n\,\limit_{k\in\omega} A_k$
with distinct normal forms to have images
in $\coprod_n A_k$ with distinct normal forms; hence
their images in $\limit_{k\in\omega}\ \coprod_n A_k$ are distinct.

The considerations for groups are identical, since their normal
form is the same as for monoids.

In the variety of associative unital $\!K\!$-algebras, one has an
analogous normal form result for a coproduct
$\coprod_{\iota\in\kappa} A_\iota$ of nonzero algebras:
One chooses for each of the $A_\iota$ an arbitrary $\!K\!$-vector-space
basis $X_\iota$ containing the identity element $1\in|A_\iota|$
and (again assuming the $A_\iota$ disjoint), the result
says that the set of parenthesis-free products of nonidentity
members of the union of these bases,
in which no two successive factors come from the same $X_\iota,$ forms
a $\!K\!$-vector-space basis of $\coprod_{\iota\in\kappa} A_\iota,$
with the empty product representing $1.$
So given nonzero algebras and subalgebras $A_\iota\subseteq B_\iota,$
we may take a basis of each $A_\iota$ which contains $1,$
extend it to a basis of $B_\iota,$ and observe that an expression in
normal form for a member of $\coprod_{\iota\in\kappa}A_\iota$ is also
the expression in normal form for the image of that element in
$\coprod_{\iota\in\kappa}B_\iota.$
As above,~(\ref{x.cP1-1}) follows.
In the situation of~(\ref{x.cPlim}), if we let $X$ denote a
$\!K\!$-vector space basis of $\limit_k A_k$ containing $1,$ then the
expression for any element of $\coprod_n\,\limit_k A_k$ involves the
copies of only finitely many elements of $X,$ and since these
finitely many elements and $1$ are linearly independent, they will
have images in some $A_k$ that are linearly independent,
and so can be incorporated into a basis of that $A_k;$ so with linear
independence in place of distinctness, the argument from the
monoid case also goes over.
(As noted, the above normal form is restricted to nonzero
algebras, equivalently, algebras in which $1\neq 0.$
Cases involving the zero algebra are trivial to handle.)

In the variety of ``nonunital'' associative $\!K\!$-algebras
(meaning algebras without a zeroary operation giving a multiplicative
identity element $1,$ which may or may not happen to have such an
element, and where such elements, when they exist, need not
be respected by homomorphisms), the normal form and the reasoning are
very similar: one merely deletes all mention
of identity elements, no longer allows
length-zero products, and does not exclude the zero algebra.

Coproducts in varieties of commutative associative $\!K\!$-algebras
have much simpler normal forms.
In the unital case, one again takes a basis (which may, but need not,
contain $1)$ for each algebra;
a basis for an $\!n\!$-fold coproduct $(n\in\omega)$
is then given by the set of products of one
basis element from each algebra, with factors $1$ not excluded:
\begin{equation}\begin{minipage}[c]{34pc}\label{x.x_1inX_1}
$\prod_{m=1,\dots,n} X_m\ =
\ \{x_1\dots x_n\ \mid\ x_1\in X_1,\,\dots\,,\,\<x_n\in X_n\}.$
\end{minipage}\end{equation}
In the nonunital case, this doesn't suffice, since a product of
basis elements from a proper subset of the $X_m$ will not
lie in~(\ref{x.x_1inX_1}); so we allow expressions with
variable numbers of factors; a basis is given by
\begin{equation}\begin{minipage}[c]{34pc}\label{x.x_*iinX_*i}
$\bigcup_{\emptyset\<\neq\<s\<\subseteq\<\{1,\dots,\<n\}}\ %
\prod_{m\in s} X_m.$
\end{minipage}\end{equation}
In either case, the reader can easily supply the arguments
giving~(\ref{x.cP1-1}) and~(\ref{x.cPlim}).

The first of our negative assertions concerns the variety $\C$ of
groups generated by the infinite dihedral group.
The example given following Definition~\ref{D.pure}
shows that~(\ref{x.cP1-1}) fails in that
variety for $\kappa=2$ and the inclusion
$2\<\mathbb{Z}\to\mathbb{Z}$ of infinite cyclic groups.
For a counterexample to~(\ref{x.cPlim}),
we shall use an inverse system given by a chain
$\dots\subseteq\nolinebreak A_2\subseteq\nolinebreak[3]
A_1\subseteq\nolinebreak A_0$ of subgroups
of a certain group $A;$ thus, the inverse limit of that
system will be the intersection of those subgroups.
Let $A$ be the abelian group
$\mathbb{Z}\times\bigoplus_\omega\mathbb{Z}/2\mathbb{Z},$
let $q:\mathbb{Z}\to\mathbb{Z}/2\mathbb{Z}$ be the quotient map,
and for each $k\in\omega,$ let $A_k$ be the subgroup of $A$
consisting of all elements $(m;a_0,a_1,\dots)$
$(m\in\mathbb{Z},\,a_j\in\mathbb{Z}/2\mathbb{Z})$
such that $q(m)=a_0=\dots=a_{k-1}.$
(As usual, the groups $\mathbb{Z}$ and $\mathbb{Z}/2\mathbb{Z}$ are
written additively; but $A$ and its subgroups will be thought of as
written multiplicatively, since we will be dealing with the
noncommutative group structures of their copowers in~$\C.)$
Letting $z=(2;0, 0,\dots)\in|A\<|,$ we see that in each $A_k,$ $z$ is
a square, namely of any element $(1;1,\dots,1,0,\dots)$ with initial
string of components $1\in\mathbb{Z}/2\mathbb{Z}$ of length
$\geq k;$ but the group
$\limit_{k\in\omega} A_k= \bigcap_{k\in\omega} A_k$
has no elements with odd first
component, since that would force infinitely many components
of the $\bigoplus_\omega\mathbb{Z}/2\mathbb{Z}$ summand to be nonzero.
Rather, we see that $\limit_{k\in\omega} A_k$ is free on the
single generator $z;$ hence the two copies of $z$
in $(\limit_{k\in\omega} A_k)\cP(\limit_{k\in\omega} A_k)$ do not
commute; but their images commute in $\limit_{k\in\omega} (A_k\cP A_k),$
since $z$ is a square in each~$A_k;$ so the map
of~(\ref{x.cPlim}) is not one-to-one.

Since these examples used identities not involving
the inverse or identity-element operations, they also work in the
subvarieties of $\fb{Monoid}$ and of $\Se$
generated by the infinite dihedral group.

In the four varieties of $\!\mathbb{Z}\!$-algebras listed,
one can get counterexamples using, similarly, cases where an
element $z$ in a subring $B$ of a ring $A,$ respectively
in an inverse limit of
rings $A_k$ does not admit a solution to $z=2z_0$ in that subring
or inverse limit, though it does in $A,$ respectively in all the $A_k.$
In this case, in place of the fact, used in the preceding
examples, that in our subvarieties of
$\Gp,$ etc., squares commute but some non-squares do not,
we can use the fact that in rings,
an element $z$ for which there is a solution to $z=2z_0$
must have product $0$ with any element $y$ satisfying $2y=0,$
but that other elements $z$ need not behave in this way.
The reader should not find it hard to construct specific examples.
\end{proof}

Our construction of the precoalgebra from which we obtained the
final object of $\coalg{\fb{Set}}{\Bi}$ was
in several ways simpler than the general construction
summarized in Proposition~\ref{P.P_is_final} above.
Point~(iii) of the next result indicates a class of cases
to which one of these simplifications applies.

\begin{proposition}\label{P.dropbase}
For $S$ a pseudocoalgebra, let us, in this
proposition, write $\pi^S: S\ba\to\prod_{\alpha\in\Omega_\D}S_\alpha$
for the map induced by the pseudo-co-operations
$\alpha^S: S\ba\to S_\alpha$ $(\alpha\in\Omega_\D).$

Let $S$ be a $\!\D\!$-pseudocoalgebra in $\C,$ $\theta>0$ an ordinal,
and $P$ the final $\!\theta\!$-indexed $\!\D\!$-precoalgebra with
$P_0=S,$ constructed as in Proposition~\ref{P.P_is_final}.
Then\\[.5em]
\textup{(i)}\ \ If $\pi^S$ is one-to-one, then
$\pi^{P_k}$ is one-to-one for every $k\in\theta.$\\[.5em]
\textup{(ii)}\ \ If $\pi^S$ is surjective, and
$\D$ has no identities \textup{(}i.e., is the
variety of all $\!\Omega_\D\!$-algebras\textup{)},
then $\pi^{P_k}$ is surjective for all finite $k\in\theta.$\\[.5em]
\textup{(iii)}\ \ If $\pi^S$ is bijective \textup{(}e.g., if $S$ is the
trivial pseudocoalgebra\textup{)} and $\D$ has no identities,
then $\pi^{P_k}$ is bijective for every $k\in\theta.$

In this last situation, if we use the bijections $\pi^{P_k}$ to
identify each $(P_k)\ba$ with $\prod_{\alpha\in\Omega_\D} (P_k)_\alpha,$
then the homomorphisms $(p_{j,\<k})\ba:(P_k)\ba\to(P_j)\ba$
correspond to the maps $\prod_{\alpha\in\Omega_\D}(P_{k})_\alpha\to
\prod_{\alpha\in\Omega_\D}(P_j)_\alpha$ induced by
the maps $(p_{j,\<k})_\alpha:(P_{k})_\alpha\to (P_j)_\alpha.$
\end{proposition}

\begin{proof}
We will get (i) and (iii) by transfinite induction on $k,$
and show that surjectivity as in~(ii) carries over from any ordinal
$j$ (finite or infinite) to its successor $j\<{+}1,$
which by finite induction gives~(ii).

First suppose $k$ is a successor ordinal $j\<{+}1,$ so
that $(P_k)\ba$ is constructed as
in the discussion beginning with~(\ref{x.Qj+1}).

Under the hypothesis of~(i), let $s$ and $t$ be distinct elements of
$(P_{j+1})\ba\<;$ we need to show that they are separated by at
least one of the pseudo-co-operations $\alpha^{P_{j+1}}.$
By the construction of $P_{j+1}$ referred to,
they are separated either by one
of these, or by the map $(p_{j,\<j+1})\ba:(P_{j+1})\ba\to(P_j)\ba\<.$
In the latter case, by inductive assumption their images under
this map are separated by some $\alpha^{P_j},$ and from the commuting
square made by $\alpha^{P_j}$ and $\alpha^{P_{j+1}},$
we see that $s$ and $t$ are also separated by the latter map.

For~(ii), consider any
$u$ in $\prod_{\alpha\in\Omega_\D}(P_{j+1})_\alpha.$
By our assumption that~(ii) holds
for $P_j,$ the image of $u$ in
$\prod_{\alpha\in\Omega_\D}(P_j)_\alpha$ can be written
$\pi^{P_j}(v)$ for some $v\in|(P_j)\ba|.$
Let $w=(v,u)\in|Q_{j+1}|=|(P_j)\ba|\times\,
|\prod_{\alpha\in\Omega_\D} (P_{j+1})_\alpha|.$
Our choice of $v$ guarantees that $w$ satisfies the
equalizer conditions imposed, in the construction of $(P_{j+1})\ba$
from $Q_{j+1},$ to
insure that $p_{j,\<j+1}$ is morphism of pseudocoalgebras.
The other equalizer conditions imposed there were to
insure that the identities of $\D$ were cosatisfied, so
for $\D$ having no identities, that set of conditions is empty.
Hence $(v,u)$ is an element of $P_{j+1},$ and it maps
to $u$ under $\pi^{P_{j+1}},$ as required.

The $k=j\<{+}1$ step for~(iii) follows from the conjunction
of the steps for~(i) and~(ii).
(Although~(ii) is only stated for finite $k,$ we formulated
our inductive step for an arbitrary successor ordinal,
as required to obtain this step of~(iii).)

In the case where $k$ is an infinite limit ordinal, $P_k$ is
just the $Q_k$ of~(\ref{x.Pkbaselim}) if $\D$ has no identities,
as in~(iii), or a subalgebra thereof in the general case.
Since inverse limits preserve products, one-one-ness and bijectivity
(though not, in general, surjectivity), the transfinite
inductive steps of~(i) and~(iii) follow.
\end{proof}

The condition that $\D$ have no identities in part~(iii) above
is highly restrictive; but in~\S\ref{S.Set}, after describing
the final object of $\coalg{\fb{Set}}{\Bi}$ using that idea,
we obtained the final objects of $\coalg{\fb{Set}}{\D},$ for
subvarieties $\D$ of $\Bi$ defined by various
identities, as {\em subcoalgebras} of that object.
Points~(ii) and~(iv) of the next result note a general context in which
this technique works.

\begin{proposition}\label{P.imagecoalg}
Let $\C$ be a variety satisfying~\textup{(\ref{x.cP1-1})},
and $\D$ any variety.
Let $S$ \textup{(}in
parts~\textup{(ii)} and~\textup{(iv)} below\textup{)}
be any $\!\D\!$-pseudocoalgebra in $\C,$ and let $\theta$
\textup{(}in parts~\textup{(iii)} and~\textup{(iv))}
be any nonzero ordinal.
Then\\[.5em]
\textup{(i)}\ \ The image of any homomorphism of $\!\D\!$-coalgebras
in $\C,$ $f\,{:}\ R\to R',$ is a subcoalgebra of $R',$ and cosatisfies
every identity that is cosatisfied by $R$ or by $R'.$\\[.5em]
\textup{(ii)}\ \ If $R'$ denotes the final $\!\Omega_\D\!$-coalgebra
over $S,$ then the final $\!\D\!$-coalgebra $R$ over $S$ can be
identified with a subcoalgebra of $R',$ which will be
the largest such subcoalgebra
that cosatisfies the identities in $\Phi_\D.$\\[.5em]
\textup{(iii)}\ \ Given a homomorphism $f\,{:}\ P\to P'$ of
$\!\theta\!$-indexed $\!\D\!$-precoalgebras in $\C,$ let $P''$ denote
the system of subalgebras $(P''_k)\ba\subseteq(P'_k)\ba$
and $(P''_k)_\alpha\subseteq(P'_k)_\alpha$ given by the images
of the components of $f,$ except for $(P''_k)_\alpha$
when $k$ is a nonzero limit ordinal, which, in accordance
with~\textup{(\ref{x.P_k*a}),} we define to be the subalgebra of
$(P'_k)_\alpha$ naturally isomorphic to $\limit_{j<k} (P''_j)_\alpha.$
Then $P''$ is subprecoalgebra of $P',$ and cosatisfies
every identity cosatisfied by $P$ or by~$P'.$\\[.5em]
\textup{(iv)}\ \ If $P'$ denotes the final $\!\theta\!$-indexed
$\!\Omega_\D\!$-precoalgebra with $P'_0=S,$ then the final
$\!\theta\!$-indexed $\!\D\!$-precoalgebra $P$ with $P_0=S$ can be
identified with a subprecoalgebra of $P',$ which will be the largest
such subprecoalgebra that cosatisfies the identities in $\Phi_\D.$
\end{proposition}

\begin{proof}
The hypothesis~(\ref{x.cP1-1}) insures that subalgebras in $\C$
are copower-pure.
It is then easy to see that in~(i), the image of $f$ will be a
subcoalgebra, and will cosatisfy the identities referred to.
The argument for~(iii) is similar.
(The existence of a subalgebra of $(P'_k)_\alpha$ naturally
isomorphic to $\limit_{j<k} (P''_j)_\alpha$ follows
from the fact that the $(P''_j)_\alpha$ are subalgebras of the
$(P'_j)_\alpha,$ mapped into one another by the connecting maps.)

To get~(ii), we apply~(i) to the case where
$R$ is the final $\!\D\!$-coalgebra over $S,$
$R'$ the final $\!\Omega_\D\!$-coalgebra over $S,$
and $f$ the morphism $R\to R'$ given by the universal property of $R'.$
By~(i), $f(R)$ is a coalgebra; so as $R$ is strongly quasifinal
over $S,$ the map $f:R\to f(R)$ must be an isomorphism, so
we may identify $R$ with the subcoalgebra $f(R)$ of $R'.$
If $R'$ had any $\!\D\!$-subcoalgebra $R''$ not contained
in $R,$ then the unique map of $R''$ to the final $\!\D\!$-coalgebra
$R$ over $S$ would give a map $R''\to R'$ over $S$ distinct
from the inclusion, contradicting the universal property of $R'.$

Similarly,~(iv) follows from~(iii).
\end{proof}

(Remark: The awkward construction of the precoalgebra $P''$ of~(iii)
above is necessitated by~(\ref{x.P_k*a}).
It might be preferable, instead, to modify the definition of
precoalgebra by weakening~(\ref{x.P_k*a}) to merely require
that for $k$ a limit ordinal $>0,$ the natural map
$(P_k)_\alpha\to\limit_{j<k} (P_j)_\alpha$ be a monomorphism.
When $\fb{A}$ is a variety $\C$ this simply says that the family of maps
$(P_k)_\alpha\to(P_j)_\alpha$ separates points of $(P_k)_\alpha.$
If this change of definition proves feasible, then
one could take $P''$ in the above result simply to be
the image of $P$ in $P',$ in analogy with~(i).)

\vspace{.5em}
We will apply these results to several examples in Part~II below.
We end Part~I by noting --

\section{A few directions for further investigation.}\label{S.further}

In \cite[\S\S63-64]{coalg}, generalizing work of D.\,Tall and
G.\,Wraith \cite{TW.Rep} \cite[\S9]{GW.aar}, I defined a
``TW-{\em comonad}\/'' on a variety $\C$ to be a representable
functor $F:\C\to\C$ with a comonad structure (morphisms $F\to FF$
and $F\to\mathrm{Id}_\C$ making appropriate diagrams commute).

Actually, the main focus there was on the left adjoints of such
functors, called ``TW-{\em monads}\/''.
As noted in \cite{Freyd}, for any coalgebra $R,$ the left adjoint
to the functor represented by $R$ can be viewed as a construction of
``generalized tensor product with $R.$''
Coalgebras given with morphisms that make the functors they represent
TW-comonads and the adjoints thereof TW-monads,
I dubbed ``TW-monad objects'' in \cite{coalg}.
But since we have not dealt with left adjoints of representable functors
here, I will mainly use the language of TW-comonads in these remarks.
(To complicate things further, in \cite{coalg} I also
defined ``WT-monads'', ``WT-comonads'', and ``WT-comonad objects'',
to mean representable functors with {\em monad} structures,
respectively their left adjoints, i.e., ``tensor product-like'' functors
with comonad structures, respectively
the coalgebras with additional structure that induce these functors.
But these will not be relevant below.)

Recall that a ``coalgebra object'' with respect to a comonad $F$
-- a different usage from the sense of coalgebra with which we are
concerned in this paper -- means an object $A$ of $\C$ with
a morphism $A\to F(A)$ making appropriate diagrams commute.
It was shown in \cite{coalg} that coalgebra objects (in that sense)
with respect to a TW-comonad are the same as algebra objects (in the
analogous sense) with
respect to the corresponding TW-monad, and that the category of
these is equivalent to a new variety of algebras (in the
sense of this paper), obtained by adjoining to $\C$
additional unary operations, corresponding to the elements of any
generating set (as an algebra) for the
representing coalgebra, and appropriate identities.

It was also noted that the {\em initial} representable functor from a
variety $\C$ to itself (for the existence of which, in the general
case, I referred
to this not-yet-written paper) always has a structure of TW-comonad.
For $\C=\fb{Ring}^1,$
it was shown, using the description of the initial representable functor
$\fb{Ring}^1\to\fb{Ring}^1$ as $S\mapsto S\times S^\mathrm{op},$
that the coalgebras with respect
to this comonad are the {\em rings with involution}
(an involution meaning an antiautomorphism of exponent $2).$

Now the concept of a ring with involution is of considerable independent
interest; it gives, for instance, the class of rings over which one
can define the conjugate transpose of a matrix,
and hence such constructions as the group of unitary matrices.
(The transpose operation alone, without conjugation, is of little use
for matrices over noncommutative rings; it behaves nicely in the
commutative case only because there, the identity map is an involution.)
This suggests that for other varieties, the extensions arising
from their initial representable endofunctors may also be of use.
The initial representable endofunctors (final coalgebras) of several
varieties will be determined in the remaining sections of this note.
I have no present plans to pursue the study of the corresponding
extended varieties; some reader might find it interesting to do so.

Of course, if we start with a variety such as $\fb{Monoid}$
or $\Gp$ having a unique derived zeroary operation, then the
initial representable functor, and hence the resulting TW-monad and
comonad, are trivial by Theorem~\ref{T.7/9}(ii).
But if $\C$ is any variety,
and we fix a nontrivial object $A$ of $\C,$ then the comma category
$(A\downarrow\C)$ of $\!\C\!$-algebras with homomorphisms of $A$
into them can be regarded as a variety, with any generating set for $A$
as additional zeroary operations; and perhaps the
above construction on such varieties would give useful information.

It would be interesting to know whether the above ``variety of
coalgebras with respect to the initial representable endofunctor''
construction is idempotent (up to natural equivalence).

\vspace{.5em}
To see another direction for investigation, let us start
with the following observations:
In $\fb{Set},$ the kinds of algebra structures that can arise are
extremely diverse, while coalgebra structures are much more limited.
On the other hand, once one has varieties $\C$ of set-based
algebras, the consideration of objects of $\C$ with algebra
structure (in the sense of an algebra structure on an object
of a general category $\fb{A}$ with products) gives nothing
essentially new; an object of $\C$ with a $\!\C'\!$-algebra structure
is simply an algebra in yet another variety $\C''.$
This can be thought of as a consequence of the fact that
products in $\C'$ are formed using products of underlying sets.
(There are, though, some interesting questions on how
$\C''$ relates to $\C$ and $\C';$ cf.\ \cite[\S\S9.12-13]{245}.)

However, {\em coalgebra} objects in a variety of algebras $\C$ can
be very diverse \cite{coalg}.

In a category $\coalg{\C}{\D}$ of such coalgebras, if we now look at
$\!\D'\!$-coalgebra objects for another variety $\D',$ this again gives
nothing essentially new; these are simply $\!\D''\!$-coalgebra objects
of $\C$ for an appropriate $\D'';$ essentially because coproducts in
$\coalg{\C}{\D}$ arise from coproducts of underlying $\!\C\!$-algebras.

However, {\em products} in $\coalg{\C}{\D}$ do not in general arise
from products of underlying $\!\C\!$-algebras; they are examples of
limits, which we have seen are ``exotic'' in these categories.
So -- what can one say about $\!\C'\!$-{\em algebra} objects
of $\coalg{\C}{\D},$ for $\C'$ another variety?
I don't know.

We have seen (Theorem~\ref{T.cofree}, first paragraph) that
given varieties $\C$ and $\D,$ every object $A$ of $\C$ determines
an object $R$ of $\coalg{\C}{\D}$ with a universal map $|R\<|\to A$
in $\C.$
This says that the contravariant functor $\coalg{\C}{\D}\to\fb{Set}$
given by $S\mapsto\C(|S|,A)$ is represented by $R\<.$
But, of course, the sets $\C(|S|,A)$ have natural structures
of $\!\D\!$-algebra, $\C(S,A);$ so the above coalgebra $R$
has a $\!\D\!$-algebra structure in $\coalg{\C}{\D}.$
It would be desirable to understand this structure.

\vspace{.5em}
In addition to the above open-ended directions for investigation,
there are several straightforward ways that it should be possible
to generalize the results of this note.
A {\em quasivariety} \cite{PC.UA} is a class of algebras of a given type
defined by identities and universal {\em Horn sentences}, i.e.,
universal equational implications.
(For example, if, to the operations and identities for groups, we
add the implications $(\forall\,x)\ (x^n\<{=}\,e\implies x\,{=}\,e),$
one for each positive integer $n,$ we get
the {\em quasivariety} of torsion-free groups.)
The left sides of these implications are required to be
{\em finite} conjunctions of equations; if this finiteness restriction
is removed, the resulting more general classes
of algebras are called {\em prevarieties}.
As is the case with identities, the satisfaction of such
a universal implication
by all the values of a representable functor $F$ can be
detected in the structure of the representing coalgebra $R:$
An implication of this sort in $\kappa$ variables holds for the values
of $F$ if and only if the joint coequalizer in
$\C$ of the family of pairs of
maps $|R\<|\to\coprod_\kappa |R\<|$ corresponding to the
hypotheses of the implication also coequalizes the
pair of maps corresponding to the conclusion.
One should be able to adapt the material of the preceding sections
to the case where $\C$ and $\D$ are general prevarieties.

Categories of algebras in which some of the operations are
partial, with domain given by the set at which specified relations
in other operations hold (subject to conditions that avoid
vicious circles) should also be amenable to such treatment.
Cf.~\cite[\S6]{tower} for a bit of motivation, and an example.
(In contrast, the category of {\em fields},
where the domain of the inverse operation is the set of elements
where the relation $x=0$ does {\em not} hold, is notoriously
uncooperative from the point of view of universal constructions.)
The same should be true of categories of algebraic structures involving
relations as well as operations (subject to appropriate sorts of
axioms), and to categories of many-sorted algebras.

Indeed, it was this potential embarrassment of riches that led me to
restrict the present development to varieties.

\vspace{.8em}
{\samepage\begin{center}{\bf II. EXAMPLES.}\end{center}

In the remaining sections, we shall compute some explicit
examples of limit coalgebras, especially final coalgebras.

\section{The final object of $\coalg{\Bi}{\Bi}.$}\label{S.BiBi}

In} \S\ref{S.Set} we described the final
co-$\!\Bi\!$ object of $\fb{Set};$
let us now determine the final co-$\!\Bi\!$ object of $\Bi.$
As before, the unique element of $\Omega_{\Bi},$ i.e., the
primitive binary operation of our algebras, will be denoted $\beta.$

Because our ``$\!\C\!$'' is $\Bi,$ which has no identities,
Propositions~\ref{P.compar_iso} and~\ref{P.red} tell us that we can
get our final coalgebra as $P_\omega,$ where $P$ is the
$\!\omega\<{+}1\!$-indexed
precoalgebra built from the trivial pseudocoalgebra, and that
this object $P_\omega$
will be the inverse limit of the $P_k$ with $k\in\omega.$
On the other hand, because our ``$\!\D\!$'' is $\Bi,$
Proposition~\ref{P.dropbase}(iii) tells us that each $(P_{j+1})\ba$
$(j\in\omega)$ can be constructed as $(P_j)\ba\cP(P_j)\ba\<.$
For that result says that the ``product of pseudo-co-operations''
map $\pi^{P_{j+1}}$ is an isomorphism, and with only one
pseudo-co-operation, $\pi^{P_{j+1}}$ is just $\beta^{P_{j+1}}.$

Let us look at this last result in the more general context
where $\D$ is still $\Bi$ and $P_0$ is still the trivial
pseudocoalgebra, but $\C$ is an arbitrary variety.
Let us write $(P_0)\ba$ as $\{x\},$ and
express the fact that for every operation
$\alpha$ of $\C,$ $\alpha(x,\dots,x)=x,$ by calling $x$ ``idempotent''.
Thus $(P_1)\ba,$ the coproduct of two copies of $(P_0)\ba=\{x\},$
will be freely generated as a $\!\C\!$-algebra by two
idempotent elements, which we shall denote $x_0$ and $x_1;$
$(P_2)\ba,$ the coproduct of two copies of
$(P_1)\ba,$ will be freely generated by four idempotent
elements $x_{00},\,x_{01},\,x_{10},\,x_{11},$ etc..
We fix our notation so that for each idempotent generator
$x_a$ of $(P_j)\ba,$ where $a\in\{0,1\}^j,$ the
two coprojections $(P_j)\ba\to(P_j)\ba\cP(P_j)\ba=(P_{j+1})\ba$
take $x_a$ to $x_{0a}$ and $x_{1a}$ respectively.
It is then easy to verify by induction on
$j$ that the maps $(p_{j,\<j+1})\ba:(P_{j+1})\ba\to (P_j)\ba$
act by dropping indices from the other end of our subscripts,
sending generators $x_{a0}$ and $x_{a1}$ of
$(P_{j+1})\ba$ $(a\in\{0,1\}^j)$ to $x_a$ in $(P_j)\ba\<.$
Hence our inverse system of algebras $(P_k)\ba$ $(k\in\omega)$
and their connecting morphisms takes the form
\begin{equation}\begin{minipage}[c]{34pc}\label{x.x->->}
$\dots\ \to\ \{x_{00}\}\cP\{x_{01}\}\cP\{x_{10}\}\cP\{x_{11}\}\ \to\ %
\{x_0\}\cP\{x_1\}\ \to\ \{x\},$
\end{minipage}\end{equation}
where the morphisms act by
\begin{equation}\begin{minipage}[c]{34pc}\label{x.droplast}
$x_{a\,i}\ \mapsto\ x_a$ $(a\in\{0,1\}^j,\ i\in\{0,1\}).$
\end{minipage}\end{equation}

Back, now, to the case $\C=\Bi.$
I claim that in this case one can encode elements
of $(P_1)\ba=\{x_0\}\cP\{x_1\}$ by certain $\!\{0,1\}\!$-valued
functions on the Cantor set.
We shall formally regard the Cantor set as
$\{0,1\}^\omega,$ but for convenience in description, I will often
use language corresponding to the ``middle thirds in the unit
interval'' Cantor set, referring, for instance,
to the ``left'' and ``right'' halves of the set.
(I hope this will be more helpful than confusing.)
Here is the encoding:

The elements $x_0$ and $x_1$ will be represented by the constant
functions $0$ and $1$ respectively.
If two elements $s$ and $t$ are represented by functions $f_s$ and
$f_t$ respectively, let us represent $\beta(s,t)$ by the
function whose graph on the ``left half'' of the Cantor set is
a compressed copy of the graph of $f_s,$ and on the ``right half'',
a compressed copy of the graph of $f_t.$
In terms of strings of $\!0\!$'s and $\!1\!$'s (which we will write
without parentheses or commas), this says that
\begin{equation}\begin{minipage}[c]{34pc}\label{x.beta}
$f_{\beta(s,\<t)}(0a)\ =\ f_s(a),$\quad
$f_{\beta(s,\<t)}(1a)\ =\ f_t(a)$\quad $(a\in\{0,1\}^\omega).$
\end{minipage}\end{equation}
Note that~(\ref{x.beta}) is consistent with the relations
$\beta(x_0,\<x_0)=x_0,$ $\beta(x_1,\<x_1)=x_1;$ and that
the $\!\{0,1\}\!$-valued functions on $\{0,1\}^\omega$ that we get
by recursive application of~(\ref{x.beta}) starting with the
constant functions $0$ and $1$
are {\em continuous} relative to the standard topology on
$\{0,1\}^\omega$ and the discrete topology on $\{0,1\}.$
In fact, this construction yields a bijection between
elements of $\{x_0\}\cP\{x_1\}$ and such continuous functions: this
follows from the facts that every element $y\in|\<\{x_0\}\cP\{x_1\}\<|$
is uniquely expressible as $\beta(s,t),$ where unless $y$
is $x_0$ or $x_1,$ each of $s$ and $t$ has smaller depth
than $y;$ and that every continuous $\!\{0,1\}\!$-valued function
$f$ on the Cantor set is uniquely expressible as in~(\ref{x.beta})
for unique continuous functions $g$ and $h$ in place of $f_s$
and $f_t,$ where unless $f$ is the constant function $0$ or $1,$
these functions are constant on ``coarser'' clopen subsets than~$f.$

In exactly the same way, elements of
$(P_2)\ba=\{x_{00}\}\cP\{x_{01}\}\cP\{x_{10}\}\cP\{x_{11}\}$
may be represented by continuous $\!\{00,\,01,\,10,\,11\}\!$-valued
functions on $\{0,1\}^\omega,$ and so forth.
The connecting morphisms of~(\ref{x.x->->}),~(\ref{x.droplast})
are seen to correspond to the operation of sending
$\!\{0,1\}^{j+1}\!$-valued functions to $\!\{0,1\}^j\!$-valued
functions by dropping the final $0$ or $1$ in the output-symbols.

It is not hard to deduce that on passing to the inverse limit, we get a
bijection between elements of
$(P_\omega)\ba=\limit_{k\in\omega} (P_k)\ba$ and
functions $\{0,1\}^\omega\to\{0,1\}^\omega$ continuous with respect
to the standard topology on both copies of $\{0,1\}^\omega.$
The $\!\Bi\!$-operation on this set is again described
by~(\ref{x.beta}).

Turning to co-operations, we have noted that by
Proposition~\ref{P.dropbase}(iii), for each $j\in\omega,$ the
pseudo-co-operation
$\beta^{P_{j+1}}: (P_{j+1})\ba\to (P_{j+1})_\beta=(P_j)\ba\cP(P_j)\ba$
is an isomorphism, taking each generator of the
form $x_{0a}$ to $q^{(P_j)\ba}_{2,\<0}(x_a),$ and each
generator of the form $x_{1a}$ to $q^{(P_j)\ba}_{2,\<1}(x_a).$
Identifying $(P_{j+1})_\beta$ with $(P_{j+1})\ba$ via this
map, we see that the image of the first pseudocoprojection,
$c^{P_{j+1}}_{\beta,0}:(P_{j+1})\ba\to (P_{j+1})_\beta,$ is
the subalgebra generated by the elements of the form $x_{0a}$
$(a\in\{0,1\}^j),$ and that of the second pseudocoprojection,
$c^{P_{j+1}}_{\beta,1},$
is the subalgebra generated by the elements $x_{1a}.$
Passing to the inverse limit, and using the ``middle thirds''
image of the Cantor set, we see that $q^{P_\omega}_{2,\<0}$ acts
by {\em vertically} compressing the graph of an element of $P_\omega$ so
as to get a function taking values in the lower half of the
Cantor set, while $q^{P_\omega}_{\<2,1}$ similarly compresses it to
a function with values in the upper half.
Regarding $P_\omega$ as a coalgebra $R,$ the co-operation
$\beta^R: |R\<|\to|R\<|\cP|R\<|$ thus works by ``breaking up''
the graph of a general element $y$ of $R$ into segments
in which the value is in the lower half-Cantor-set and segments
in which it is in the upper half-Cantor-set (each segment having for
domain a ``sub-Cantor-set'' consisting of all elements beginning
with some specified string of $\!0\!$-s and $\!1\!$-s), and using
this decomposition to write $y$ as a $\!\beta_{|R\<|}\!$-word in the
corresponding elements of $q^{P_\omega}_{\<2,\<0}(|R\<|)$ and
$q^{P_\omega}_{\<2,1}(|R\<|).$

This completes the description of the universal object $R$!
(That description was sketched, without using the language of
coalgebras, as \cite[last two exercises of \S8.3]{245}.)

\vspace{.5em}
Though the set $||R\<||$ has a more complicated form than
the underlying set of the initial object of $\coalg{\fb{Set}}{\Bi}$
found in \S\ref{S.Set}, note that its cardinality is still that
of the continuum; for a continuous map $\{0,1\}^\omega\to\{0,1\}^\omega$
is determined by its values at a countable dense subset of its domain,
so there are at most $(2^{\aleph_0})^{\aleph_0}=2^{\aleph_0}$ such maps.

A curious property of this example is that not only does the
co-operation $\beta^R$ give an isomorphism of $|R\<|$ with
the coproduct in $\Bi$ of two copies of itself,
as required by Proposition~\ref{P.dropbase}; the operation
$\beta_{|R\<|}$ likewise gives a bijection of $||R\<||$ with
the product in $\fb{Set}$ of two copies of itself.
I do not know whether this is an instance of some general result.

The set $||R\<||$ of continuous maps from the Cantor set to itself
also has a natural monoid structure, given by composition of maps.
I likewise do not know whether this has any interpretation as
``additional structure'' on the coalgebra $R$ we have constructed.

\section{The final object of $\coalg{\Bi}{\Se}.$}\label{S.BiSe}

When we turn to co-$\!\Se\!$ objects of $\Bi,$ we cannot apply
Proposition~\ref{P.dropbase}(iii) directly, since the variety
$\Se$ is defined using a nonempty set of identities; but we
can take the above description of the final $\!\Bi\!$-precoalgebra
and $\!\Bi\!$-coalgebra,
and use Proposition~\ref{P.imagecoalg}(iv) and~(ii)
to obtain the corresponding
final $\!\Se\!$-precoalgebra and coalgebra as substructures thereof.
Let us write $P'$ for the precoalgebra we called $P$ in the preceding
section, and let $P$ here
denote the subprecoalgebra thereof determined by
the condition that $\beta$ cosatisfy the associative identity.

The first instance of that identity arises in $P'_2.$
In constructing the two derived pseudo-co-operation maps that we must
equalize, we in each case begin
by mapping $(P'_2\<)\ba$ into $(P'_1\<)\ba\cP(P'_1\<)\ba$ by
$\beta^{P'_2},$ and then map this, in different ways, into
\begin{equation}\begin{minipage}[c]{34pc}\label{x.lmr}
$(P'_0\<)\ba\,\cP\,(P'_0\<)\ba\,\cP\,(P'_0\<)\ba\<.$
\end{minipage}\end{equation}
Namely, to get one derived pseudo-co-operation,
we map our first $(P'_1\<)\ba$
by $\beta^{P'_1}$ into $(P'_0\<)\ba\cP(P'_0\<)\ba,$
which we identify with the subalgebra of~(\ref{x.lmr})
generated by the first two copies of~$(P'_0\<)\ba,$
while mapping the second $(P'_1\<)\ba$ by
$(p_{0,1})\ba$ into $(P'_0\<)\ba$ which we identify with the third copy.
To get the other derived pseudo-co-operation,
we apply $(p_{0,1})\ba$ to our
first $(P'_1\<)\ba,$ and $\beta^{P'_1}$ to the second,
identifying its codomain with the subalgebra of~(\ref{x.lmr})
generated by the second and third copies.
Denoting the images of the unique element $x$ of $(P'_0\<)\ba$
under the three coprojections into~(\ref{x.lmr})
by $x_\lambda,\,x_\mu,\,x_\rho$ (mnemonic for the {\em left,
middle} and {\em right} terms in the associative law), we find that
our two derived pseudo-co-operations act by
\begin{equation}\begin{minipage}[c]{34pc}\label{x.lmr_r}
$x_{00}\ \mapsto\ x_\lambda,\qquad
x_{01}\ \mapsto\ x_\mu,\qquad
x_{10}\ \mapsto\ x_\rho,\qquad
x_{11}\ \mapsto\ x_\rho,$
\end{minipage}\end{equation}
and
\begin{equation}\begin{minipage}[c]{34pc}\label{x.lmr_l}
$x_{00}\ \mapsto\ x_\lambda,\qquad
x_{01}\ \mapsto\ x_\lambda,\qquad
x_{10}\ \mapsto\ x_\mu,\qquad
x_{11}\ \mapsto\ x_\rho.$
\end{minipage}\end{equation}
respectively.
Note that the only generators of $(P'_2\<)\ba$ on which these
two maps agree are $x_{00}$ and $x_{11}.$
Now the object of $\Bi$ freely generated by any set $X$ of
idempotent elements has a normal form, consisting of all
expressions in these generators with no subexpression which can
be simplified by the idempotence relations, i.e., no
$\beta(x,x)$ with $x\in X.$
Comparing normal forms in the domains and codomains of the
above two maps, it is not hard to
verify that their equalizer will be the subcoalgebra
of $(P'_2\<)\ba$ generated by $x_{00}$ and $x_{11}.$
(Perhaps the easiest way to see this is to regard members of
$P_2',$ as in the preceding section, as $\!\{00,\<01,10,11\}\!$-valued
functions on the Cantor set, and~(\ref{x.lmr_r}),~(\ref{x.lmr_l})
as determining maps from these to
$\!\{\lambda,\<\mu,\rho\}\!$-valued functions on that set.)
Hence $(P_2)\ba$ is the subalgebra $\langle x_{00},\,x_{11}\rangle$
of $(P'_2\<)\ba\<.$

Instances of coassociativity at higher levels similarly give
the condition that for any generator $x_{i_0\dots i_k}$
occurring in the expression for any element of $(P_k)\ba,$
successive indices $i_j,\,i_{j+1}$ must be equal, so
$(P'_j\<)\ba$ is just $\langle x_{0\dots 0},\,x_{1\dots 1}\rangle.$
The limit object is therefore
$\langle x_{0^\infty}, x_{1^\infty}\rangle;$
its co-operation takes $x_{0^\infty}$
to its own image under the first coprojection, and
$x_{1^\infty}$ to its own image under the second coprojection.
The corresponding initial object of $\Rep{\Bi}{\Se}$
can be described as taking each binar, $A,$ to the set of pairs
$\{(a,b)\in|A\<|\times|A\<|\mid a^2{\<=\<}a,\, b^2{\<=\<}b\},$
with operation $(a,b)\cdot(c,d)=(a,d).$
(Like the construction of the initial semigroup-valued
representable functor on $\fb{Set},$ this is a
``rectangular band'' construction.)

\vspace{.5em}
What if we modify the task of the preceding section in the opposite way,
and seek to describe the final object of $\coalg{\Se}{\Bi},$
equivalently, the initial representable functor $\Se\to\Bi$?

Since our new ``$\!\C\!$'', $\Se,$ like $\Bi,$ satisfies~(\ref{x.u=}),
while ``$\!\D\!$'' is again $\Bi,$ we can again say
that the representing object will have for underlying algebra the
inverse limit of the chain of idempotent-generated objects (in this
case, semigroups) written as in~(\ref{x.x->->}) and~(\ref{x.droplast}).
In the preceding section, we used
the ``rigidity'' of the normal form in $\Bi$ to identify elements
of that inverse limit with certain functions on the Cantor set.
No such model is evident when the base-category is $\Se.$
An element of $(P_1)\ba$ can be represented by a
finite alternating string of $\!0\!$'s and $\!1\!$'s;
an element of $(P_2)\ba$ mapping to such an element
is obtained by replacing each of the $\!0\!$'s with a
finite alternating string of $\!00\!$'s and $\!01\!$'s,
and each of the $\!1\!$'s with a finite alternating string
of $\!10\!$'s and $\!11\!$'s; and so on.
But I don't see any ``geometric'' description of the inverse
limit of these semigroups.

Nevertheless, the description of
each $P_k$ $(k\in\omega)$ as a semigroup
freely generated by certain idempotents will prove
useful in the next section, where we will construct the subprecoalgebra
of this $P$ that yields the final object of $\coalg{\Se}{\Se}.$

\section{The final object of $\coalg{\Se}{\Se}.$}\label{S.SeSe}

In our construction of the final object of $\coalg{\Bi}{\Se}$ above,
the pseudocoalgebras $P_j$ stabilized quickly: $P_1$ already
made all the distinctions that were going to be made among elements
of sets $(P_j)\ba,$ and once these propagated up to the pseudocopower
object at the next stage, we could have verified that $P_2$ was a
genuine coalgebra, and was our desired final object.

In constructing the final object of $\coalg{\Se}{\Se},$
our structure will also stabilize early:
All the distinctions among elements of that coalgebra will be present
in $(P_2)\ba\<;$ when we reach $(P_3)\ba$ the process of eliminating
elements at which the associative identity is not
cosatisfied will stabilize, and at the next step, the pseudocopower
object $(P_4)_\beta$ will catch up, making $P_4$ the desired coalgebra.

Let us, in this section, write $P'$ for the $\!\omega\!$-indexed
$\!\Bi\!$-precoalgebra in $\Se$ referred to at the end of the last
section, built up universally on the trivial pseudocoalgebra $P'_0;$
thus the inverse system of semigroups $(P'_k\<)\ba$
is given by~(\ref{x.x->->}) and~(\ref{x.droplast}),
with ``$\!\cP\!$'' interpreted as coproduct of semigroups.
By Proposition~\ref{P.imagecoalg}(iv), the precoalgebra we want is the
largest subprecoalgebra of $P'$ cosatisfying the associative identity.

As before, the first instance of that identity occurs in $P'_2.$
The elements of $(P'_2)\ba$ at which that identity is cosatisfied
are those at which the two homomorphisms into the semigroup freely
generated by three idempotent elements $x_\lambda, x_\mu, x_\rho$ given
by~(\ref{x.lmr_r}) and~(\ref{x.lmr_l}) agree.
However, because semigroups have a less rigid structure than
binars, the equalizer of those two homomorphisms is larger than
the subsemigroup generated by the two elements of
$\{x_{00},\<x_{01},\<x_{10},\<x_{11}\}$ on which they agree.
For instance, it is easy to verify that it
contains the product $x_{00}\<x_{01}\<x_{10}\<x_{11};$
namely, that both maps send this to $x_\lambda\<x_\mu\<x_\rho.$

It would be convenient for the computations to come
if we could say that the common image of an
element of $(P_2)\ba$ under the homomorphisms~(\ref{x.lmr_r})
and~(\ref{x.lmr_l}) was also their image under the homomorphism
\begin{equation}\begin{minipage}[c]{34pc}\label{x.lmr_m}
$x_{00}\ \mapsto\ x_\lambda,\qquad
x_{01}\ \mapsto\ x_\mu,\qquad
x_{10}\ \mapsto\ x_\mu,\qquad
x_{11}\ \mapsto\ x_\rho.$
\end{minipage}\end{equation}
However, this is not true.
The simplest counterexample I know is the element
\begin{equation}\begin{minipage}[c]{34pc}\label{x.mess}
$(x_{00}\ x_{01})^2\ x_{10}\ x_{00}\ x_{01}
\ x_{10}\ x_{11}\ x_{01}\ (x_{10}\ x_{11})^2.$
\end{minipage}\end{equation}
The reader can verify that~(\ref{x.lmr_r})
and~(\ref{x.lmr_l}) both send this to
\begin{quote}
$(x_\lambda\ x_\mu)^2\ x_\rho\ x_\lambda\ (x_\mu\ x_\rho)^2,$
\end{quote}
so that it belongs to $(P_2)\ba,$
but that~(\ref{x.lmr_m}) sends it to the different value
\begin{quote}
$(x_\lambda\ x_\mu)^2\ x_\lambda\ x_\mu\ x_\rho\ (x_\mu\ x_\rho)^2.$
\end{quote}

However, the next result shows that~(\ref{x.lmr_m})
does agree with~(\ref{x.lmr_r}) and~(\ref{x.lmr_l})
on the image of $(p_{2,\<3})\ba:(P_3)\ba\to(P_2)\ba\<;$
indeed, that this is true for the corresponding precosemigroup
in any variety $\C$ of algebras.

\begin{lemma}\label{L.m_also}
Let $\C$ be a variety of algebras, and $P$ the precoalgebra arising
in the construction of the final object of $\coalg{\C}{\Se}.$
Then the homomorphism determined by\textup{~(\ref{x.lmr_m})},
from $(P'_2)\ba$ to the algebra freely generated by
three idempotents $x_\lambda,\<x_\mu,\<x_\rho,$ agrees
on the subalgebra $(p_{3,\<2})\ba((P_3)\ba)\subseteq(P_2)\ba$ with the
common value of the homomorphisms determined
by\textup{~(\ref{x.lmr_r})} and\textup{~(\ref{x.lmr_l})}.
\end{lemma}

\begin{proof}
Since the map $(p_{3,\<2})\ba$ takes each generator
$x_{ijk}$ of $(P'_3)\ba$
to $x_{ij}$ in $(P'_2)\ba,$ what we wish to show is that the
composites of that map with~(\ref{x.lmr_r}) and~(\ref{x.lmr_l}),
namely the maps given on our idempotent generators by
\begin{equation}\begin{minipage}[c]{34pc}\label{x.lmr_r2}
$x_{00k}\ \mapsto\ x_\lambda,\qquad
x_{01k}\ \mapsto\ x_\mu,\qquad
x_{10k}\ \mapsto\ x_\rho,\qquad
x_{11k}\ \mapsto\ x_\rho,$\qquad and
\end{minipage}\end{equation}
\begin{equation}\begin{minipage}[c]{34pc}\label{x.lmr_l2}
$x_{00k}\ \mapsto\ x_\lambda,\qquad
x_{01k}\ \mapsto\ x_\lambda,\qquad
x_{10k}\ \mapsto\ x_\mu,\qquad
x_{11k}\ \mapsto\ x_\rho$
\end{minipage}\end{equation}
(which by what we already know agree on $(P_3)\ba$ with each other),
also agree on $(P_3)\ba$ with the
composite of $(p_{2,\<3})\ba$ with~(\ref{x.lmr_m}), namely, the
map given on generators by
\begin{equation}\begin{minipage}[c]{34pc}\label{x.lmr_m2}
$x_{00k}\ \mapsto\ x_\lambda,\qquad
x_{01k}\ \mapsto\ x_\mu,\qquad
x_{10k}\ \mapsto\ x_\mu,\qquad
x_{11k}\ \mapsto\ x_\rho.$
\end{minipage}\end{equation}

The equality of the maps~(\ref{x.lmr_r2}) and~(\ref{x.lmr_l2})
is the $\!(3,0)\!$-instance of the associative identity in $P;$
a stronger condition is the $\!(3,1)\!$-instance of that identity.
This involves the corresponding maps into
the coproduct of three copies of $(P'_1)\ba\<.$
That coproduct is freely generated by six idempotents:
``$\!\lambda\!$'', ``$\!\mu\!$'' and ``$\!\rho\!$'' copies of
the two idempotent generators $x_0$ and $x_1$ of $(P'_1)\ba\<.$
Let us write these $x_{\lambda 0},\,
x_{\lambda 1},\,x_{\mu 0},\,x_{\mu 1},\,x_{\rho 0},\,x_{\rho 1}.$
Then the $\!(3,1)\!$-instance of the associative law says that the maps
defined on $(P'_3\<)\ba$ by
\begin{equation}\begin{minipage}[c]{34pc}\label{x.lmr_r3}
$x_{00k}\ \mapsto\ x_{\lambda k},\qquad
x_{01k}\ \mapsto\ x_{\mu k},\qquad
x_{10k}\ \mapsto\ x_{\rho 0},\qquad
x_{11k}\ \mapsto\ x_{\rho 1},$\qquad and
\end{minipage}\end{equation}
\begin{equation}\begin{minipage}[c]{34pc}\label{x.lmr_l3}
$x_{00k}\ \mapsto\ x_{\lambda 0},\qquad
x_{01k}\ \mapsto\ x_{\lambda 1},\qquad
x_{10k}\ \mapsto\ x_{\mu k},\qquad
x_{11k}\ \mapsto\ x_{\rho k}$
\end{minipage}\end{equation}
agree on $(P_3)\ba\<.$
Note the loss of the subscript $k$ in the last two terms
of~(\ref{x.lmr_r3}) and the first two terms of~(\ref{x.lmr_l3}).
This happens because, in the definition of cosatisfaction of
the associativity identity $\beta(\beta(-,-),-)=\beta(-,\beta(-,-)),$
the second argument of the outer $\beta$ on the left, and
the first argument of the outer $\beta$ on the right, in order to
be comparable with the other arguments, have to be carried down from
$(P'_2)\ba$ to $(P'_1)\ba$ by $(p_{1,\<2})\ba,$ which acts by
lopping off the last subscript.

Now elements of $(P_3)\ba$ are
restricted not only by the cosatisfaction of associativity on that
object, but by the fact that they lie in the subalgebra generated by
the images of $(P_2)\ba$ in $(P'_2)_\beta=(P'_3)\ba,$
under the two coprojection maps $x_{ij}\mapsto x_{0ij}$
and $x_{ij}\mapsto x_{1ij},$ and that the elements of that subalgebra
$(P_2)\ba$ themselves cosatisfy the associative identity.
To see how to use this fact, let us
look at~(\ref{x.lmr_r3}) as defining a homomorphism on
$(P'_3\<)\ba=(P'_2)\ba\cP(P'_2)\ba,$
and note that it carries the first copy of $(P'_2)\ba$
{\em isomorphically} to the subalgebra of
$(P'_1\<)\ba\cP(P'_1\<)\ba\cP(P'_1\<)\ba$ generated by the four
idempotent elements $x_{\lambda k}$ and$ x_{\mu k},$ merely relabeling
the generators by replacing initial subscript-strings $00$ and $01$ by
$\lambda$ and $\mu$ respectively; while it acts in a non-one-one way on
the second copy; and that the codomain of~(\ref{x.lmr_r3})
is the coproduct of the images of those two copies.
Restricting attention to the action of~(\ref{x.lmr_r3}) on the
subalgebra $(P_3)\ba\subseteq(P'_3\<)\ba,$ let us write down the fact
that the elements that come from the first copy of $(P'_2\<)\ba$
(i.e., expressions where all subscripts begin with ``$\!0\!$'')
are images under $x_{ij}\mapsto x_{0ij}$
of elements cosatisfying the associative law;
this may be expressed as the equality, on
that image, of the maps obtained from~(\ref{x.lmr_r})
and~(\ref{x.lmr_l}) by prefixing a ``$\!0\!$'' to all subscripts.
Taking into account the relabeling, referred to above,
of that copy of $(P'_2\<)\ba$
via the first two arrows of~(\ref{x.lmr_r3}),
and noting that we are not making any changes
in the image of the second copy of $(P'_2\<)\ba,$ this says that on
the image of $(P_3)\ba$ under~(\ref{x.lmr_r3}), the homomorphisms
into the $\!\C\!$-algebra freely generated by idempotent elements
$x_{0\lambda},\ x_{0\mu},\ x_{0\rho},\ x_{\rho 0},\ x_{\rho 1}$
given by
\begin{equation}\begin{minipage}[c]{34pc}\label{x.lmr_r4}
$x_{\lambda 0}\ \mapsto\ x_{0\lambda},\qquad
x_{\lambda 1}\ \mapsto\ x_{0\mu},\qquad
x_{\mu k}\ \mapsto\ x_{0\rho},\qquad
x_{\rho k}\ \mapsto\ x_{\rho k},$\qquad and
\end{minipage}\end{equation}
\begin{equation}\begin{minipage}[c]{34pc}\label{x.lmr_l4}
$x_{\lambda k}\ \mapsto\ x_{0\lambda},\qquad
x_{\mu 0}\ \mapsto\ x_{0\mu},\qquad
x_{\mu 1}\ \mapsto\ x_{0\rho},\qquad
x_{\rho k}\ \mapsto\ x_{\rho k}$
\end{minipage}\end{equation}
agree.

Let us now collapse the codomain
$\langle x_{0 \lambda},\<x_{0 \mu},\<
x_{0 \rho},\<x_{\rho 0},\<x_{\rho 1}\rangle$
of~(\ref{x.lmr_r4}) and~(\ref{x.lmr_l4})
onto the algebra $\langle x_\lambda,\<x_\mu,\<x_\rho\rangle$ freely
generated by three idempotent elements, by
the homomorphism sending $x_{0\lambda}$ to $x_\lambda,$
$x_{0\mu}$ and $x_{0\rho}$ both to $x_\mu,$ and
$x_{\rho 0}$ and $x_{\rho 1}$ both to $x_\rho.$
Then~(\ref{x.lmr_r4}) and~(\ref{x.lmr_l4}) yield maps
\begin{equation}\begin{minipage}[c]{34pc}\label{x.lmr_r5}
$x_{\lambda 0}\ \mapsto\ x_\lambda,\qquad
x_{\lambda 1}\ \mapsto\ x_\mu,\qquad
x_{\mu k}\ \mapsto\ x_\mu,\qquad
x_{\rho k}\ \mapsto\ x_\rho$\qquad and
\end{minipage}\end{equation}
\begin{equation}\begin{minipage}[c]{34pc}\label{x.lmr_l5}
$x_{\lambda k}\ \mapsto\ x_\lambda,\qquad
x_{\mu k}\ \mapsto\ x_\mu,\qquad
x_{\rho k}\ \mapsto\ x_\rho,$
\end{minipage}\end{equation}
which still agree on the image of $(P_3)\ba$ under~(\ref{x.lmr_r3}).
But on $(P_3)\ba,$~(\ref{x.lmr_r3}) agrees
with~(\ref{x.lmr_l3}), so we will get the same results
by following~(\ref{x.lmr_r3}) with~(\ref{x.lmr_l5}),
and by following~(\ref{x.lmr_l3}) with~(\ref{x.lmr_r5}), namely
\begin{quote}
$x_{00k}\ \mapsto\ x_\lambda,\qquad
x_{01k}\ \mapsto\ x_\mu,\qquad
x_{10k}\ \mapsto\ x_\rho,\qquad
x_{11k}\ \mapsto\ x_\rho,$\qquad and
\end{quote}
\begin{quote}
$x_{00k}\ \mapsto\ x_\lambda,\qquad
x_{01k}\ \mapsto\ x_\mu,\qquad
x_{10k}\ \mapsto\ x_\mu,\qquad
x_{11k}\ \mapsto\ x_\rho.$
\end{quote}
These are precisely~(\ref{x.lmr_r2}) and~(\ref{x.lmr_m2}),
completing the proof of the lemma.
\end{proof}

{\em Remarks.}
The above result presumably has something to do with the identity
\begin{quote}
$(ab)(cd)\ =\ a(bc)d$
\end{quote}
in $\Se,$ where the right-hand side means the common value
of $(a(bc))d$ and $a((bc)d).$
But while that relation is trivial to derive, I was only able to find
the proof of the above lemma by ``enlightened trial and error''.
Some illuminating way of looking at such computations is
clearly desirable.

The proof of the lemma did not use the assumption that the $x_{ijk}$
were idempotent, so one really should prove it in a wider context.
But I did not want to complicate further the already messy proof.

We now return to the case $\C=\Se.$
The above lemma motivates

\begin{definition}\label{D.R}
$|R\<|\subseteq (P'_2)\ba$ will denote the subsemigroup of elements
at which the maps \textup{(\ref{x.lmr_r})},
\textup{(\ref{x.lmr_l})} and~\textup{(\ref{x.lmr_m})} all agree.
\end{definition}

In the remainder of this section,
we shall obtain a precise description of $|R\<|,$
and show that the precosemigroup structure on $P$ induces
a structure of cosemigroup on $|R\<|,$ and that the resulting
cosemigroup $R$ is the final cosemigroup we are seeking.
We will then examine the functor $\Se\to\Se$ that it represents.

The reader who intends to follow the calculations of
the next few pages in detail would do well to copy onto a slip of paper
displays (\ref{x.lmr_r}), (\ref{x.lmr_l}) and (\ref{x.lmr_m})
(perhaps in abbreviated form, e.g., showing only the subscripts),
or memorize them with their display-numbers, since we will refer to
them repeatedly.
(Also, if at some point the reader feels the need
for examples of elements of $|R\<|$
to help in thinking about the assertions we will be proving, he or
she can look ahead at~(\ref{x.p_i}), (\ref{x.q_i}), and the paragraph
preceding Lemma~\ref{L.genrel}.)

\begin{lemma}\label{L.splitR}
Let $r$ be an element of $|R\<|,$ written in the normal form
for $(P'_2\<)\ba\<,$ i.e., as a word in
$x_{00},\,x_{01},\,x_{10},\,x_{11}$ with none of these
generators occurring twice in succession.
Then\\[.5em]
\textup{(i)}\ \ The expression for $r$ begins with an $x_{00}$
or $x_{11},$ and ends with an $x_{00}$ or $x_{11}.$\\[.5em]
\textup{(ii)}\ \ In the expression for $r,$ there are no occurrences of
$x_{00}$ adjacent to $x_{10}$ in either order,
nor of $x_{11}$ adjacent to $x_{01}$ in either order.\\[.5em]
\textup{(iii)}\ \ If $r$ has two occurrences of $x_{00}$ with
no $x_{00}$ or $x_{11}$ between them, then what
occurs between them is precisely the length-$\!1\!$ string $x_{01}.$
Likewise, if $r$ has two occurrences of $x_{11}$ with
no $x_{00}$ or $x_{11}$ between them, then what
occurs between them is precisely $x_{10}.$\\[.5em]
\textup{(iv)}\ \ If $r$ has the form $s\,x_{00}\,x_{11}\,t$
for some possibly empty strings $s$ and $t,$ then the
factors $s\,x_{00}$ and $x_{11}\,t$ also belong to $|R\<|.$
Likewise, if $r$ has the form $s\,x_{11}\,x_{00}\,t,$
then $s\,x_{11}$ and $x_{00}\,t$ belong to $|R\<|.$\\[.5em]
\textup{(v)}\ \ If $r$ has length $>1,$ then its expression includes
occurrences of both $x_{00}$ and $x_{11}.$
\end{lemma}

\begin{proof}
(i): If $r$ began with $x_{01},$ then its image
under~(\ref{x.lmr_l}) would begin with $x_\lambda,$
while its image under~(\ref{x.lmr_r}) would
begin with $x_\mu,$ a contradiction.
The same reasoning applies with ``end'' in place of ``begin'', and
the analogous computation excludes elements that begin or end
with $x_{10}.$

To get the remaining assertions, let us
first note that when either~(\ref{x.lmr_r})
or~(\ref{x.lmr_m}) is applied to $r,$ only the generator
$x_{00}$ is mapped to $x_\lambda;$ hence for any $i,$
the $\!i\!$-th occurrence of $x_{00}$ in $r$ (if such exists)
yields the $\!i\!$-th occurrence of $x_\lambda$ in the image element;
hence the initial string of $r$ up to (respectively, through) the
$\!i\!$-th occurrence of $x_{00}$ is mapped
by {\em both}~(\ref{x.lmr_r}) and~(\ref{x.lmr_m}) to
the initial string of the image of $r$ up to (respectively, through)
the $\!i\!$-th occurrence of $x_\lambda;$ and the same holds for
the terminal strings following (respectively beginning with)
the $\!i\!$-th occurrences of $x_{00}$ and $x_\lambda.$
These conclusions are not true of the images
of these substrings under~(\ref{x.lmr_l}); rather,
if we write ``(\ref{x.lmr_l}) and~(\ref{x.lmr_m})''
in place of ``(\ref{x.lmr_r}) and~(\ref{x.lmr_m})'',
we get the analogous results with $x_{11}$ and $x_\rho$
in place of $x_{00}$ and $x_\lambda.$

To see (ii) now, simply observe that if the $\!i\!$-th occurrence
of $x_{00}$ in $r$ were immediately preceded by
$x_{10},$ then in the image of $r$ under~(\ref{x.lmr_r}), the $\!i\!$-th
occurrence of $x_\lambda$ would be immediately preceded
by $x_\rho,$ while in the image under~(\ref{x.lmr_m}) it
would be immediately preceded by $x_\mu,$ a contradiction.
The statement for $x_{10}$ immediately following $x_{00},$ and
the corresponding statements for $x_{11}$ and $x_{01},$
are seen in the same way.

To see (iii), suppose that the $\!i\!$-th and
$\!i\<{+}1\!$st occurrences of $x_{00}$ have no $x_{11}$ between them.
Since they cannot be adjacent (our expression for $r$ being in normal
form), they have a nonempty string of $\!x_{01}\!$'s and $\!x_{10}\!$'s
between them.
Hence in the image of $r$ under~(\ref{x.lmr_m}), the $\!i\!$-th and
$\!i\<{+}1\!$st occurrences of $x_\lambda$ have
precisely an $x_\mu$ between them;
but the only way we can get this in the image under~(\ref{x.lmr_r})
is if they have only a single $x_{01}$ between them.
Again, the second statement is proved in the same way.

To get (iv), suppose the $\!i\!$-th occurrence of $x_{00}$ in $r$ is
immediately followed by the $\!j\!$-th occurrence of $x_{11},$
and we factor $r$ at this point, writing $r=s\,x_{00}\,x_{11}\,t.$
Applying~(\ref{x.lmr_m}), we see that the
$\!i\!$-th occurrence of $x_\lambda$ in the image element is
immediately followed by the $\!j\!$-th occurrence of $x_\rho,$
giving a factorization of that image as $u=v\,x_\lambda\,x_\rho\,w.$
Here $v\,x_\lambda,$ the image of $s\,x_{00}$ under~(\ref{x.lmr_m}),
is both the initial string of $u$ through the
$\!i\!$-th occurrence of $x_\lambda,$ and the initial
string up to (but not including) the $\!j\!$-th occurrence of $x_\rho.$
By our earlier observations, the former characterization also makes
this element the image of $s\,x_{00}$ under~(\ref{x.lmr_r}), while
the latter makes it the image of $s\,x_{00}$ under~(\ref{x.lmr_l}).
Hence those three images of $s\,x_{00}$ are equal;
so $s\,x_{00}\in|R\<|.$
The statement that $x_{11}\,t\in|R\<|$ is obtained in the same way,
as are the corresponding results when $x_{00}$ follows $x_{11}.$

To show~(v), suppose first that $r$ does not contain $x_{11}.$
Then~(i) and~(iii) imply that has the form $x_{00}(x_{01}x_{00})^i$
for some $i\geq 0.$
The images of this element under~(\ref{x.lmr_r})
and~(\ref{x.lmr_l}) are $x_\lambda(x_\mu x_\lambda)^i$
and $x_\lambda$ respectively, so $i=0,$ showing
that $r$ has length~$1.$
We get the same conclusion if $r$ does not contain $x_{00},$
completing the proof of~(v).
\end{proof}

Let us now use the above tools
to dig our way into the structure of an element $r\in|R\<|.$
By~(i) above, such an element must begin with $x_{00}$ or $x_{11}.$
Assume the former without loss of generality.
That may be all of $r,$ for it is easy to see that $x_{00}\in|R\<|.$
If it is not all of $r,$ then by~(ii), the following factor must be
$x_{11}$ or $x_{01}.$
In the former case, writing $r=x_{00}\,x_{11}\,t,$ we know by~(iv)
that each of the factors $x_{00}$ and $x_{11}\,t,$ lies in $|R\<|,$
and we are reduced to studying elements with shorter expressions.
So suppose the second factor is $x_{01}.$
This may be followed either by another $x_{00}$ or by an $x_{10}.$
In the former case, this new $x_{00}$
must again be followed by $x_{01}.$
(It can't be terminal by~(v), and it can't be followed by $x_{11},$
because if it were, then by~(iv) the product of the terms up to that
point, $x_{00}\,x_{01}\,x_{00},$ would belong to $|R\<|,$ again
contradicting~(v).)
Repeating these considerations, we see that $r$ will begin with some
string $(x_{00}\,x_{01})^{i+1}$ $(i\in\omega),$ followed by an $x_{10}.$

This $x_{10}$ cannot terminate $r$ by~(i), so it must be followed
by $x_{11}$ or $x_{01}.$
If we get $x_{01},$ then the next term can only be another $x_{10};$
the apparent alternative $x_{00}$ is ruled out by the fact
that such an $x_{00}$ and the preceding $x_{00}$ would have
$x_{01}\,x_{10}\,x_{01}$ between them, contradicting~(iii).
Repeating this argument, we get a (possibly empty)
string of alternating $\!x_{10}\!$'s and $\!x_{01}\!$'s,
finally followed by $x_{10}\,x_{11};$ i.e., $r$ begins
\begin{quote}
$(x_{00}\,x_{01})^{i+1}(x_{10}\,x_{01})^j(x_{10}\,x_{11})$
\quad$(i,j\in\omega).$
\end{quote}
(Note that if $j=0,$ the ``$\!(x_{10}\,x_{01})^j\!$'' in this
expression is not a semigroup element, but merely
means that nothing is inserted between the
$(x_{00}\,x_{01})^{i+1}$ and the $(x_{10}\,x_{11}).$
This interpretation of possibly zero exponents
will be in effect throughout this section.)

If the above $x_{11}$ is not the final term of $r,$ then it must be
followed either by $x_{00}$ (in which case we can again reduce to the
study of two shorter elements of $|R\<|$ using (iv)), or by
another $x_{10}.$
Let us collect as large a power of $x_{10}\,x_{11}$ as we can, and
note what the next two factors, if any, can be; we conclude that
if the element we have been considering cannot be factored
at an $x_{00}$-$x_{11}$ interface, then it must have the form
\begin{equation}\begin{minipage}[c]{34pc}\label{x.start_p}
$r\ =\ (x_{00}\,x_{01})^{i+1}(x_{10}\,x_{01})^j
(x_{10}\,x_{11})^{k+1}\,s,$ where $i,j,k\in\omega,$ and $s$ is either
the empty string, or begins with $x_{10}\,x_{01}.$
\end{minipage}\end{equation}

Let us now apply (\ref{x.lmr_r}) and (\ref{x.lmr_l})
to this description~(\ref{x.start_p}) of $r.$

The image of~(\ref{x.start_p}) under~(\ref{x.lmr_r}) begins
with $(x_\lambda\,x_\mu)^{i+1}(x_\rho\,x_\mu)^j\,x_\rho$
(where, in evaluating the image of $(x_{10}\,x_{11})^{k+1},$ we
have used the idempotence of $x_\rho),$ and the last clause
of~(\ref{x.start_p}) shows that if this is not the whole of
that image, it is followed by another $x_\mu.$

On the other hand, the image of~(\ref{x.start_p})
under~(\ref{x.lmr_l}) begins
$x_\lambda\,(x_\mu\,x_\lambda)^j(x_\mu\,x_\rho)^{k+1}$
(where in evaluating
the first part, we use the idempotence of $x_\lambda),$
possibly followed by a string beginning $x_\mu\,x_\lambda.$

Equating the powers of $x_\lambda\,x_\mu$ with which these
images begin, we see that $j=i.$
If we then compare the number of terms $x_\rho$ occurring
before each image either ends or shows another $x_\lambda,$ we see that
this is at least $j\<{+}1$ in the first case and exactly $k\<{+}1$
in the second, so $k\geq j=i;$
so the term $(x_{10}\,x_{11})^{k+1}$ of~(\ref{x.start_p})
has a left factor $(x_{10}\,x_{11})^{i+1}.$
We conclude that $r$ begins with a factor to which we give the name
\begin{equation}\begin{minipage}[c]{34pc}\label{x.p_i}
$p_i\ =\ (x_{00}\,x_{01})^{i+1}(x_{10}\,x_{01})^i
(x_{10}\,x_{11})^{i+1}$\quad $(i\in\omega).$
\end{minipage}\end{equation}

We remark that if we write the final factor $x_{01}$ of the
initial string $(x_{00}\,x_{01})^{i+1}$ in~(\ref{x.p_i}) twice
(as we may since all our generators are idempotent),
and likewise the initial factor $x_{10}$ of the
terminal string $(x_{10}\,x_{11})^{i+1},$ and combine these
with the middle string, then the expression assumes
the non-reduced, but somewhat more elegant form
\begin{equation}\begin{minipage}[c]{34pc}\label{x.p_i_alt}
$(x_{00}\,x_{01})^{i+1}(x_{01}\,x_{10})^{i+1}(x_{10}\,x_{11})^{i+1}.$
\end{minipage}\end{equation}

It is straightforward to verify that the element $p_i$ described
by~(\ref{x.p_i}), equivalently,~(\ref{x.p_i_alt}),
lies in $|R\<|:$ its image
under each of (\ref{x.lmr_r}), (\ref{x.lmr_l}) and (\ref{x.lmr_m}) is
\begin{equation}\begin{minipage}[c]{34pc}\label{x.p_i_sig}
$(x_\lambda\,x_\mu)^{i+1}\,(x_\mu\,x_\rho)^{i+1}$
\end{minipage}\end{equation}
(where again, for elegance, I am using a non-reduced
expression, with the middle occurrence of $x_\mu$ repeated).

What about the remaining factor of $r,$ if any?
If it is nonempty,
let us write $r=p_i\,s$ and apply (\ref{x.lmr_r}), (\ref{x.lmr_l})
and (\ref{x.lmr_m}), calling the images of $s$ under these
three maps $s_1,\ s_2$ and $s_3,$ respectively.
Then the fact that $p_i\,s\in|R\<|$ tells us that
\begin{quote}
$(x_\lambda\,x_\mu)^{i+1}\,(x_\mu\,x_\rho)^{i+1}\ s_1\ =\ %
(x_\lambda\,x_\mu)^{i+1}\,(x_\mu\,x_\rho)^{i+1}\ s_2\ =\ %
(x_\lambda\,x_\mu)^{i+1}\,(x_\mu\,x_\rho)^{i+1}\ s_3.$
\end{quote}

This does not imply that $s_1=s_2=s_3,$ since the idempotent
$x_\rho$ to the left of these three elements may mask distinctions
between them.
But it clearly implies that $x_\rho\,s_1=x_\rho\,s_2=x_\rho\,s_3,$
hence that $x_{11}\,s\in|R\<|.$
Thus, if we write $r$ as $p_i\cdot x_{11}\,s,$ using the fact that
$p_i$ ends with the idempotent element $x_{11},$ then the factors
$p_i$ and $x_{11}\,s$ lie in $|R\<|,$ and we can apply the same
considerations to the second of these, which is shorter than $r.$

We assumed above that $r$ began with $x_{00}.$
If instead it begins with $x_{11},$ then the nontrivial case
is when this is followed by something other than $x_{00},$ and in
that case, by symmetry (interchanging subscripts $0$ and $1,$
subscripts $\lambda$ and $\rho,$ and applications of
(\ref{x.lmr_r}) and (\ref{x.lmr_l}) throughout the preceding
argument), we get an initial string which we name
\begin{equation}\begin{minipage}[c]{34pc}\label{x.q_i}
$q_i\ =\ (x_{11}\,x_{10})^{i+1}(x_{01}\,x_{10})^i
(x_{01}\,x_{00})^{i+1}$\quad $(i\in\omega),$
\end{minipage}\end{equation}
which likewise belongs to $|R\<|,$ and, if it is not the whole of $r,$
can be split off so as to leave a shorter factor in $|R\<|.$

We conclude that $|R\<|$ is generated by the elements
$x_{00},\ x_{11},\ p_i$ and $q_i$ $(i\in\omega).$
It is easy to check the relations these satisfy
in $(P'_2)\ba,$ and deduce

\begin{lemma}\label{L.genrel}
The semigroup $|R\<|$ has the presentation
\begin{equation}\begin{minipage}[c]{34pc}\label{x.genrel}
$\langle\<x_{00},\,x_{11},\,p_i,\,q_i\ (i{\<\in\<}\omega)\mid%
x_{00}^2{\<=\<}x_{00},\ x_{11}^2{\<=\<}x_{11},\ %
x_{00}\,p_i{\<=\<}p_i{\<=\<}p_i\,x_{11},\ %
x_{11}\<q_i{\<=\<}q_i{\<=\<}q_i\,x_{00}\<\rangle.$
\end{minipage}\end{equation}

It has a normal form consisting of all strings in the indicated
generators which contain no substrings $x_{00}^2,\ x_{11}^2,\ %
x_{00}\,p_i,\ p_i\,x_{11},\ x_{11}\,q_i$ or $q_i\,x_{00}.$\endproof
\end{lemma}

\vspace{.5em}
Our next task is to find a natural co-operation on $|R\<|.$
We know that every object $S$ of $\coalg{\Se}{\Se}$ has a unique
family of pseudocoalgebra maps $S\to P_k\subseteq P'_k$
$(k\in\omega)$ making commuting
triangles with the connecting maps $p_{j,\<j+1},$ and that for every
such $S,$ the map $|S|\to(P'_2)\ba$ lands in our subsemigroup $|R\<|.$
We shall now show that every
$r\in|R\<|$ lifts under $(p_{2,\<3})\ba$ to a unique
$r'\in|(P'_3\<)\ba|$ which has a chance of being in the image
of the map $|S|\to(P'_3\<)\ba$ for such an $S.$
We will find that this yields an isomorphism $\varphi$
of $|R\<|$ with a subsemigroup $|R\<|'\subseteq(P'_3\<)\ba\<;$
thus, $\varphi$ composed with the pseudo-co-operation
of $P'_3$ gives a co-operation $\beta^R:|R\<|\to|R\<|\cP|R\<|.$
We will verify that $\beta^R$ cosatisfies the associative identity,
and deduce that the resulting coalgebra $R=(|R\<|,\beta^R)$
is in fact the final object of $\coalg{\Se}{\Se}.$

So let $r$ be any element of $|R\<|.$
By the description of the map $(p_{2,\<3})\ba: (P'_3)\ba\to(P'_2)\ba,$
the general element $r'$ of $(P'_3)\ba$ that maps to $r$ is
obtained from $r$ by replacing each term $x_{ij}$ $(i,j\in\{0,1\})$
by a nonempty (possibly length-$\!1\!)$
alternating string of factors $x_{ij0}$ and $x_{ij1}.$
Moreover, since any element of $(P'_2\<)\ba$ that
lies in the image of a cosemigroup $S$ belongs to $|R\<|,$
any element $r'$ of $(P'_3\<)\ba$ in the image of
a cosemigroup must have the property that the pseudo-co-operation
of $P'_3$ carries it into the subsemigroup
$|R\<|\cP|R\<|\subseteq (P'_2\<)\ba\cP(P'_2\<)\ba=(P'_3)_\beta.$
Explicitly, this means that if we break $r'$ up as a product of
substrings formed from generators $x_{0jk}$ alternating with
substrings formed from generators $x_{1jk},$
each such substring is formed from an element of $|R\<|$
by prefixing the index $0,$ respectively $1,$ to all subscripts.

Consider now the case where the expression for $r$ begins with the
string $p_i,$ for some $i\in\omega.$
In particular, its initial substring consisting of generators
with first subscript $0$ is exactly $(x_{00}\,x_{01})^{i+1}.$
Hence for any $r'\in (P'_3)\ba$ that maps to $r,$
the initial substring of $r'$ consisting of generators
with first subscript $0$ is the
result of prefixing $0$ to the subscript of each factor in an element
$u$ of $|R\<|$ which, when broken up into words in
$x_{00}$ and $x_{01},$ alternating with words in
$x_{10}$ and $x_{11},$ yields $i\<{+}1$ of each sort,
beginning with one of the former sort.

Assuming that $r'$ lies in the image of a
cosemigroup, the associative identity must be cosatisfied
at $r',$ hence the results of applying~(\ref{x.lmr_r3})
and~(\ref{x.lmr_l3}) to $r'$ will agree.
(The reader might copy those two displays onto
a slip of paper for use in reading the next few paragraphs --
that will be the last such slip of paper needed.)
Now when we apply~(\ref{x.lmr_l3}), each string of factors
$x_{000}$ and $x_{001}$ collapses to one term $x_{\lambda 0}$
and each string of factors $x_{010}$ and $x_{011}$
collapses to $x_{\lambda 1},$ so the image of $r'$ under that map
begins with $(x_{\lambda 0}\,x_{\lambda 1})^{i+1},$ followed by
a term with a different subscript.
But for the image of $r'$ under the other
map,~(\ref{x.lmr_r3}), to have this form,
we see that $r'$ must begin with precisely $(x_{000}\,x_{001})^{i+1},$
followed by a term with a different subscript.
Hence the element of $|R\<|$ to whose subscripts we prefix
a $0$ to get the initial substring of $r'$ must begin with
$(x_{00}\,x_{01})^{i+1}$ but no higher power of $x_{00}\,x_{01}.$
From our study of elements of $|R\<|,$ we can therefore say that
this element begins with $p_i.$

This says that $r'$ must begin
\begin{equation}\begin{minipage}[c]{34pc}\label{x.start_r'}
$(x_{000}\,x_{001})^{i+1}(x_{010}\,x_{001})^i(x_{010}\,x_{011})^{i+1},$
\end{minipage}\end{equation}
and applying~(\ref{x.lmr_r3}) and~(\ref{x.lmr_l3}) to this in turn,
we get, from the former, $(x_{\lambda 0}\,x_{\lambda 1})^{i+1}
(x_{\mu 0}\,x_{\lambda 1})^i\nolinebreak[2]
(x_{\mu 0}\,x_{\mu 1})^{i+1},$
but from the latter (after collapsing repeated factors),
simply $(x_{\lambda 0}\,x_{\lambda 1})^{i+1}.$
Thus, on applying~(\ref{x.lmr_l3}), the terms
immediately after~(\ref{x.start_r'}) in $r'$ must yield
$(x_{\mu 0}\,x_{\lambda 1})^i\nolinebreak[2]
(x_{\mu 0}\,x_{\mu 1})^{i+1}.$
Noting how~(\ref{x.lmr_l3}) acts, we conclude that
those terms begin with an $\!i\!$-fold alternation
of single terms $x_{100}$ with alternating products of $x_{010}$
and $x_{011},$ followed at the end by $(x_{100}\,x_{101})^{i+1}.$
(A priori, there might also be an initial alternating
product of $x_{010}$ and $x_{011}$ before the first term $x_{100},$
whose image after applying~(\ref{x.lmr_l3}) would be absorbed
by the preceding $x_{\lambda 1}.$
But this would mean that the element of $|R\<|$ to whose
subscripts one prefixes $0$ to get the initial terms of $r'$
would consist of $p_i$ followed by further terms $x_{10}$ and $x_{11},$
and by our description of $|R\<|,$ this is impossible.)

Now each of the alternating products of $x_{010}$
and $x_{011}$ mentioned in the preceding paragraph must be the image
of an element $s\in|R\<|$ under prefixing $0$ to all subscripts; so
each such $s$ must consist only of factors $x_{10}$
and $x_{11},$ so by Lemma~\ref{L.splitR}(v), $s=x_{11}.$
So the additional terms we have found on our element $r'$ are
$(x_{100}\,x_{011})^i(x_{100}\,x_{101})^{i+1}.$

Another cycle of the same sort (which the reader is invited
to work out) adds on a factor $(x_{110}\,x_{101})^i\,x_{110}.$
(Just as the loss of the subscript $k$ at the left-hand
side of~(\ref{x.lmr_l3}) started this process moving,
by causing the images of~(\ref{x.start_r'}) under these two
maps to have unequal lengths, so
the loss of that subscript at the right-hand
side of~(\ref{x.lmr_r3}) slows it down at this point, so
that we add only $i\<{+}1$ terms rather than $2i\<{+}1$ as before.)
Still another cycle tacks on a factor $(x_{111}\,x_{110})^i\,x_{111},$
and there the process stops: the expression so obtained,
\begin{equation}\begin{minipage}[c]{35pc}\label{x.p'_i}
$p'_i\ =\\
\hspace*{-.5em}(x_{000}\<x_{001})^{i+1}
(x_{010}\<x_{001})^i(x_{010}\<x_{011})^{i+1}
(x_{100}\<x_{011})^i(x_{100}\<x_{101})^{i+1}
(x_{110}\<x_{101})^i(x_{110}\<x_{111})^{i+1},$
\end{minipage}\end{equation}
has the same images under~(\ref{x.lmr_r3}) and~(\ref{x.lmr_l3});
moreover, it maps precisely to $p_i$ under $(p_{3,\<2})\ba.$

(We remark, in passing, that, paralleling~(\ref{x.p_i_alt}),
we can rewrite~(\ref{x.p'_i}) as
\vspace{-1em}
\begin{equation}\begin{minipage}[c]{35pc}\label{x.p'_i_alt}
\vrule width0pt height2em depth0em
$p'_i\ =\\
\hspace*{-1.5em}(x_{000}x_{001})^{i+1}
(x_{001}x_{010})^{i+1}(x_{010}x_{011})^{i+1}
(x_{011}x_{100})^{i+1}(x_{100}x_{101})^{i+1}
(x_{101}x_{110})^{i+1}(x_{110}x_{111})^{i+1}$
\end{minipage}\end{equation}
showing that our precosemigroup has
somewhere learned how to count from $0$ to $7$ in base $2).$

Thus, if $r=p_i\,s,$ then we can write $r'=p'_i\,s'$
for $p'_i$ as in~(\ref{x.p'_i}), and some $s'.$
As noted earlier, $s$ may or may not belong to $|R\<|;$ similarly,
$s'$ may or may not inherit the properties that~(\ref{x.lmr_r3})
and~(\ref{x.lmr_l3}) agree on it and that $\beta^{P'_3}$
carries it into $|R\<|\cP|R\<|;$ but if we repeat the idempotent
generator $x_{11}$ at the right end of $p_i,$ and the $x_{111}$
at the right end of $p'_i,$ and attach the extra copies to $s$ and $s'$
respectively, we get factorizations $r=p_i\cdot x_{11} s$
and $r'=p'_i\cdot x_{111} s',$ for which it is immediate to verify
that both factors inherit these properties, and
that $(p_{2,\<3})\ba$ carries the latter factorization to the former.

By symmetry, what we have proved for $p_i$ and $p'_i$ likewise holds
for $q_i$ and the element
\begin{equation}\begin{minipage}[c]{35pc}\label{x.q'_i}
$q'_i\ =\\
\hspace*{-.5em}(x_{111}\<x_{110})^{i+1}
(x_{101}\<x_{110})^i(x_{101}\<x_{100})^{i+1}
(x_{011}\<x_{100})^i(x_{011}\<x_{010})^{i+1}
(x_{001}\<x_{010})^i(x_{001}\<x_{000})^{i+1}$
\end{minipage}\end{equation}
(for which
\vspace{-1em}
\begin{equation}\begin{minipage}[c]{35pc}\label{x.q'_i_alt}
\vrule width0pt height2em depth0em
$q'_i =\\
\hspace*{-1.5em}(x_{111}x_{110})^{i+1}
(x_{110}x_{101})^{i+1}(x_{101}x_{100})^{i+1}
(x_{100}x_{011})^{i+1}(x_{011}x_{010})^{i+1}
(x_{010}x_{001})^{i+1}(x_{001}x_{000})^{i+1}$
\end{minipage}\end{equation}
is again a variant form).
The reader can quickly verify the easier cases we have passed over:
that if the expression for $r$ as in Lemma~\ref{L.genrel} begins,
not with a $p_i$ or a $q_i,$ but with $x_{00}$ or $x_{11},$
then $r'$ will begin with $x_{000}$ or $x_{111},$ and that if
these are not all of $r$ and $r',$ we can peel them off so as
to leave shorter elements still having the properties in question.

Complementing the hard work above with some
easier calculations, we get

\begin{lemma}\label{L.lift}
For each $r\in||R\<||$ there exists a unique element
$\varphi(r)\in|(P'_3\<)\ba|$ satisfying the following three
conditions:\\[.5em]
\textup{(i)}~~$(p_{2,\<3})\ba(\varphi(r))=r.$\\[.5em]
\textup{(ii)}~~The homomorphisms\textup{~(\ref{x.lmr_r3})}
and\textup{~(\ref{x.lmr_l3})} agree on $\varphi(r).$\\[.5em]
\textup{(iii)}~~$\beta_3^P(\varphi(r))\in|R\<|\cP|R\<|.$

\vspace{.5em}
The resulting map $\varphi: |R\<|\to (P'_3\<)\ba$ is an embedding
of semigroups, and sends the generators
$x_{00},\ x_{11},\ \linebreak[0]p_i,\ q_i$ of $|R\<|$
respectively to the elements
$x_{000},\ x_{111},\ p'_i,\ q'_i$ of $(P'_3\<)\ba\<.$
\end{lemma}

\begin{proof}
It is clear that $x_{000},\ x_{111},\ p'_i,\ q'_i$ satisfy
the relations among $x_{00},\ x_{11},\ p_i,\ q_i$ comprising
the presentation~(\ref{x.genrel}) of $|R\<|,$ so mapping the latter
to the former, we get a homomorphism $\varphi:|R\<|\to(P'_3)\ba\<.$

Our preceding discussion shows that $\varphi(r)$ is the only
element of $(P'_3\<)\ba$ that can satisfy~(i),~(ii) and~(iii),
and that it does satisfy~(i) and~(ii).
To verify~(iii), we take the image under $\varphi$
of each of our generators of $|R\<|,$ break it into strings of
generators whose subscripts begin with $0$ alternating with
strings whose subscripts begin with $1,$
and look at the former as images of elements of $P'_2$
under $q^{(P'_2)\ba}_{2,\<0},$ and the latter as images of elements
of the same semigroup under $q^{(P'_2)\ba}_{2,1}.$
Abbreviating $q^{(P'_2)\ba}_{2,\<i}(u)$ to $u^{(i)}$ to avoid
long expressions (and confusion with our elements $q_i),$ we find that
\begin{quote}
$\begin{array}{ccc}
\beta^{P'_3}(x_{000})\ =\ x_{00}^{(0)}, &&
\beta^{P'_3}(x_{111})\ =\ x_{11}^{(1)},\\[.5em]
\beta^{P'_3}(p'_i)\ =\ p_i^{(0)}
(x_{00}^{(1)}\,x_{11}^{(0)})^i\,p_i^{(1)}, &&
\beta^{P'_3}(q'_i)\ =\ q_i^{(1)}
(x_{11}^{(0)}\,x_{00}^{(1)})^i\,q_i^{(0)}.
\end{array}$
\end{quote}
Clearly the terms from $(P'_2)\ba$ to which $^{(0)}$ and $^{(1)}$
are above applied are members of~$|R\<|.$
\end{proof}

This makes $\beta^{P'_3}\circsm\varphi$ a co-operation
$\beta^R:|R\<|\to|R\<|\cP|R\<|;$ explicitly
\begin{equation}\begin{minipage}[c]{34pc}\label{x.beta^R}
$\begin{array}{ccc}
\beta^R(x_{00})\ =\ x_{00}^{(0)}, &&
\beta^R(x_{11})\ =\ x_{11}^{(1)},\\[.5em]
\beta^R(p_i)\ =\ p_i^{(0)}(x_{00}^{(1)}\,x_{11}^{(0)})^i\,p_i^{(1)}, &&
\beta^R(q_i)\ =\ q_i^{(1)}(x_{11}^{(0)}\,x_{00}^{(1)})^i\,q_i^{(0)},
\end{array}$
\end{minipage}\end{equation}
making $R=(|R\<|,\beta^R)$ a coalgebra.
It is straightforward to verify that $\beta^R$ is coassociative.
For example, the reader is encouraged to check that
the derived co-operations corresponding to the two sides
of the associative identity, when applied to $p_i,$ both yield
\begin{equation}\begin{minipage}[c]{34pc}\label{x.p_i_lmr}
$p_i^{(\lambda)}\,(x_{00}^{(\mu)}\,x_{11}^{(\lambda)})^i\,p_i^{(\mu)}\,
(x_{00}^{(\rho)}\,x_{11}^{(\mu)})^i\,p_i^{(\rho)}.$
\end{minipage}\end{equation}

Thus, $R$ belongs to $\coalg{\Se}{\Se}.$
The proof that it is the final object of that category
will now be quick.

Let $S$ be any object of $\coalg{\Se}{\Se},$ and let the components of
the unique morphism from $S$ to the precoalgebra $P$ be $f_k:S\to P_k$
$(k\in\omega).$
Since these maps respect (pseudo-)co-operations, we have,
in particular,
\begin{equation}\begin{minipage}[c]{34pc}\label{x.f_3}
$(f_3)_\beta\circsm\beta^S\ =\ \beta^{P_3}\circsm(f_3)\ba\<.$
\end{minipage}\end{equation}
Now looking at the commuting square formed from the
(pseudo)coprojections of $S$ and of $P_3,$ and
invoking~(\ref{x.P_j+1*a}),
we see that the term $(f_3)_\beta$ in~(\ref{x.f_3}) can be
written as a copower, $\coprod_{\ari(\beta)}\,(f_2)\ba.$
On the other hand, in the right-hand side of~(\ref{x.f_3}),
we can insert $\varphi\circsm(p_{2,\<3})\ba$ between the two factors,
since $(p_{2,\<3})\ba$ carries $f_3$ of an element of $|S\<|$ to $f_2$
of that element, and $\varphi,$ by construction, will carry the
resulting element back to its only preimage in $(P'_3)\ba$ that
can possibly be in the image of $f_3.$
When we make this insertion, the combination $\beta^{P_3}\circsm\varphi$
gives, by definition, $\beta^R,$ and~(\ref{x.f_3}) becomes
\begin{equation}\begin{minipage}[c]{34pc}\label{x.S>P}
$\coprod_{\ari(\beta)}\,(f_2)\ba\circsm\beta^S\ =
\ \beta^R\circsm(f_2)\ba\<,$
\end{minipage}\end{equation}
showing that $(f_2)\ba$ is a morphism of coalgebras $S\to R.$

To show that this is the only morphism of coalgebras $S\to R,$
note first that the inclusion of the subsemigroup
$|R\<|$ in $(P_2)\ba$ induces a morphism $g$ from the coalgebra
$R$ to the final $\!3\!$-indexed precobinar $(P'_k\<)_{k\in 3},$
and that this morphism separates elements of $|R\<|.$
Hence if we had two distinct
morphisms $S\to R,$ then composing these with $g,$
we would get distinct morphisms $S\to(P'_k\<)_{k\in 3},$
contradicting the universal property
of that $\!3\!$-indexed precoalgebra
(Proposition~\ref{P.P_is_final}).

Thus

\begin{theorem}\label{T.finalcosemi}
The cosemigroup $R$ with underlying semigroup~\textup{(\ref{x.genrel})}
and co-opera\-tion~\textup{(\ref{x.beta^R})}
is the final object of $\coalg{\Se}{\Se}.$\endproof
\end{theorem}

What does the functor represented by this coalgebra $R,$ i.e.,
the initial object $E$ of $\Rep{\Se}{\Se},$ look like?
Writing $a,$ $b,$ $c_i,$ $d_i$ for the images of $x_{00},$ $x_{11},$
$p_i$ and $q_i$ respectively under a semigroup homomorphism on $|R\<|,$
it can be described as taking a semigroup
$A$ to the semigroup with underlying set
\begin{equation}\begin{minipage}[c]{34pc}\label{x.efab}
$|E(A)|\ =\ \{\,(a;\,b;\,c_0,c_1,\dots;\,d_0,d_1,\dots\,)\in
|A\<|^{1+1+\omega+\omega}\ \mid\\[.17em]
\hspace*{7em}a^2=a,\ \ b^2=b,\ \ a\<c_i=c_i=c_i\<b,\ \ b\<d_i=d_i=d_i\<a
\ \ (i\in\omega)\}$
\end{minipage}\end{equation}
(cf.~(\ref{x.genrel})), and operation
\begin{equation}\begin{minipage}[c]{34pc}\label{x.mult}
$(a;\,b;\,\dots,\,c_i,\,\dots;\ \dots,\,d_i,\,\dots\,)\,%
(a';\,b';\,\dots,c'_i,\,\dots;\,\dots,d'_i,\,\dots\,)\\[.25em]
\hspace*{4em}=\ (a;\ b';\ \dots,\ c_i\,(a'\<b)^i\,c'_i,\ \dots;
\ \dots,\ d'_i\,(b\<a')^i\,d_i,\ \dots\,)$
\end{minipage}\end{equation}
(cf.~(\ref{x.beta^R})).
The formula for the $c_i$ and $d_i$ components of~(\ref{x.mult})
is an instance of what semigroup theorists \cite{C+P} call a
{\em Rees matrix semigroup} construction.

As an example of the universal property of $R,$ consider,
in $\coalg{\Se}{\Se},$ the object $S$ representing the identity functor.
This is represented by the free semigroup on one generator $y,$ with
comultiplication given by
\begin{equation}\begin{minipage}[c]{34pc}\label{x.g0g1}
$\beta^S(y)\ =\ y^{(0)}\<y^{(1)}.$
\end{minipage}\end{equation}

When we map this into $P,$
$y$ is necessarily mapped to $x$ in $(P_0)\ba,$ so
by~(\ref{x.g0g1}) the image of $\beta^S(y)$ in $(P_1)_\beta$
is $x^{(0)}\<x^{(1)};$ so the image of
$y$ in $(P_1)\ba$ is the corresponding element, $x_0\,x_1.$
(Cf.\ the discussion of how such maps are constructed following
Proposition~\ref{P.P_is_final}.)
Hence, again using~(\ref{x.g0g1}), the image of $\beta^S(y),$
in $(P_2)_\beta$ is $(x_0\,x_1)^{(0)}\<(x_0\,x_1)^{(1)};$
so the image of $y$ in $(P_2)\ba$ is $x_{00}\,x_{01}\,x_{10}\,x_{11}.$
Note that this is
$(x_{00}\,x_{01})^1(x_{10}\,x_{01})^0\,(x_{10}\,x_{11})^1=p_0\in|R\<|.$
Thus, the unique map from the cosemigroup $S$ representing the identity
functor to the final
cosemigroup $R$ is the semigroup homomorphism taking the free
generator $y$ of $|S\<|$ to $p_0\in|R\<|\<.$

Since $p_0$ corresponds to the coordinate $c_0$ of~(\ref{x.efab}),
the corresponding morphism of representable functors
from $E$ to the identity takes the infinite tuples of~(\ref{x.efab})
and ignores all coordinates except $c_0,$ which it uses as its value.
This mapping is not in general onto: its image is the subsemigroup of
$A$ consisting of
those elements that are fixed by multiplication by at least one
idempotent on the left and by at least one idempotent on the right;
in particular, it is empty if $A$ has no idempotents.
This nonsurjectivity is inevitable:
there {\em are} representable semigroup-valued functors $F$ which
give the empty semigroup when applied
to semigroups having no idempotents, so to
be an initial representable functor, and thus have homomorphisms
to such $F,$ $E$ must also give the empty semigroup in such cases.
On the other hand, since representable functors between varieties
respect final objects, every representable functor $F:\Se\to\Se$
will, when applied
to the $\!1\!$-element semigroup, give a $\!1\!$-element semigroup;
hence $F$ must also give nonempty semigroups on all semigroups into
which the $\!1\!$-element semigroup can be mapped, i.e., all
semigroups that do have idempotent elements.

As easy variants of this
example, the unique morphism from $E$ to the direct product
of two copies of the identity functor takes each element as
in~(\ref{x.efab}) to the pair $(c_0,c_0);$ the unique morphism to
the opposite-semigroup functor $A\mapsto A^\mathrm{op}$ projects
to the coordinate $d_0,$ rather than~$c_0.$

For a different sort of example, let $F$ be the functor
taking every semigroup $A$ to its underlying set, with
the {\em left zero} multiplication $s\cdot t=s.$
Then the unique morphism $E\to F$ will carry each
infinite tuple as in~(\ref{x.efab}) to its first coordinate $a.$
Again this will in general be nonsurjective,
though in this case, the image of $E$ in $F$ will be a
representable subfunctor, taking each semigroup $A$ to
the semigroup consisting of the set of idempotent elements of $|A\<|,$
with the left zero multiplication.

Though each of the above morphisms from the initial
representable functor throws away {\em almost} all the information
contained in~(\ref{x.efab}), a morphism from $E$ to a nontrivial
representable functor $F$ cannot throw away all such information.
If it did, it would have to send all of $E(A)$ to a single
distinguished element of $F(A);$ but a nontrivial representable
functor on a variety without zeroary operations does not admit
such a choice of distinguished element.
From the coalgebra point of view, the corresponding observation is that
the image in $R$ of the coalgebra $S$ representing $F$ cannot be empty
unless $S$ is the empty coalgebra, i.e.,
unless $F$ is the trivial functor.

These examples have not involved the ``exotic'' coordinates
of~(\ref{x.mult}), the $c_i$ and $d_i$ with $i>0.$
There are no representable functors that I was previously aware of
whose morphisms from the initial representable functor
involve those coordinates; but one can easily design such functors.
Our initial functor $E$ itself is one, of course.
One can also cut it down, throwing away all the relations
in~(\ref{x.efab}), and all but three of the coordinates,
getting the construction sending $A$ to the set of all
$\!3\!$-tuples $(a,b,c)$ of elements of $|A\<|,$ with the
multiplication $(a,\,b,\,c)(a',\,b',\,c')=(a,\,b',\,c\,(a'b)^i c'),$
for arbitrary fixed~$i.$
It is easy to verify that this is associative, and that the unique
morphism from $E$ to this functor uses the coordinates $a,$ $b,$
and $c_i.$

More generally, if one takes any semigroup word $w$ in $m+n$
variables, one can define a representable functor $F$ taking every
semigroup $A$ to the set of all $\!m+n+1\!$-tuples
\begin{equation}\begin{minipage}[c]{34pc}\label{x.m+n+1}
$(a_1,\ \dots,\ a_m;\ b_1,\ \dots,\ b_n;\ c)$
\end{minipage}\end{equation}
of elements of $A,$ with the operation
\begin{equation}\begin{minipage}[c]{34pc}\label{x.w-mult}
$(a_1,\dots,a_m;\ b_1,\dots,b_n;\ c)
\ (a'_1,\dots,a'_m;\ b'_1,\dots,b'_n;\ c')=\\[.17em]
\hspace*{1em}
=\ (a_1,\,\dots,\,a_m;\ b'_1,\,\dots,\,b'_n;\ c
\ w(a'_1,\,\dots,\,a'_m,\,b_1,\,\dots,\,b_n\,)\ c'),$
\end{minipage}\end{equation}
which it is easy to verify is associative.
To describe the map from $E$ to this functor, break the
word $w$ into alternating blocks of occurrences of
the ``$\!a\!$'' variables (the first $m)$ and
the ``$\!b\!$'' variables (the last $n),$ and count these blocks,
starting from the first block of ``$\!a\!$'' variables
(ignoring any ``$\!b\!$'' variables  that may precede them), and
ending with the last block of ``$\!b\!$'' variables
(ignoring any ``$\!a\!$'' variables that may follow that block).
Say these constitute $i$ blocks of ``$\!a\!$'' variables alternating
with $i$ block ``$\!b\!$'' variables, for some $i\in\omega.$
(This $i$ may be $0,$ namely, if $w$ consists only of a
word in the ``$\!b\!$'' variables followed by a word in the ``$\!a\!$''
variables, with one of those words possibly empty.)
Then the unique morphism $E\to F$ assigns to all $a_r$ the value
$a,$ to all $b_r$ the value $b,$ and to $c$ the value $c_i.$

We remark that~(\ref{x.w-mult}) looks more natural if we rewrite
our $\!m\<{+}\<n\<{+}1\!$-tuples so as to put
the coordinate~$c$ in the middle,
and abbreviate $a_1,\dots,a_m$ to $a,$ $b_1,\dots,b_n$ to $b,$ and
$w(a'_1,\,\dots,\,a'_m,\nolinebreak[3]\,b_1,\,\dots,\,b_n)$ to $b*a'.$
Then~(\ref{x.w-mult}) becomes
\begin{equation}\begin{minipage}[c]{34pc}\label{x.w-mult2}
$(a,\ c,\ b)\ (a',\ c',\ b')\ =\ (a,\ c\ (b*a')\ c',\ b').$
\end{minipage}\end{equation}
In particular, the fact that the multiplication keeps $a$ and
$b'$ as coordinates, while the expression between $c$ and $c'$
involves, not these, but $b$ and $a',$ looks reasonable in this
notation, and the calculation by which one verifies
associativity becomes ``trivial''.

\vspace{.5em}
What do we learn about general representable functors between semigroups
from our description of the initial such functor $E$?
I do not have a good answer.
One interesting consequence is that for any representable functor
$F,$ the unique morphism $E\to F$ yields elements of the
semigroups $F(A)$ on which any two morphisms of representable
functors from $F$ to a common functor must agree.
So, for instance, any two morphisms from the
underlying-set-with-left-zero-operation functor
into a common representable functor $G$ must agree on the set of
idempotent elements (i.e., elements idempotent with respect
to the operation of $A.$
Under the operation of $F(A),$ all elements are idempotent.)
That this is not true of other elements can be seen by comparing
the identity endomorphism of $F$ with the endomorphism taking
each element $s$ to the element $s^2$ (where the ``square'' is
interpreted in terms of the original multiplication of $A).$
These endomorphisms agree on idempotent elements only.

The same argument tells us
that any two morphisms from the {\em identity} functor to
a common representable functor must agree on all elements that are
fixed on each side by multiplication by at least one idempotent.
But in this case, they must agree everywhere; for
the unique cosemigroup morphism from the coalgebra
representing the identity functor into $R,$ described at the end
of the paragraph following~(\ref{x.g0g1}), is one-to-one, hence it is
a monomorphism, so the corresponding
map of representable functors is an epimorphism.
(This is not true of identity functors of all varieties.
For instance, in $\mathbf{Ab},$ multiplication by each
integer $n$ gives an endomorphism of the identity functor, so that
functor has many distinct morphisms to itself, so the map from the
initial functor to it is not epic.)

Another possible use of our description of the final object
of $\coalg{\Se}{\Se}$ lies in the fact that elements of a general
cosemigroup in $\Se$ may be classified according to their images under
the unique morphism to $R,$ giving a start toward the task of
describing all such objects (cf.\ \cite[Problem~21.7, p.94]{coalg}).
An easy case, which we leave to the reader to verify, is

\begin{proposition}\label{P.x00.coalg}
Consider objects $S$ of $\coalg{\Se}{\Se}$ such that the
unique homomorphism $S\to R$ carries all elements of $||S\<||$
to $x_{00}\in||R\<||.$

Each such coalgebra $S$ may be obtained by taking a semigroup $B$ and
an idempotent endomorphism $\epsilon: B\to B,$ and letting
\begin{equation}\begin{minipage}[c]{34pc}\label{x.Sbase}
$|S\<|\ =\ B,$
\end{minipage}\end{equation}
\begin{equation}\begin{minipage}[c]{34pc}\label{x.betaS}
$\beta^S\ =\ q^B_{2,\<0}\ \epsilon\,;$\quad i.e.,\quad
$\beta^S(b)\ =\ \epsilon(b)^{(0)}$ $(b\in|B\<|).$
\end{minipage}\end{equation}

The corresponding representable functor carries every semigroup
$A$ to the set of homomorphisms $h: B\to A,$ under
the semigroup operation
\begin{equation}\begin{minipage}[c]{34pc}\label{x.h.h'}
$h\cdot h'\ =\ h\,\epsilon\<.$
\end{minipage}\end{equation}
\endproof
\end{proposition}

It would be interesting to know the answer to

\begin{question}\label{Q.gens}
Let $S$ be an object of $\coalg{\Se}{\Se},$ and $f$ the unique
morphism in that category from $S$ to the final object $R.$

Must the semigroup $|S|$ be generated by those of
its elements that are sent by $f$ to members of the generating set
$\{x_{00},\ x_{11},\ p_i,\ q_i\ (i\in\omega)\}$ of $|R\<|$?

If so, must $|S|$ in fact have a presentation such that the common
value in $|R\<|$ of the two sides of every relation
also belongs to that generating set?
\end{question}

Let me conclude this section by mentioning that my original
investigation of the structure of the final object $R$ of
$\coalg{\Se}{\Se},$ though it led to the same result,
Theorem~\ref{T.finalcosemi} above, took a somewhat different route.
I considered an arbitrary object $S$ of $\coalg{\Se}{\Se}$
and an arbitrary element $s\in||S||,$ and looked at the images of $s$
under the binary, ternary, etc., co-operations of $S$ corresponding to
multiplication of two, three, etc., elements of a semigroup.
(Thus, I used coassociativity from the start, rather than
thinking of $S$ as a cobinar, and then adding the assumption
that it cosatisfied the associative law.)
From the image of $s$ under the binary co-operation, I extracted a
``$\!\lambda,\rho\!$-signature'', a word in the alphabet
$\{\lambda,\rho\},$ which from our present point of view
encodes the image of $s$ in $\{x\}^{(\lambda)}\cP\{x\}^{(\rho)}.$
Its images in the $\!3\!$- and $\!4\!$-fold copowers similarly
gave a ``$\!\lambda,\mu,\rho\!$-signature'' and a
``$\!\lambda,\mu_1,\mu_2,\rho\!$-signature''.
Examining the relations between these signatures, I determined
which strings could occur, and verified that the semigroup
of $\!\lambda,\mu,\rho\!$-signatures that occurred could be given
the structure of our desired final cosemigroup.

That development, used in the first drafts of this note,
was about as lengthy as the present one, though the formulas
were a bit shorter, as suggested by comparison of~(\ref{x.p_i_alt})
and~(\ref{x.p_i_sig}).
I eventually switched to the present development for the sake of
coherence with the methods of other sections of this note.
But the alternative approach might be kept in mind by anyone
wishing to look further at coalgebras with coassociative co-operations.

\section{Examples with $>1$ derived zeroary
operations.}\label{S.>1zeroary}

We saw in Theorem~\ref{T.7/9} (last paragraph) that the only
situations where the initial object of $\Rep{\C}{\D}$ can have
nontrivial nonempty algebras among its values
are when $\C$ and $\D$ either both have no zeroary operations,
or both have more than one derived zeroary operation.
So far, we have looked at cases of the former sort.

A very elementary example of a variety with more than one derived
zeroary operation is the one
determined by a single primitive zeroary operation $\alpha_0,$
a single primitive unary operation $\alpha_1,$
and no identities; i.e., the variety of sets with a distinguished
element, and an endomap not assumed to fix that element.
(The derived zeroary operations are $\alpha_0,
\ \alpha_1(\alpha_0),\ \dots,\ \alpha_1^i(\alpha_0),\ \dots\ .)$

Taking any variety $\C,$ and the above variety as $\D,$
let $P$ be the final $\!\D\!$-precoalgebra in $\C$ having
for $P_0$ the trivial pseudocoalgebra.
By Proposition~\ref{P.dropbase}(iii),
$(P_1)\ba$ can be identified with the direct product
of the $\!0\!$-fold and $\!1\!$-fold copowers of $(P_0)\ba\<.$
The $\!1\!$-fold copower is just $(P_0)\ba,$ the trivial algebra;
hence that direct product can be identified with the $\!0\!$-fold
copower of $(P_0)\ba;$ i.e., the initial object of $\C.$
Let us denote this by $I_\C;$ we know that
it will be nonempty and nontrivial precisely if $\C$ has more
than one derived zeroary operation.

We find that passage to each subsequent level of $P$
brings in one more copy of $I_\C,$ so that $(P_k)\ba=I_\C^k.$
Taking the limit, we get
$(P_\omega)\ba=I_\C^\omega\<,$ where the pseudo-co-operation
$\alpha_0^{P_\omega}$ projects $I_\C^\omega$
to its component indexed by $0,$
while $\alpha_1^{P_\omega}$ is the ``left shift'' map.

Propositions~\ref{P.compar_iso} and~\ref{P.red} tell us that
if $\C$ is a finitary variety satisfying~\textup{(\ref{x.u=})},
then $P_\omega$ will be a coalgebra,
the final object of $\coalg{\C}{\D}.$
But in fact, those hypotheses on $\C$ are not needed here.
Their purpose was to guarantee that $\!\ari(\alpha)\!$-fold
copowers commuted with $\!\omega\!$-indexed inverse limits
for all $\alpha\in\Omega_\D,$ but this is automatic
when all elements of $\Omega_\D$ are zeroary or unary.
(The $\!1\!$-fold copower functor is the identity and
commutes with everything; the $\!0\!$-fold copower is the
constant functor giving the initial algebra, and constant functors
commute with limits over connected categories.)

Indeed, either by following the above precoalgebra approach,
or by directly checking the universal property, it is not hard to
verify the following general statement.

\begin{lemma}\label{L.01ary}
In a category $\fb{X}$ with an initial object, let such an
object be denoted~$I_\fb{X}.$

Suppose $\fb{A}$ is a category having small products and
coproducts, and $\D$ a variety with only zeroary and unary
primitive operations \textup{(}with or without identities\textup{)}.
Then $\coalg{\fb{A}}{\D}$ has a final object $R,$ with
underlying $\!\fb{A}\!$-object $|R\<|=I_\fb{A}^{|I_\D|},$ where the
\textup{(}primitive or arbitrary\textup{)}
zeroary $\!\D\!$-co-operations of $R$ are given by the projections to
the components indexed by the values in $|I_\D|$ of the corresponding
zeroary operations of $I_\D,$ and where the
\textup{(}primitive or arbitrary\textup{)}
unary $\!\D\!$-co-operations of $R$ are determined
\textup{(}contravariantly\textup{)} by the action of the corresponding
unary operations of $I_\D$ on the index-set $|I_\D|.$\endproof
\end{lemma}

For concreteness, let us return to the case where $\D$
is the variety defined by one zeroary operation,
one unary operation, and no identities.
Let $\C$ be an arbitrary variety with more than one derived zeroary
operation, and let $R$ be the final $\!\D\!$-coalgebra of $\C,$
which by the above observations
has underlying $\!\C\!$-algebra $I_\C^\omega.$
If $I_\C$ is finite or countable, then $R$ has the
cardinality of the continuum; and for many choices of $\C$ -- e.g.,
$\D$ itself, or the variety of sets with two distinguished elements,
or the variety of vector-spaces with a distinguished vector,
over a countable or finite field -- we find that the corresponding
representable functor $\C(R,-)$ takes all nontrivial finite or
countable objects of $\C$ to
objects of $\D$ having cardinality~$2^{\aleph_0^{\aleph_0}}.$

As a curious exception, suppose $\C$ is the
category of abelian groups with one distinguished element, so that
$I_\C$ is $\mathbb{Z}$ with distinguished element $1.$
Then $|R\<|$ is the group $\mathbb{Z}^\omega$
of all sequences of integers, with the
constant sequence $(1,1,\dots)$ as distinguished element.
It is known that the only group homomorphisms
$\mathbb{Z}^\omega\to\mathbb{Z}$ are the finite linear combinations
of the projection maps \cite[Proposition~94.1]{Fuchs}.
Hence, if we apply our functor $\C(R,-)$ to $\mathbb{Z},$ with any
$n\in|\mathbb{Z}|$ made the distinguished element, we get a
countable $\!\D\!$-object, the set of those
finite linear combinations of projection maps
$\mathbb{Z}^\omega\to\mathbb{Z}$ whose coefficients sum to $n$
(with distinguished element given by $n$ times the projection onto the
$\!0\!$-component, and endomap given by the {\em right} shift operator
on those strings of coefficients).
On the other hand, if we apply $\C(R,-)$ to
$\mathbb{Z}/p\mathbb{Z}$ or $\mathbb{Q}$
(with any distinguished element), we find that the resulting object
again has cardinality $2^{\aleph_0^{\aleph_0}}.$

\vspace{.5em}
A very different class of cases where we can get a nice handle on
the final object of a category $\coalg{\C}{\D}$ is when $\D$ is an
arbitrary variety, and $\C$ the variety $\fb{Boole}$ of Boolean rings
(commutative associative unital
rings satisfying the identities $2=0$ and $x^2=x).$
The variety $\fb{Boole}$ is dual to the category $\fb{Stone}$ of
Stone spaces, i.e., totally disconnected compact Hausdorff spaces,
via the functor taking every Boolean ring $A$ to its prime spectrum,
equivalently, to the set $\fb{Boole}(A,2)$ of its homomorphisms into
the $\!2\!$-element Boolean ring, under the function topology;
the inverse functor takes each Stone space $X$ to
the set $\fb{Stone}(X,2)$ of continuous $\!\{0,1\}\!$-valued
functions on $X,$ made a Boolean ring under pointwise operations,
equivalently, to the set of clopen subsets of $X,$ made a
Boolean ring in the usual way \cite[Introduction and \S II.4]{PJ.Sto}.

Hence, {\em $\!\D\!$-coalgebras} in $\fb{Boole}$ correspond
to {\em $\!\D\!$-algebra} objects in $\fb{Stone}.$
These can be described as $\!\D\!$-algebras given with totally
disconnected compact Hausdorff topologies on their underlying sets,
with respect to which the $\!\D\!$-operations are continuous.
The representable functor $\fb{Boole}\to \D$
corresponding to such a topological $\!\D\!$-algebra $X$ can
be described as taking each Boolean ring $B$ to the
$\!\D\!$-algebra of continuous $\!X\!$-valued functions on the
Stone space of $B.$
(In particular, we can recover the underlying $\!\D\!$-algebra of
$X$ from the functor by evaluating the latter on the Boolean ring $2,$
corresponding to the $\!1\!$-point Stone space.)

Any finite set with the discrete topology is a Stone space, so any
finite $\!\D\!$-algebra gives a $\!\D\!$-algebra object of $\fb{Stone}.$
For some varieties $\D,$ such as those of associative (unital
or nonunital) rings, semigroups,
monoids, groups, and distributive lattices (and, automatically,
all subvarieties of these), it is known
\cite[VI.2.6-9]{PJ.Sto}, \cite{250B_stone} that
\begin{equation}\begin{minipage}[c]{34pc}\label{x.profin}
The topological $\!\D\!$-algebras with Stone topologies are precisely
the inverse limits of systems of finite $\!\D\!$-algebras, with the
topologies induced by the discrete topologies on those finite algebras.
\end{minipage}\end{equation}

When this holds, every such Stone algebra is, in particular,
the inverse limit of all its finite homomorphic images;
and given an arbitrary $\!\D\!$-algebra $A,$
the category of Stone $\!\D\!$-algebras furnished with homomorphisms
of $A$ into them will have as initial object the inverse limit
of all finite homomorphic images of $A.$

Now note that the category $(A\downarrow\D)$ of
(non-topologized) $\!\D\!$-algebras with
homomorphisms of $A$ into them can be regarded as a variety $\D',$
by taking any presentation $\langle X\mid Y\rangle_\D$ for $A$
in $\D,$ and letting $\D'$ have, in addition to the operations and
identities of $\D,$ an $\!X\!$-tuple of additional zeroary
operations, subject to the relations $Y.$
So the above inverse limit of finite homomorphic images of $A$
can be regarded as the initial Stone topological $\!\D'\!$-algebra.

For example, suppose that $\D=\Gp$ and
$\langle X\mid Y\rangle_\D$ is the infinite cyclic group (where
we can take $X$ a singleton and $Y$ empty), or that
$\D=\fb{Ring}^1$ and $\langle X\mid Y\rangle_\D$ is
the initial ring $(X$ and $Y$ both empty).
In these cases, the initial Stone $\!\D'\!$-algebra will be the
inverse limit of the finite groups or rings $\mathbb{Z}/n\mathbb{Z}.$
(This inverse limit, called the ``profinite completion
of $\mathbb{Z}\!$'', is the direct product over all primes $p$
of the groups or rings of $\!p\!$-adic integers.)

In this example, our Stone algebra, and hence the value of our initial
representable functor at the Boolean ring $2,$
merely has the cardinality of
the continuum, and the final coalgebra to which it corresponds is
merely countable, in contrast to some of our earlier examples.
This is an instance of a general phenomenon.
Before formulating it, let us note that if $\D$
satisfies~(\ref{x.profin}), then that property will be inherited
by $\D'=(A\downarrow\D)$ for any $\!\D\!$-algebra $A,$ and that
if $A$ is finitely generated, then if $\D$ is describable in terms
of finitely many primitive operations, all of finite arity, these
properties will be inherited by $\D'.$
Hence it suffices to prove the next result for such a variety $\D,$
understanding that it will then be applicable to varieties
$\D'=(A\downarrow\D)$ when $A$ is a finitely generated $\!\D\!$-algebra.

\begin{lemma}\label{L.countable}
Suppose $\D$ is a finitary variety with only finitely many primitive
operations, which satisfies\textup{~(\ref{x.profin})}.
Then the final object of $\coalg{\fb{Boole}}{\D}$ is at most countable;
equivalently, the initial Stone $\!\D\!$-algebra has separable
topology.
\end{lemma}

\begin{proof}
From the fact that $\D$ is finitary with finitely many primitive
operations, one sees that up to isomorphism, there are at
most countably many finite $\!\D\!$-algebras.
Hence the initial $\!\D\!$-algebra
has at most countably many finite homomorphic images,
so the initial Stone $\!\D\!$-algebra is the inverse limit
of an at most countable inverse system of finite algebras.
Every continuous $\!\{0,1\}\!$-valued function on that inverse
limit is induced by such a function on one of these algebras; so the
Boolean ring of all such functions is at most countable.
\end{proof}

Sometimes our final coalgebras are even smaller than the
above result requires.
If one takes for $\D$ the variety of groups, and
lets $\langle X\mid Y\rangle_\D$ be an infinite group with no finite
homomorphic images, such as
\begin{equation}\begin{minipage}[c]{34pc}\label{x.wxyz}
$\langle\,w,x,y,z\ \mid\ wxw^{-1}=x^2,\ xyx^{-1}=y^2,\ %
yzy^{-1}=z^2,\ zwz^{-1}=w^2\,\rangle,$
\end{minipage}\end{equation}
(or, going outside the context of Lemma~\ref{L.countable} to
a non-finitely-generated example, but still using the fact
that $\Gp$ satisfies~(\ref{x.profin}), the
additive group of rational numbers), then the initial
Stone $\!\D'\!$-algebra is trivial, making the final object of
$\coalg{\fb{Boole}}{\D'}$ the $\!2\!$-element (initial) Boolean ring.
In this case, the derived zeroary operations of $\D'$ cannot be
distinguished by representable functors on $\fb{Boole},$ and we are
effectively reduced to case~(iii) of Theorem~\ref{T.7/9}.

In the other direction, taking varieties $\D$ not
satisfying~(\ref{x.profin}), one again
gets examples where the final object
of $\coalg{\fb{Boole}}{\D}$ has continuum cardinality.
For instance, if we let $\D$ be, as at the beginning of this section,
the variety of sets with a single zeroary and a single binary
operation, then by the discussion there, the final object of
$\coalg{\fb{Boole}}{\D}$ will be $2^\omega,$ the countable direct
power of the initial object $2=\{0,1\}$ of $\fb{Boole}.$
The corresponding Stone $\!\D\!$-algebra can be described as the
Stone-\v{C}ech compactification of $\omega,$ with the element
$0\in\omega$ as the value of the zeroary operation, and the
endomap induced by $n\mapsto n+1$ as the unary operation.
Some more familiar varieties not
satisfying~(\ref{x.profin}) are the variety
of Lie algebras over a finite field (this may be deduced from
\cite[Example~25.49, p.126]{coalg}) and the variety
of all lattices (\cite[second example on p.10]{250B_stone}).

Incidentally, one's first impression on looking at the statement of
Lemma~\ref{L.countable} might be that it should surely be possible
to weaken the hypothesis ``finitely many operations''
to ``at most countably many''; but this is not so.
If $\D$ is the variety of sets with an $\!\omega\!$-tuple of
zeroary operations, then~(\ref{x.profin}) is inherited from
the variety $\fb{Set},$ but by Lemma~\ref{L.01ary}, the final
object of $\coalg{\fb{Boole}}{\D}$ again has for underlying algebra
the uncountable Boolean algebra $2^\omega.$

Infinite compact Hausdorff spaces are ``typically''
uncountable, but not always.
Can the initial Stone object of a variety $\D$ be countably infinite?
I found this a hard one to answer; but here is an example.
Let $\D$ be the variety of lattices (or upper semilattices,
or lower semilattices) with one additional zeroary operation
$\alpha_0,$ and one additional unary operation $\alpha_1,$ subject
to the identities saying that $\alpha_1$ is {\em increasing,}
i.e., $\alpha_1(x)\wedge x=x$ and/or $\alpha_1(x)\vee x=\alpha_1(x)$
(depending on which operations are assumed present).
I leave the determination of the initial Stone topological
object of $\D$ to the reader who would like a challenging exercise.
(Note that one cannot assume that $\D$ satisfies~(\ref{x.profin}).)

If $\D$ is a variety whose initial object is finite, then even
without~(\ref{x.profin}) one can see that this finite object,
regarded as a finite Stone space, will be the initial Stone object
of $\D,$ and so will determine the initial representable functor
$\fb{Boole}\to\D.$
We find, in particular, that the initial representable
functor $\fb{Boole}\to\fb{Boole}$ is the identity.

\vspace{.5em}
There are other dualities between varieties of algebras and
structures with compact topologies; cf.\ \cite{RA+Kap}, \cite{DC+BD}.
It would be worth seeing whether these also lead to useful
results about coalgebras, equivalently, representable functors.

It would also be worth examining whether some of
the properties of varieties
$\D$ that allow one to prove~(\ref{x.profin}) might also allow
one to prove results about $\!\D\!$-coalgebras in general varieties.
In both~\S\ref{S.BiSe} and~\S\ref{S.SeSe} above, the
associative identity trimmed uncountable cobinars $P'_\omega$
down to countable cosemigroups $R=P_\omega.$
Is this a case of a more general phenomenon?
Cf.\ also \cite[\S32, first paragraph]{coalg}, where a
statement formally resembling~(\ref{x.profin})
is translated into a dual statement about ``coalgebras'',
though in one of the other senses of that word.

\vspace{.5em}
Yet another example of an initial representable functor between
varieties with more than one derived zeroary operation is the one
that motivated this investigation, as noted in \S\ref{S.motivate}:
the initial object of $\Rep{\fb{Ring}^1}{\,\fb{Ring}^1},$
given by $S\mapsto S\times S^{\mathrm{op}}$
\cite[Corollary~25.22]{coalg}.
(That result is generalized a bit in \cite[Exercise~31.12(vi)]{coalg}.)

\section{Universal constructions in $\coalg{\Gp}{\Gp}$ and\\
$\coalg{\fb{Monoid}}{\fb{Monoid}}.$}\label{S.Group}

Since the varieties of groups and of monoids have unique derived zeroary
operations (giving the identity element), Theorem~\ref{T.7/9}
tells us that the initial representable functor from either of
them to any variety, or from any variety to either of them, is trivial;
equivalently, that the corresponding final coalgebras have
initial algebras for underlying algebra.
But other limits of coalgebras
involving these varieties need not be trivial.
Let us briefly examine {\em products} in $\coalg{\Gp}{\,\Gp}.$

Kan \cite{DK.Mon} determined the structure of all comonoid
objects in $\Gp.$
These all extend uniquely to cogroup structures, so we shall describe
his result as determining the cogroups in $\Gp.$
The functors these
represent are simply the direct powers of the identity functor.
If we take an arbitrary cogroup $R$ in $\Gp,$ the precise
statement is that its underlying group $|R\<|$ is
free on the set consisting of the nonidentity elements
$x$ satisfying $\beta^R(x)=x^{(0)} x^{(1)}$ (where
$\beta\in\Omega_\Gp$ again denotes the primitive binary operation).
Now a morphism $R\to S$ of cogroups must take elements of $R$
satisfying $\beta^R(x)=x^{(0)} x^{(1)}$ to elements of $S$ with the same
property; so it must take each member of the canonical free generating
set for $|R\<|$ either to a member of the corresponding
generating set for $|S|,$ {\em or} to the identity element, $e.$
If we write $\mathrm{Id}_\Gp$ for the identity functor of
$\Gp,$ and $\mathrm{Id}_\Gp^n$ for the $\!n\!$-fold
direct power of this functor (the functor taking each group $G$
to $G^n$ -- not, of course, the $\!n\!$-fold composite
of $\mathrm{Id}_\Gp),$ then this is equivalent to saying that a morphism
$\mathrm{Id}_\Gp^m\to \mathrm{Id}_\Gp^n$ is determined
by specifying, as the value to be assigned at each of the $n$
coordinates of the codomain groups, either a specified
one of the $m$ coordinates of the domain groups, or the value $e.$

So, though one could phrase Kan's
result as saying that the isomorphism classes of cogroups
in $\Gp$ correspond to the isomorphism classes of sets, the
category $\coalg{\Gp}{\,\Gp}$
is actually equivalent, not to $\fb{Set},$ but
to the category $\fb{Set}^\mathrm{pt}$ of pointed sets,
by the functor taking each coalgebra $R$ to the set
$\{x\mid\beta^R(x)=x^{(0)} x^{(1)}\},$ with $e$ as basepoint.

The product of two pointed sets $(X,e_X)$ and $(Y,e_Y)$ is
$(X\times Y,\ (e_X,e_Y)).$
Thus, the product of the cogroups in $\Gp$ represented by
the free groups on $m$ and on $n$ generators will be free on
$(m+1)(n+1)-1$ generators (with $+1$ to count the basepoint $e$
in the description of the relevant pointed sets, and $-1$ to
discount the basepoint in the product set).
In terms of representable functors, the coproduct of
$\mathrm{Id}_\Gp^m$ and $\mathrm{Id}_\Gp^n$ in
$\Rep{\Gp}{\,\Gp}$ is
therefore $\mathrm{Id}_\Gp^{mn+m+n}.$
To describe the universal maps
$\mathrm{Id}_\Gp^m\to \mathrm{Id}_\Gp^{mn+m+n}$ and
$\mathrm{Id}_\Gp^n\to \mathrm{Id}_\Gp^{mn+m+n},$
note that for {\em any} such pair of maps,
each coordinate of $\mathrm{Id}_\Gp^{mn+m+n}$ must select either
some coordinate of $\mathrm{Id}_\Gp^m$ or the identity, and
either some coordinate of $\mathrm{Id}_\Gp^n$ or the identity.
The universal pair is such that
the $mn+m+n$ coordinates together cover all possible choices exactly
once, except for the simultaneous choice of the identity for
both coordinates.

It is interesting to observe that for $S$ and $S'$ objects of
$\coalg{\Gp}{\,\Gp}$ and $S\times S'$ their product
in this category, the pair of projection maps $S\times S'\to S$ and
$S\times S'\to S'$ does not in general separate elements of
the underlying set $||S\times S'\<||$ of the product coalgebra.
For example, let us write the underlying free group of the coalgebra
representing $\mathrm{Id}_\Gp$ as $|S\<|=\langle x\rangle,$
and the underlying free group of the product coalgebra $S\times S$ as
$|S\times S\<|= \langle g_{x,\<x},\,g_{x,\<e},\,g_{e,\<x}\rangle,$ where
the subscripts indicate where the generators are sent
under the two projection maps.
Since this group is free, we have
$g_{x,\<e}\,g_{e,\<x}\neq g_{e,\<x}\,g_{x,\<e};$ but the two sides
of this inequality fall together under both projection maps to $|S\<|.$

\vspace{.5em}
We saw in Theorem~\ref{T.cofree} that for every object $A$
of a variety $\C,$ and every variety $\D,$ there exists an object
$R$ of $\coalg{\C}{\D}$ with a universal $\!\C\!$-algebra homomorphism
$|R\<|\to A;$ a coalgebra ``cofree'' on $A,$ corresponding to a
representable functor ``free'' on an arbitrary
representable set-valued functor.
For $\C=\D=\Gp$ and any group $A,$ it is not hard to show that
this universal $R$ corresponds, under the above description
of $\coalg{\Gp}{\,\Gp},$ to
the pointed set $(|A\<|,e_A),$ with the universal homomorphism taking
the generator of the free group $|R\<|$ corresponding to each
$a\in|A\<|$ to that element of $A.$
So, for instance, if $A$ is the free group on one generator
(representing the {\em underlying-set} functor), then $|R\<|$ will
be a free group on countably many generators, indexed by the nonzero
integers, and the corresponding representable functor
will be the direct power $\mathrm{Id_\Gp}^{|\mathbb{Z}|-\{0\}}$
of the identity functor.
The underlying-set functor is mapped
to $U_\Gp\circsm\mathrm{Id_\Gp}^{|\mathbb{Z}|-\{0\}}$ by sending
each element $b$ of the underlying set of a group $B$ to
$(b^n)_{n\in|\mathbb{Z}|-\{0\}}.$

The description we have given of the category of representable
functors from groups to groups goes over almost unchanged to
representable functors from groups to monoids, and from monoids
to groups, the only adjustment being that the identity
functor on groups is replaced in the first case by the forgetful
functor from groups to monoids, and in the second by the
group-of-invertible-elements functor from monoids to groups.
But the structure of representable functors from monoids to monoids,
worked out in \cite[\S20]{coalg} and
\cite[\S9.5]{245}, is more complicated.
(In \cite{coalg},
incidentally, we called monoids ``semigroups with neutral element''.)
Such functors $F$ correspond not to sets, but to certain sorts of
bipartite directed graphs, where vertices on one side correspond to
coordinates of $F(S)$ that are multiplied as in $S,$ those on the other
side to elements that use the opposite multiplication, and the edges of
the graph represent conditions specifying certain coordinates of one
sort as left or right inverses of certain coordinates of the other sort.
(Here the identity element sits in the intersection of the two
sets, violating the usual definition of bipartite graph.
It is not connected by edges to any other elements, and could be
excluded if we were only concerned with the structures of
individual functors and their representing coalgebras.
But its importance, as above, comes out in the
description of morphisms among these functors and coalgebras.)

For details, the reader can see \cite{coalg}
or \cite{245}; but let me sketch some elementary examples.
Let $R$ be the comonoid representing the functor that takes
every monoid $A$ to the monoid $A\times A^\mathrm{op};$
let $R'$ represent the functor that takes
every $A$ to the submonoid of $A\times A^\mathrm{op}$
consisting of pairs $(a,b)$ with $ab=e;$ and
let $R''$ represent the functor that takes
every $A$ to the still smaller submonoid
consisting of pairs $(a,b)$ with $ab=e=ba.$
Each of these functors has an inclusion morphism into the one that
precedes; hence our comonoids have maps $R\to R'\to R''.$
These are in fact surjective; on underlying monoids they are
\begin{equation}\begin{minipage}[c]{34pc}\label{x.xy}
$\langle x,y\rangle\ \to\ %
\langle x,y\mid xy=e\rangle\ \to\ %
\langle x,y\mid xy=e=yx\rangle,$
\end{minipage}\end{equation}
where the arrows carry generators to generators with the same
name, and where, in each case, the comultiplication is given by
\begin{equation}\begin{minipage}[c]{34pc}\label{x.xLyR}
$x\ \mapsto\ x^{(0)}\<x^{(1)},\qquad y\ \mapsto\ y^{(1)}\<y^{(0)}.$
\end{minipage}\end{equation}

Now by the results of \cite{coalg} or \cite{245},
every object $S$ of $\coalg{\fb{Monoid}}{\fb{Monoid}}$ has underlying
monoid generated by $\{z\in||S||\mid\beta^S(z)=z^{(0)}z^{(1)}$~or
$z^{(1)}z^{(0)}\},$ hence every morphism of such objects
is determined by its behavior on this set.
For each of the above objects, this set is $\{x,y,e\},$ and the
maps~(\ref{x.xy}) are clearly bijective on these sets.
(In the formalism of \cite{coalg} and \cite{245}, these
objects correspond to the bipartite graphs $[\,\cdot\ |\ \cdot\,],$
$[\,\cdot\hspace{.25em}|\hspace{-.6em}{\rightarrow}\cdot\,]$ and
$[\,\cdot\hspace{.33em}|\hspace{-.63em}{\leftrightarrow}\cdot\,]$
respectively.)
Hence these maps of coalgebras are both monomorphisms and epimorphisms,
but not isomorphisms.

Let us consider one more comonoid, $S,$ representing the
identity functor, with underlying monoid $\langle x\rangle$ free
on one generator (and bipartite graph $[\,\cdot\ |\ \ ]).$
Note that the inclusion $S\to R''$ is also both a monomorphism
and an epimorphism
(because it has these properties in $\fb{Monoid}),$ though in
contrast to the morphisms~(\ref{x.xy}), it is not surjective.
If we factor this map in the obvious
way as $S\to R\to R'',$ we see that a composite
$a\circsm b$ of morphisms of coalgebras can be both epic and monic,
and $a$ also be epic and monic, without $b$ being epic.

What about the cofree coalgebra on a monoid $A$?
One finds that it corresponds to the bipartite graph obtained by
taking two copies of $A,$ say
$\{a_0\mid a\in|A\<|\}\cup\linebreak[0]\{a_1\mid a\in|A\<|\},$
using one as the left side of our graph and the other as
the right side, collapsing $e_0$ and $e_1$ to a single identity
element, and drawing directed edges
from $a_0$ to $b_1$ and from $a_1$ to $b_0$ whenever $ab=e.$

\section{A quick look
of $\coalg{\fb{Ring}^1}{\fb{Ab}}.$}\label{S.Ring>Ab}

Let $\fb{Ring}^1$ denote the variety of associative rings with $1,$ and
$\fb{Ab}$ the variety of abelian groups, which we will write additively.
In \cite[Theorem~13.15]{coalg} it is shown, inter alia, that
$\coalg{\fb{Ring}^1}{\fb{Ab}}$ is equivalent to $\fb{Ab}.$
Precisely, if we denote by $U$ the forgetful functor taking every
ring to its underlying abelian group, then
every representable functor $\fb{Ring}^1\to\fb{Ab}$
can be obtained by following $U$ with a representable functor
$\fb{Ab}(A,-):\fb{Ab}\to\fb{Ab},$ for some abelian group $A.$
(As is well-known, every abelian group is the representing
object of a unique representable functor
$\fb{Ab}\to\fb{Ab},$ the group structure being given
by elementwise operations on group homomorphisms.)
The object of $\coalg{\fb{Ring}^1}{\fb{Ab}}$ representing
this composite is the tensor ring
$\mathbb{Z}\langle A\<\rangle,$ i.e., the result of applying to
$A$ the left adjoint to $U,$ with a
co-$\!\fb{Ab}\!$-structure that the reader can easily write down; and
these constructions yield a functorial equivalence.

Limits in $\fb{Ab}$ are easy to describe, and yield descriptions
of the corresponding limits in $\coalg{\fb{Ring}^1}{\fb{Ab}}.$

Note that the tensor ring construction does not respect
one-one-ness of maps.
For instance, let $A$ be the
abelian group $\langle a,b\mid pb=0\rangle$ for any prime $p$
(the direct sum of $\mathbb{Z}$ and $\mathbb{Z}/p\mathbb{Z}),$ let $B$
be the subgroup of $A$ generated by $pa$ and $b,$ and let $f:B\to A$ be
the inclusion map.
Then in $\mathbb{Z}\langle B\<\rangle$
the element $(pa)b$ is easily shown
to be nonzero, but its image in $\mathbb{Z}\langle A\<\rangle$
can be written $p(ab)=a(pb)=a0=0.$

Consequently, statements about $\fb{Ab}$ that refer to one-one-ness
may not go over to $\coalg{\fb{Ring}^1}{\fb{Ab}}.$
For instance, in $\fb{Ab},$ every one-to-one map is an equalizer (e.g.,
of its canonical map onto its cokernel group, and the zero map
thereto); in particular, this is true of the inclusion $A\subseteq B$
of the preceding paragraph; hence the induced map
in the equivalent category $\coalg{\fb{Ring}^1}{\fb{Ab}}$
is also an equalizer; showing that equalizers in that category
need not be one-to-one on underlying sets.


\end{document}